\documentclass[11pt]{article}
\usepackage{amsmath}
\usepackage{amssymb,amsbsy,amsthm,dsfont}
\usepackage{graphicx}
\usepackage[dvipsnames]{xcolor}
\usepackage{caption}
\usepackage{subcaption}
\usepackage{enumerate}
\usepackage{enumitem}
\usepackage{cases}
\usepackage[margin=1in]{geometry}
\usepackage{hyperref}

\allowdisplaybreaks[1]

\numberwithin{equation}{section}

\DeclareMathOperator{\R}{\mathbb{R}} % Real numbers
\DeclareMathOperator{\N}{\mathbb{N}} % Natural numbers
\newcommand{\p}{\mathbb{P}} % Probability P
\newcommand{\E}{\mathbb{E}} % Expectation E
\newcommand{\eps}{\varepsilon} % Epsilon
\DeclareMathOperator{\Var}{Var} % Variance Var
\DeclareMathOperator{\Cov}{Cov} % Covariance Cov 
\newcommand{\TV}{\mathrm{TV}} % Total Variation TV 
\DeclareMathOperator{\poly}{poly} % Polynomial polys

\newcommand{\Beta}{\mathrm{Beta}} % Beta distribution
\newcommand{\Dir}{\mathrm{Dir}} % Dirichlet distribution Dir
\DeclareMathOperator{\sgn}{sgn} % Signum sgn

\newcommand{\wh}{\widehat}  % widehat wh
\newcommand{\wt}{\widetilde} % widetilde wt 

\newcommand{\PA}{\mathrm{PA}} % Preferential Attachment PA
\newcommand{\UA}{\mathrm{UA}} % Uniform Attachment UA

\newcommand{\cCG}{\mathcal{CG}} % Correlated caligraphic G, CG
\newcommand{\CPA}{\mathrm{CPA}} % Correlated PA CPA
\newcommand{\CUA}{\mathrm{CUA}} % Correlated UA CUA
\newcommand{\Range}{\mathrm{Range}} % Range Range

\newcommand{\tcorr}{t_{\star}} % correlation time 
\def\cA{{\mathcal A}}
\def\cB{{\mathcal B}}
\def\cC{{\mathcal C}}
\def\cD{{\mathcal D}}
\def\cE{{\mathcal E}}

\def\cG{{\mathcal G}}
\def\cH{{\mathcal H}}

\def\cN{{\mathcal N}}

\newtheorem{theorem}{Theorem}[section]
\newtheorem{lemma}[theorem]{Lemma}

\newtheorem{corollary}[theorem]{Corollary}

\newtheorem{definition}[theorem]{Definition}

\begin{document}

\title{Correlated randomly growing graphs}
\author{
	Mikl\'os Z.\ R\'acz
	\thanks{Princeton University; \texttt{mracz@princeton.edu}. Research supported in part by NSF grant DMS 1811724 and by a Princeton SEAS Innovation Award.} 
	\and
	Anirudh Sridhar
	\thanks{Princeton University; \texttt{anirudhs@princeton.edu}. Research supported in part by NSF grant DMS 1811724.}
}
\date{\today}

\maketitle

%%%%%%%%%%%%%%%%
%%% Abstract %%%
%%%%%%%%%%%%%%%%

\begin{abstract}
We introduce a new model of correlated randomly growing graphs and study the fundamental questions of detecting correlation and estimating aspects of the correlated structure. The model is simple and starts with any model of randomly growing graphs, such as uniform attachment (UA) or preferential attachment (PA). Given such a model, a pair of graphs $(G_1, G_2)$ is grown in two stages: until time $t_{\star}$ they are grown together (i.e., $G_1 = G_2$), after which they grow independently according to the underlying growth model. 

We show that whenever the seed graph has an influence in the underlying graph growth model---this has been shown for PA and UA trees and is conjectured to hold broadly---then correlation can be detected in this model, even if the graphs are grown together for just a \emph{single time step}. We also give a general sufficient condition (which holds for PA and UA trees) under which detection is possible with probability going to $1$ as $t_{\star} \to \infty$. Finally, we show for PA and UA trees that the amount of correlation, measured by $t_{\star}$, can be estimated with vanishing relative error as $t_{\star} \to \infty$. 
\end{abstract}

%%%%%%%%%%%%%%%%
%%% Document %%%
%%%%%%%%%%%%%%%%

%%%%%%%%%%%%%%%%%%%%%%%%%%%%%%%%%%%%%%%%%%%%
\section{Introduction} \label{sec:intro} %%%
%%%%%%%%%%%%%%%%%%%%%%%%%%%%%%%%%%%%%%%%%%%%

Understanding computational and inference tasks on networks is of paramount importance to solving problems in a variety of fields, including biology, sociology, and machine learning. While many of these tasks are NP-hard in the worst case, most graphs occurring in practice are not worst case, motivating the study of these problems under probabilistic generative models. Increasingly, these problems involve not just a single network but multiple networks that are correlated, and often the crux of the problem lies in understanding how the networks are correlated. Here we introduce a new model of \emph{correlated randomly growing graphs} and study the fundamental questions of detecting correlation and estimating aspects of the correlated structure.

The model is simple and starts with any model of randomly growing graphs. 
A model of randomly growing graphs is specified by a \emph{seed graph} $S$ and a (probabilistic) \emph{growth rule} $\cG$ (also referred to as an \emph{attachment rule}). 
We say that $\left\{ G_{t} \right\}_{t \geq \left| S \right|}$ is a sequence of randomly growing graphs with seed $S$ (with $\left| S \right|$ vertices) and growth rule $\cG$, 
if the following two things hold. 
First, $G_{\left| S \right|} = S$. 
Subsequently, the sequence of graphs is defined inductively using $\cG$: 
given $G_{t}$, the graph $G_{t+1}$ is formed from $G_{t}$ by adding a single vertex that is attached to some of the vertices in $G_{t}$, chosen according to the attachment rule $\cG$. 
We write $G_{n} \sim \cG \left( n, S \right)$ for an $n$-vertex graph generated in this way; see Figure~\ref{fig:schematic} for an illustration.

\begin{figure}[t]
\centering
\includegraphics[width=\textwidth]{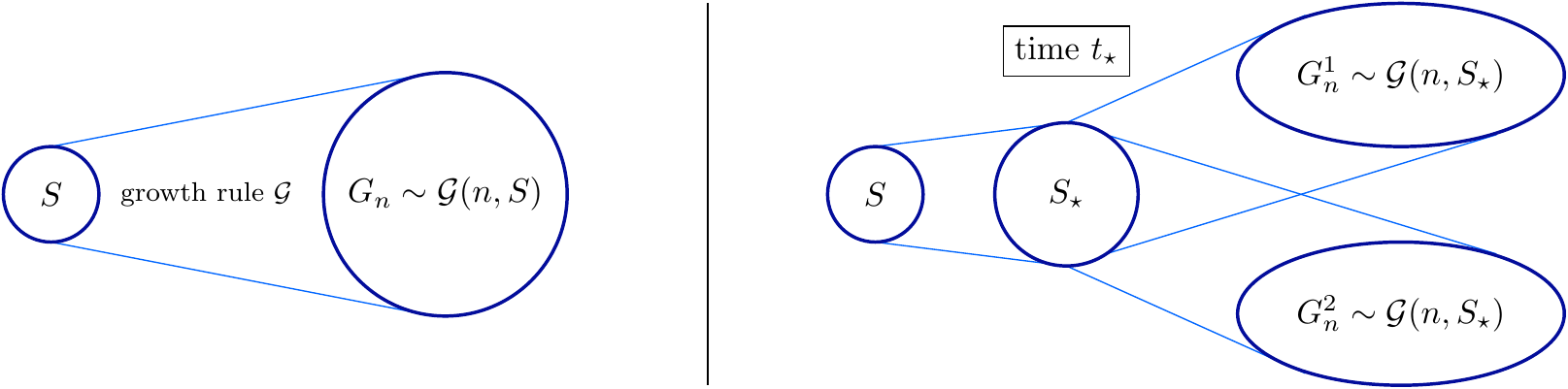}
\caption{Schematic illustrations of the models studied in this paper. \emph{Left:} a randomly growing graph, started from seed $S$ and growing according to growth rule $\cG$. \emph{Right:} two correlated randomly growing graphs, started from seed $S$, grown together 
%(according to $\cG$) 
until time $\tcorr$, and then growing independently.}
\label{fig:schematic}
\end{figure}

For instance, 
an attachment rule might involve a positive integer $m$ 
and the new vertex attaching to $m$ existing vertices chosen i.i.d.\ according to some distribution on the existing vertices. 
Canonical examples include uniform attachment (UA)~\cite{drmota2009random}, 
where each existing vertex is chosen with equal probability, 
and preferential attachment (PA)~\cite{Mah92,BA99,BRST01}, 
where each existing vertex is chosen with probability proportional to its degree. 
The case $m = 1$ corresponds to randomly growing trees.  
We write $\UA \left( n, S \right)$ for a UA tree on $n$ vertices started from the seed tree $S$, 
and similarly $\PA \left( n, S \right)$ for a PA tree on $n$ vertices started from $S$.

We are now ready to introduce the new model of \emph{correlated} randomly growing graphs. To keep things simple, we focus on the setting of two correlated graphs. In addition to a seed graph~$S$ and a growth rule $\cG$, the model takes an additional parameter $\tcorr$, which is a positive integer satisfying $\tcorr \geq \left|S \right|$. 
The model is simple: 
the two graphs $G_{t}^{1}$ and $G_{t}^{2}$ grow together until time~$\tcorr$, 
after which they grow independently. 
More precisely, the distribution of the sequence of 
the pair of graphs $\left\{ \left( G_{t}^{1}, G_{t}^{2} \right) \right\}_{t \geq \left| S \right|}$ is defined as follows. 
\begin{itemize}
\item Initially, the two graphs grow together: 
for $\left| S \right| \leq t \leq \tcorr$ we have that 
$G_{t}^{1} = G_{t}^{2} =: G_{t}$ and $G_{t} \sim \cG \left( t, S \right)$. 
\item Subsequently, the two graphs grow independently: 
conditioned on $G_{\tcorr}$, 
the two sequences of graphs $\left\{ G_{t}^{1} \right\}_{t \geq \tcorr}$ and $\left\{ G_{t}^{2} \right\}_{t \geq \tcorr}$ are independent randomly growing graphs, 
both starting from the graph $G_{\tcorr}$ and growing according to $\cG$. 
\end{itemize} 
This can model, for instance, the citation networks~\cite{redner1998popular} of two scientific fields which initially shared common beginnings but then grew apart. 
We write $\left( G_{n}^{1}, G_{n}^{2} \right) \sim \cCG \left( n, \tcorr, S \right)$ 
for two $n$-vertex graphs $G_{n}^{1}$ and $G_{n}^{2}$ generated according to this model; see Figures~\ref{fig:schematic} and~\ref{fig:correlated_trees_illustration} for illustrations. 
We also write $\CPA \left( n, \tcorr, S \right)$ and $\CUA \left( n, \tcorr, S \right)$ 
for correlated~PA trees and correlated UA trees, respectively. 
To the best of our knowledge, this model of correlated randomly growing graphs has not been studied before; see Section~\ref{sec:related} for discussion of related work.

\begin{figure}[h!]
\centering
\begin{subfigure}{0.45\textwidth}
\centering
\includegraphics[width=0.68\textwidth]{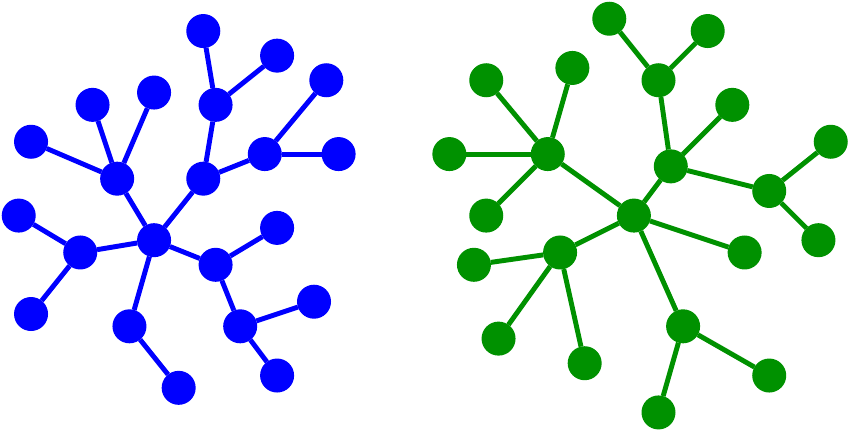}
\caption{Independent trees: $(T_n^1,T_n^2) \sim \mathcal{G}(n,S)^{\otimes 2}$.}
\end{subfigure}
\begin{subfigure}{0.45\textwidth}
\centering
\includegraphics[width=0.68\textwidth]{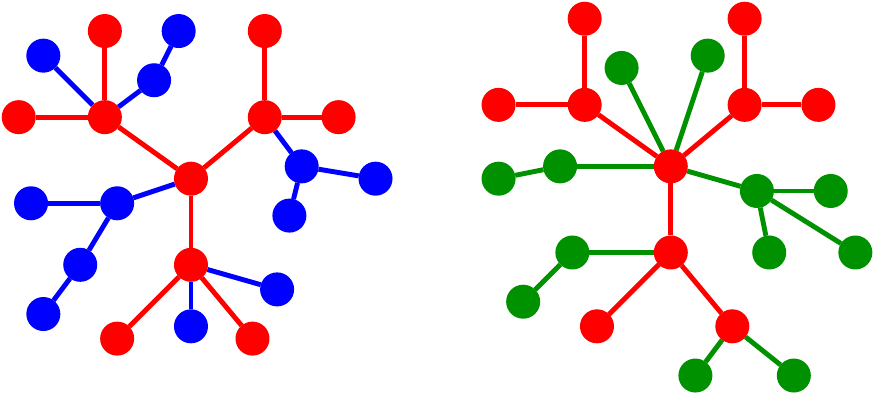}
\caption{Correlated trees: $(T_n^1,T_n^2) \sim \cCG(n,\tcorr,S)$ for $\tcorr = 10$, shared structure in red.}
\end{subfigure}
\caption{Differences between independent trees and correlated trees. 
Here $n = 22$ and $S$ is the unique tree on 3 vertices.}
\label{fig:correlated_trees_illustration}
\end{figure}

This model of correlation satisfies the natural property that the marginal processes are still randomly growing graphs with seed $S$ and rule $\cG$. 
That is, if 
$\left( G_{n}^{1}, G_{n}^{2} \right) \sim \cCG \left( n, \tcorr, S \right)$, 
then 
$G_{n}^{1} \sim \cG \left( n, S \right)$ and $G_{n}^{2} \sim \cG \left( n, S \right)$. 
Also, if $\tcorr = \left| S \right|$, 
then $\left\{ G_{t}^{1} \right\}_{t \geq \left| S \right|}$ and $\left\{ G_{t}^{2} \right\}_{t \geq \left| S \right|}$ 
are independent; 
we then write $\left( G_{n}^{1}, G_{n}^{2} \right) \sim \cG \left( n, S \right)^{\otimes 2}$ to emphasize the independence. 
Thus we see that $\tcorr$ (more precisely, $\tcorr - \left| S \right|$) explicitly measures the amount of correlation among the two graphs. 

%%%%%%%%%%%%%%%%%%%%%%%%%%%%%%%%%%%%%%%%%%%%
\subsection{Questions: detection and estimation} \label{sec:questions} %%%
%%%%%%%%%%%%%%%%%%%%%%%%%%%%%%%%%%%%%%%%%%%%

We study the fundamental questions of detecting correlation and estimating aspects of the correlated structure 
in the model of correlated randomly growing graphs introduced above. 

\textbf{Detection.} 
Given two (unlabeled) $n$-vertex graphs, $G_{n}^{1}$ and $G_{n}^{2}$, can we detect whether they are correlated or not? 
This question can be phrased as a simple hypothesis testing problem. 
Under the null hypothesis $H_{0}$, the two graphs are independent: 
$\left( G_{n}^{1}, G_{n}^{2} \right) \sim \cG \left( n, S \right)^{\otimes 2}$. 
Under the alternative hypothesis, denoted $H_{\tcorr}$, 
the two graphs are correlated, with a shared history until time~$\tcorr$: 
$\left( G_{n}^{1}, G_{n}^{2} \right) \sim \cCG \left( n, \tcorr, S \right)$. 
In brief: 
\begin{equation}\label{eq:hyp_test}
H_{0} : \left( G_{n}^{1}, G_{n}^{2} \right) \sim \cG \left( n, S \right)^{\otimes 2}, 
\qquad \qquad 
H_{\tcorr} : \left( G_{n}^{1}, G_{n}^{2} \right) \sim \cCG \left( n, \tcorr, S \right).
\end{equation} 
Note that we only observe a \emph{snapshot} of the two graphs at time $n$, 
we do not observe their history leading up to this snapshot. 
Is there a test that can distinguish between the two hypotheses with asymptotically (in $n$) non-negligible power? 
Under what circumstances can we distinguish with probability close to $1$? 
Studying these questions is equivalent to understanding the total variation distance between $\cG \left( n, S \right)^{\otimes 2}$ and $\cCG \left( n, \tcorr, S \right)$; 
recall that the total variation distance between two probability measures $P$ and~$Q$ is defined as 
$\TV \left( P, Q \right) := \frac{1}{2} \left\| P - Q \right\|_{1} = \sup_{A} \left| P(A) - Q(A) \right|$. 
We are particularly interested in the limit as $n \to \infty$: 
\begin{equation}\label{eq:limitingTV}
\lim_{n \to \infty} \TV \left( \cCG \left( n, \tcorr, S \right), \cG \left( n, S \right)^{\otimes 2} \right),
\end{equation}
a limit which is well-defined, 
because this total variation distance is non-increasing in $n$ (since one can simulate the future evolution of the process) and nonnegative. 
There exists a test with asymptotically non-negligible power for the hypothesis testing problem in~\eqref{eq:hyp_test} if and only if the quantity in~\eqref{eq:limitingTV} is positive. 

\textbf{Estimation.} If detection is possible, the natural next questions concern estimation. Is it possible to estimate the amount of correlation between two correlated randomly growing graphs? Is it possible to estimate the common shared subgraph? 
Formally, suppose that 
$\left( G_{n}^{1}, G_{n}^{2} \right) \sim \cCG \left( n, \tcorr, S \right)$, 
but $\tcorr$ is unknown. 
How well can we estimate $\tcorr$? 
How well can we estimate the shared subgraph $G_{\tcorr}$?

%%%%%%%%%%%%%%%%%%%%%%%%%%%%%%%%%%%%%%%%%%%%
\subsection{Summary of results and methods} \label{sec:results} %%%
%%%%%%%%%%%%%%%%%%%%%%%%%%%%%%%%%%%%%%%%%%%%

Our results concern the detection and estimation questions discussed in Section~\ref{sec:questions}, and can be summarized as follows. 

\begin{itemize}
\item \textbf{Detecting correlation whenever the seed has an influence.} 
We show that there exists a test with asymptotically (in $n$) non-negligible power for the hypothesis testing problem in~\eqref{eq:hyp_test} whenever the seed graph $S$ has an influence on the randomly growing graph $\cG \left( n, S \right)$ (in a sense to be made precise). 
This latter property has been shown for PA trees~\cite{BMR15,CDKM15} and UA trees~\cite{BEMR17}---and is conjectured to hold more broadly---which implies that detecting correlation is possible for these models. 
Remarkably, the results show that correlation can be detected whenever $\tcorr > \left|S \right|$, that is, even if the graphs are grown together for just a \emph{single time step}. 
 
\item \textbf{Detecting correlation with probability going to $1$ as $\tcorr \to \infty$.} 
We give a general condition under which correlation can be detected with probability going to $1$ as $\tcorr \to \infty$. We conjecture that this condition holds for a broad family of randomly growing graphs, and in particular, we show that it holds for PA and UA trees. 

\item \textbf{Estimating $\tcorr$ with vanishing relative error as $\tcorr \to \infty$.} 
Focusing on PA and UA trees, we show that the amount of correlation, measured by $\tcorr$, can be estimated with vanishing relative error as $\tcorr \to \infty$. 
\end{itemize}

In the most general setting, we establish results for sequential attachment rules that are \emph{Markov}, in the sense that for every $t \geq \left|S \right|$, we have that 
\[
\p \left( G_{t+1} = G \, \middle| \, G_{\left|S \right|}, G_{\left| S \right| + 1}, \ldots, G_{t} \right) 
= \p \left( G_{t+1} = G \, \middle| \, G_{t} \right),
\]
where 
$\left\{ G_{t} \right\}_{t \geq \left| S\right|}$ 
is a sequence of randomly growing graphs starting from seed $S$. 
This is a natural assumption, since in many real-world networks 
new nodes added to the network will not have access to the history of the network. 
We also establish stronger results for PA and UA trees, which are canonical models of randomly growing graphs. For what follows it will be useful to define 
\[
\Range \left( \cG, S \right) := \left\{ G : \exists n \text{ such that if } G_{n} \sim \cG \left( n, S \right) \text{ then } \p \left( G_{n} = G \right) > 0 \right\}, 
\]
the set of all possible graphs that can be obtained with positive probability starting from seed graph $S$ via the attachment rule $\cG$. 
We are now ready to detail our results. 

\subsubsection{Detecting correlation whenever the seed has an influence} 
Our first result is a general result that shows that correlation can be detected 
% among correlated randomly growing graphs 
whenever the seed graph has an influence in the underlying randomly growing graph model. 

\begin{theorem}[Detecting correlation whenever the seed has an influence]\label{thm:positive_power_detection}
Fix a seed graph $S$, 
a positive integer $\tcorr$ such that $\tcorr > \left| S \right|$, 
and a Markov sequential attachment rule $\cG$. 
Suppose that there are graphs $G$ and $G'$ satisfying that 
$\left| G \right| = \left| G' \right| = \tcorr$, 
that 
$G, G' \in \Range \left( \cG, S \right)$, 
and that 
\begin{equation}\label{eq:influence_seed}
\lim_{n \to \infty} \TV \left( \cG \left( n, G \right), \cG \left( n, G' \right) \right) > 0.
\end{equation}
Then 
\[
\lim_{n \to \infty} \TV \left( \cCG \left( n, \tcorr, S \right), \cG \left( n, S \right)^{\otimes 2} \right) > 0.
\]
\end{theorem}

Remarkably, this result holds whenever $\tcorr > \left| S \right|$, showing that correlation can be detected even if the graphs are grown together for just a single time step. 

The condition in~\eqref{eq:influence_seed} captures formally what it means for the seed to have an influence. 
The study of the influence of the seed in randomly growing graphs was initiated by Bubeck, Mossel, and R\'acz, who studied this question in PA trees~\cite{BMR15}. 
They showed that for any two seed trees $S$ and $T$ with at least $3$ vertices and different degree profiles, 
$\lim_{n \to \infty} \TV \left( \PA \left( n, S \right), \PA \left( n , T \right) \right) > 0$ holds. This already implies that~\eqref{eq:influence_seed} holds for PA trees whenever $\tcorr > 3$. 
In subsequent work, Curien, Duquesne, Kortchemski, and Manolescu 
showed that 
$\lim_{n \to \infty} \TV \left( \PA \left( n, S \right), \PA \left( n , T \right) \right) > 0$ 
whenever $S$ and $T$ are nonisomorphic trees with at least $3$ vertices~\cite{CDKM15}. 
This was then showed for UA trees as well by Bubeck, Eldan, Mossel, and R\'acz~\cite{BEMR17}. 
We refer to the recent survey~\cite{RB17} for an exposition of these results and the associated techniques. 
These results are summarized in the following two theorems. 
\begin{theorem}[\cite{BMR15,CDKM15}]\label{thm:PA_seed}
The seed has an influence in PA trees in the following sense.
We have that 
$\lim_{n \to \infty} \TV \left( \PA \left( n, S \right), \PA \left( n , T \right) \right) > 0$ for any trees $S$ and $T$ that are nonisomorphic and have at least 3 vertices. 
\end{theorem}
\begin{theorem}[\cite{BEMR17}]\label{thm:UA_seed}
The seed has an influence in UA trees in the following sense.
We have that 
$\lim_{n \to \infty} \TV \left( \UA \left( n, S \right), \UA \left( n , T \right) \right) > 0$ for any trees $S$ and $T$ that are nonisomorphic and have at least 3 vertices. 
\end{theorem}
These two theorems, together with Theorem~\ref{thm:positive_power_detection}, 
directly imply that correlation can be detected in PA and UA trees. 
These results are formalized in the following two corollaries. 
\begin{corollary}[Detecting correlation in PA trees]\label{thm:detection_PA}
Let $S$ be a finite tree with at least two vertices. 
Let $\tcorr \in \N$ be such that $\tcorr > \left| S \right|$ and $\tcorr > 3$. 
Then 
\[
\lim_{n \to \infty} \TV \left( \CPA \left( n, \tcorr, S \right), \PA \left( n, S \right)^{\otimes 2} \right) > 0.
\]
\end{corollary}
\begin{corollary}[Detecting correlation in UA trees]\label{thm:detection_UA}
Let $S$ be a finite tree. 
Let $\tcorr \in \N$ be such that $\tcorr > \left| S \right|$ and $\tcorr > 3$. 
Then 
\[
\lim_{n \to \infty} \TV \left( \CUA \left( n, \tcorr, S \right), \UA \left( n, S \right)^{\otimes 2} \right) > 0.
\]
\end{corollary}

Theorem~\ref{thm:positive_power_detection} reduces detecting correlation to detecting the influence of the seed. As such, it can be viewed as an \emph{existence} result, since it does not give specific statistics of the two graphs that can detect correlation. 
We therefore complement Theorem~\ref{thm:positive_power_detection} and Corollaries~\ref{thm:detection_PA} and~\ref{thm:detection_UA} by providing  alternative, \emph{algorithmic} proofs of Corollary~\ref{thm:detection_PA} and Corollary~\ref{thm:detection_UA}. Specifically, inspired by~\cite{BMR15}, we will show that the maximum degrees of the two trees can be used to detect correlation in PA trees. Furthermore, inspired by~\cite{BEMR17}, we will show that there are certain statistics that measure global balancedness properties of a tree (and which are efficiently computable) that can be used to detect correlation in UA trees. See Section~\ref{sec:detection_specific} for details. 

\subsubsection{Detecting correlation with probability going to $1$ as $\tcorr \to \infty$} 
Ideally, we would like to detect correlation with probability close to $1$. 
However, for any fixed finite $\tcorr$, the probability of successfully being able to detect correlation is strictly bounded away from~$1$. This is simply because if $G_{\tcorr}^{1} \sim \cG \left( \tcorr, S \right)$ and $G_{\tcorr}^{2} \sim \cG \left( \tcorr, S \right)$ are independent, then there is a positive probability (which depends only on $\cG$ and $\tcorr$) that $G_{\tcorr}^{1} = G_{\tcorr}^{2}$. With this probability we may couple $\cG \left( n, S \right)^{\otimes 2}$ and $\cCG \left( n, \tcorr, S \right)$, 
showing that there exists $\eps = \eps \left( \cG, \tcorr \right) > 0$ such that 
\begin{equation}\label{eq:TV_bounded_1minuseps}
\TV \left( \cCG \left( n, \tcorr, S \right), \cG \left( n, S \right)^{\otimes 2} \right) 
\leq 1 - \eps 
\end{equation}
for every $n \geq \tcorr$. 
Our focus is thus to show that correlation can be detected with probability going to $1$ as $\tcorr \to \infty$. 
We first present a general result, 
which gives a sufficient condition on the underlying model of randomly growing graphs for this to occur. 
\begin{theorem}[Detecting correlation with probability going to $1$ as $\tcorr \to \infty$]\label{thm:high_prob_detection}
Fix a seed graph~$S$ and a Markov sequential attachment rule $\cG$. 
Let $\left\{ G_{t} \right\}_{t \geq \left| S \right|}$ be a sequence of randomly growing graphs with seed $S$ and attachment rule $\cG$. 
Suppose that there is a function $f : \Range \left( \cG, S \right) \to \R$ such that the limit $\lim_{t \to \infty} f \left( G_{t} \right) =: f_{\infty}$ exists almost surely and that $f_{\infty}$ is an absolutely continuous random variable. 
Then we have that 
\[
\lim_{\tcorr \to \infty} \lim_{n \to \infty} 
\TV \left( \cCG \left( n, \tcorr, S \right), \cG \left( n, S \right)^{\otimes 2} \right)
= 1.
\]
\end{theorem}
The test that distinguishes correlated graphs from independent graphs is simple: 
we compare $\left| f \left( G_{n}^{1} \right) - f \left( G_{n}^{2} \right) \right|$ to an appropriately chosen threshold. 
The idea behind the proof is that this quantity tends to $0$ as $\tcorr \to \infty$ under the alternative hypothesis $H_{\tcorr}$, 
but $f \left( G_{n}^{1} \right)$ and $f \left( G_{n}^{2} \right)$ are independent under the null hypothesis $H_{0}$, so the difference stays away from $0$ in this case. 

Theorem~\ref{thm:high_prob_detection} is a general theorem that we expect applies to a wide class of models of randomly growing graphs. To demonstrate its utility, we show that PA trees and UA trees satisfy its conditions. 
For PA trees, we may choose $f$ to be the normalized maximum degree. 
For both cases, we may choose $f$ to be a function that is closely related to notions of centrality in trees. 
These have been used to study a variety of statistical problems, such as estimating the source of a rumor on a tree~\cite{shah2010detecting,sz11rumor,shah2016finding} and estimating the seed in randomly growing trees~\cite{BDL16, LP19, DR19}. 
We thus obtain the following results for PA and UA trees. 
\begin{theorem}\label{thm:high_prob_detection_PA}
Let $S$ be a finite tree with at least two vertices. 
Then 
\[
\lim_{\tcorr \to \infty} \lim_{n \to \infty} 
\TV \left( \CPA \left( n, \tcorr, S \right), \PA \left( n, S \right)^{\otimes 2} \right) 
= 1.
\]
\end{theorem}
\begin{theorem}\label{thm:high_prob_detection_UA}
Let $S$ be a finite tree. 
Then 
\[
\lim_{\tcorr \to \infty} \lim_{n \to \infty} 
\TV \left( \CUA \left( n, \tcorr, S \right), \UA \left( n, S \right)^{\otimes 2} \right) 
= 1.
\]
\end{theorem}

\subsubsection{Estimating $\tcorr$ with vanishing relative error as $\tcorr \to \infty$} 
We now turn to questions of estimation. These are more involved than questions concerning detection and hence we restrict our attention to PA and UA trees, started from the seed $S = S_{2}$, the unique tree on two vertices. 
We focus on estimating $\tcorr$, which measures the amount of correlation between the two correlated trees; 
we leave the very interesting question of estimating the common subgraph $G_{\tcorr}$ for future work (see Section~\ref{sec:discussion}). 
Ideally, we would like good estimates of $\tcorr$ that hold with probability close to $1$.  
From~\eqref{eq:TV_bounded_1minuseps} it follows that this is only possible as $\tcorr \to \infty$. 

Our main result on estimation is that $\tcorr$ can be estimated with vanishing relative error as $\tcorr \to \infty$; this is the content of the following theorem.  

\begin{theorem}[Estimating $\tcorr$ in PA and UA trees]\label{thm:estimation} 
Let $S = S_{2}$ be the unique tree on two vertices 
and let $\left( T_{n}^{1}, T_{n}^{2} \right) \sim \CPA \left( n, \tcorr, S \right)$. 
There exists an estimator 
$\wh{t}_{n} \equiv \wh{t} \left( T_{n}^{1}, T_{n}^{2} \right)$, 
computable in polynomial time, such that 
\[
\lim_{\tcorr \to \infty} \liminf_{n \to \infty} 
\p \left( \left( 1 - \tfrac{\log \log \tcorr}{\sqrt{\log \tcorr}} \right) \tcorr 
\leq \wh{t}_{n} 
\leq \left( 1 + \tfrac{\log \log \tcorr}{\sqrt{\log \tcorr}} \right) \tcorr \right)
= 1. 
\]
The same result also holds when 
$\left( T_{n}^{1}, T_{n}^{2} \right) \sim \CUA \left( n, \tcorr, S \right)$. 
\end{theorem}

In other words, the relative error of the estimator $\wh{t}_{n}$ is bounded by $\log \log \left( \tcorr \right) / \sqrt{\log \tcorr}$, 
with probability close to $1$, for large enough $\tcorr$. 
The proof of Theorem~\ref{thm:estimation} is the most involved proof in this paper and so we give here a high level overview of the proof strategy. 
The proof works equally for both PA and UA trees, with only minor changes.

The main idea is to match several pairs of vertices across the two trees. To explain this more precisely, we introduce some notation. Let $\left\{T_{n} \right\}_{n \ge 2}$ be a sequence of growing trees with seed $S_2$. For a vertex $v$ in $T_{n}$, let $\tau(v)$ be the {\it timestamp} of $v$. That is, $\tau(v) = k$ if $v$ is not in $T_{k-1}$ but is introduced in $T_{k}$. The two initial vertices are labelled 
% with timestamps 
$1$ and $2$ arbitrarily. We say that a pair of vertices 
$\left(v^{1},v^{2} \right)$, where $v^{1} \in V \left(T_{n}^{1} \right)$ and $v^{2} \in V\left(T_{n}^{2} \right)$, is {\it correctly matched} if $\tau(v^1) = \tau(v^2)$. 

\textbf{Correctly matching the centroids.} Let $\theta^{1}(n), \theta^{2}(n)$ be the centroids of the trees $T_n^1$ and~$T_n^2$, respectively (we rigorously define the notion of a tree centroid in Section \ref{sec:detection_large_tcorr_specific}). Jog and Loh~\cite{persistent_centrality} proved that PA and UA trees with seed $S_{2}$ have the {\it persistent centroid property}: almost surely, there is a finite time $N$ such that for all $t \geq N$, we have that $\theta(t) = \theta(N) := \theta$. Using this fact, it follows that the pair $\left(\theta^{1}(n), \theta^{2}(n) \right)$ is correctly matched with probability tending to 1 as $\tcorr \to \infty$. Although we have so far only matched one pair of vertices in the two graphs, this provides an important frame of reference going forward, to analyze the correlated structure in the two trees. 

\textbf{Matching neighbors of the centroids.} 
Next, assuming the high-probability event $\theta^1(n) = \theta^2(n) = \theta$, we consider the rooted trees $\left(T_{n}^{1}, \theta \right)$ and $\left(T_{n}^{2},\theta \right)$, with the goal of matching many neighbors of the centroids. We do so by examining subtrees of the two rooted trees. 
% If $v$ is a neighbor of $\theta$ in $T_n^i$, then 
Let $(T_n^i,\theta)_{v \downarrow}$ denote the subtree of the rooted tree $(T_n^i,\theta)$ that has root $v$. In other words, the tree $(T_n^i,\theta)_{v \downarrow}$ consists of all vertices $u$ such that the unique path connecting $u$ and $\theta$ passes through $v$. 

The idea behind matching neighbors of the centroid is the ``rich-get-richer'' property of subtrees. To illustrate this concept, suppose that for a tree growing via uniform attachment, we consider neighbors $u$ and $v$ of $\theta$, and $|(T_{\tcorr},\theta)_{u \downarrow} |$ is much larger than $|(T_{\tcorr},\theta)_{v \downarrow} |$. Under the UA rule, the probability that a new vertex joins a subtree is proportional to the number of vertices in the subtree; thus it is very unlikely that $|(T_t,\theta)_{v \downarrow}|$ exceeds $|(T_t,\theta)_{u \downarrow}|$ at any future time $t$. 
Similar behavior holds for PA trees as well. This intuition tells us that if $|(T_{\tcorr},\theta)_{u \downarrow} |$ is much larger than $|(T_{\tcorr},\theta)_{v \downarrow} |$, then we should have $|(T_n^i, \theta)_{u \downarrow} | > | (T_n^i,\theta)_{v \downarrow} |$ for both $i = 1$ and $i = 2$. 

Taking this idea one step further, we may expect that if the largest $R$ subtrees (for some positive integer $R$) of $(T_{\tcorr}, \theta)$ do not have sizes that are too close to each other, then these should be the same $R$ largest subtrees in $(T_n^i,\theta)$, for both $i = 1$ and $i = 2$. Therefore, we will match the neighbors of the centroids with the largest subtrees, the second largest subtrees, and so on, until the $R$th largest subtrees. We indeed prove that such a matching procedure for the neighbors of the centroids, based on subtree ranking, gives us all correct matchings with probability tending to 1 as $\tcorr \to \infty$.

\textbf{Constructing estimators for $\tcorr$.} 
Suppose that $\left(v^{1},v^{2} \right)$ are a correctly matched pair of neighbors of the centroid. We can construct an estimator for $\tcorr$ by comparing the subtree sizes corresponding to $v^1$ and $v^2$. The evolution of subtree sizes in PA and UA trees exhibit the following stability property: the fraction of vertices that lie in a particular subtree has a limit almost surely as the size of the tree tends to infinity. This follows from viewing the subtree growth as a P\'{o}lya urn process. 

We then expect that as we send $\tcorr \to \infty$, the difference between $\frac{1}{n} |(T_n^1,\theta)_{v^1 \downarrow} |$ and $\frac{1}{n} |(T_n^2,\theta)_{v^2 \downarrow} |$ is close to 0, even for large $n$. We exploit this property to construct a nearly unbiased estimator for $\tcorr$ based on the difference between $\frac{1}{n} |(T_n^1,\theta)_{v^1 \downarrow} |$ and $\frac{1}{n} |(T_n^2,\theta)_{v^2 \downarrow} |$. However, the variance of the estimator corresponding to the matched pair $(v^1,v^2)$ is not small enough to ensure that we can estimate $\tcorr$ with vanishing relative error. 
This is the reason for matching many pairs of points: we can then average the estimators corresponding to many correctly matched pairs of vertices, in order to reduce the variance. We finish by applying Chebyshev's inequality.

%%%%%%%%%%%%%%%%%%%%%%%%%%%%%%%%%%%%%%%%%%%%
\subsection{Related work} \label{sec:related} %%%
%%%%%%%%%%%%%%%%%%%%%%%%%%%%%%%%%%%%%%%%%%%%

Though this paper is, to the best of our knowledge, 
the first to introduce this model of correlated randomly growing graphs, 
it is closely related to several well-studied problems in the literature.

\paragraph{Graph matching and the correlated Erd\H{o}s-R\'{e}nyi model.} 
Perhaps the most well-known related problem is graph matching. In this setting, we are given two graphs and we want to find a labeling on the vertices that maximizes the similarity between the two graphs. The applications of this problem are numerous, spanning data privacy in social networks~\cite{narayanan2009anonymizing,pedarsani2011privacy}, protein-protein interaction networks~\cite{singh2008global}, computer vision~\cite{cho2012progressive}, 
pattern recognition~\cite{conte2004thirty,berg2005shape}, machine learning~\cite{cour2007balanced}, and more. 
This problem is NP-hard in the worst case 
(see, e.g., the surveys~\cite{conte2004thirty,livi2013graph}); 
in fact, it is even hard to approximate under some hardness assumptions~\cite{o2014hardness}. 
However, most graphs occurring in applications are not worst case, 
which motivates the study of the graph matching problem under 
probabilistic generative models. 

The simplest random graph model is the Erd\H{o}s-R\'enyi random graph $G(n,p)$, 
which has $n$ vertices and every pair is connected with probability $p$, 
independently of any other pair. 
Thus naturally the simplest model of correlated random graphs involves two Erd\H{o}s-R\'enyi random graphs that are correlated. 
This model was introduced by Pedarsani and Grossglauser~\cite{pedarsani2011privacy} 
and has been widely studied in the past decade in several communities, including computer science, network science, information theory, probability, and statistics~\cite{yartseva2013performance,lyzinski2014seeded,kazemi2015growing,kazemi2015can,korula2014efficient,cullina2016improved,cullina2017exact,barak2019,mossel2019seeded,ding2018efficient,fan2019spectral1,fan2019spectral2,ganassali2020tree}. 
These works have resulted in obtaining the fundamental information-theoretic limits~\cite{cullina2016improved,cullina2017exact} and recent algorithmic advances~\cite{barak2019,mossel2019seeded,ding2018efficient,fan2019spectral1,fan2019spectral2}. 
The model of correlated randomly grown graphs introduced in this paper is fundamentally different from the correlated Erd\H{o}s-R\'enyi model 
and thus it is not possible to directly compare our results with those in these papers. 
Importantly, while Erd\H{o}s-R\'enyi random graphs have no inherent structure, 
the model of correlated randomly grown graphs is motivated by the fact that many real-world networks form via a growth process.

In the correlated Erd\H{o}s-R\'enyi model the pair $\left(G^{1}, G^{2} \right)$ is constructed as follows. 
First, sample an unobserved base graph $G^0 \sim G(n,p)$. 
Next, conditioned on $G^0$, construct $G^1$ and $G^2$ independently by including any given edge with probability $q$. 
Both $G^1$ and $G^2$ are distributed according to $G(n,pq)$, and they are correlated in the sense that the presence of specified edges are correlated. There is also a ``true'' labelling of the vertices in $G^1$ and $G^2$, given by inheriting the labels of the unobserved base graph $G^0$. The goal of the graph matching problem 
%for a pair of correlated Erd\H{o}s-R\'{e}nyi graphs 
is to recover this true labelling (up to isomorphism). 
There is also a modified version of the problem in which the algorithm has side information in the form of a small number of matched vertices.

The problems of detecting and estimating correlation in a pair of randomly grown graphs can be viewed as an analog of the graph matching problem (without side information) for these kind of graphs. We highlight several papers in the graph matching literature that have related ideas. Barak, Chou, Lei, Schramm, and Sheng study the problem of detecting correlated structure for a pair of Erd\H{o}s-R\'{e}nyi graphs~\cite{barak2019}. Their approach to solving the detection problem in certain regimes relies on subgraph counts. Our approach is vastly different, relying on extremal statistics of the graphs (e.g., maximum degree, minimum anti-centrality) and general balancedess properties (all of which may be computed efficiently). Kazemi, Yartseva, and Grossglauser study a variant of the graph matching problem in a pair of correlated Erd\H{o}s-R\'{e}nyi graphs when there is {\it partial overlap} between the graphs; that is, there are vertices in either graph that are not part of any correlated structure \cite{kazemi2015can}. Our model of correlated randomly grown graphs has a similar characteristic: the subgraph of the shared history, $G_{\tcorr}$, is common, and the other vertices in the pair of graphs do not necessarily correspond to each other if they were born after time $\tcorr$. Their goal is somewhat different from ours; they aim to estimate the common part, with knowledge of the amount of overlap. On the other hand, we focus on estimating the amount of correlation, or equivalently, the size of the common part. 

Korula and Lattanzi study a version of the graph matching problem for preferential attachment graphs~\cite{korula2014efficient}, though the manner in which they generate a pair of correlated 
% preferential attachment 
graphs is fundamentally different from our model. Similar to the process of generating correlated Erd\H{o}s-R\'{e}nyi graphs, they generate a base graph $G^0$ according to preferential attachment and independently construct $G^1$ and $G^2$ by including a given edge in $G^0$ with some fixed probability. However, in this case $G^1$ and $G^2$ are not distributed according to preferential attachment, which is unnatural. 
We also note that they require the use of side information in their algorithm, while we do not assume this, since it is possible to match key information in our case (e.g., matching the centroid).

\paragraph{Inferring the history of a dynamic graph process from a snapshot.} Our work naturally fits under this broad category in terms of the problem scope and the techniques used. There have been a variety of works of this theme in recent years, 
including rumor source estimation~\cite{shah2010detecting,sz11rumor,shah2016finding,fanti2015spy,fanti2016irregular,fanti2017hide}, 
the influence of the seed in randomly growing graphs~\cite{BMR15,CDKM15,BEMR17}, and 
finding the earliest vertices in randomly growing graphs~\cite{BDL16,LP19,DR19}. 
Applications include reconstructing the evolution of biological networks~\cite{navlakha2011network}. 

The works on the influence of the seed in randomly growing graphs~\cite{BMR15,CDKM15,BEMR17} are particularly relevant to our work---we refer to Section~\ref{sec:results} for a discussion of these detailed connections. These connections are further touched upon in the proofs. 

The notion of centrality in trees plays a significant role in our techniques (for the results specific to PA and UA trees), and in many of the cited works. Shah and Zaman formulated the notion of {\it rumor centrality} for maximum likelihood estimation of the source of a diffusion on a tree~\cite{shah2010detecting,sz11rumor,shah2016finding}. 
Bubeck, Devroye, and Lugosi introduced a related centrality measure based on subtree sizes to obtain confidence intervals for the first vertex in a PA or UA tree~\cite{BDL16}. This centrality measure lends itself to an easier analysis with PA and UA trees, since the evolution of subtree sizes can be understood as P\'{o}lya urn processes. Subsequently, this centrality measure was used by Lugosi and Pereira~\cite{LP19} and by Devroye and Reddad~\cite{DR19} for the more general problem of obtaining confidence intervals for the seed graph of a UA tree, as well as for the earliest vertices. Jog and Loh showed that UA trees and PA trees exhibit the persistent centroid property: the location of the centroid (with respect to the centrality measure of~\cite{BDL16}) only changes finitely many times as the number of vertices in the tree increases~\cite{persistent_centrality, jog_loh_sublinear_pa}. 
We are able to leverage these previous results on centrality 
in our study of the detection and estimation problems for PA and UA trees.

Bhamidi, Jin, and Nobel studied a variant of the preferential attachment model with a change point~\cite{bhamidi2018change} (see also~\cite{banerjee2018fluctuation})---this shares some similar elements to our model but is fundamentally different. In their model, they examine a single PA tree where, at some time point, the attachment rule changes. The goal is to estimate this change point, and to do so, they use knowledge of the history of the graph. Our problem can be viewed as a change point problem as well, but in a much different sense. Both of the randomly grown graphs have the marginal distribution of a standard randomly grown graph, and the correlation time $\tcorr$ may be interpreted as a change point when the two growing graphs begin to evolve independently. Also, we observe a single snapshot, rather than the entire history, which is a more appropriate and interesting setting for our problem.

Finally, there are many important aspects of modeling network formation that are beyond the scope of the present article. 
We refer the reader to the recent work of Overgoor, Benson, and Ugander~\cite{overgoor2019choosing}, which unifies a host of network formation models using a framework based on discrete choice theory. 
(See also the references therein for an overview of the related literature.) 
Our hope is that the novel phenomena presented in this article 
can contribute to the broader discussion on modeling the formation of multiple correlated networks.

%%%%%%%%%%%%%%%%%%%%%%%%%%%%%%%%%%%%%%%%%%%%
\subsection{Discussion and open problems} \label{sec:discussion} %%%
%%%%%%%%%%%%%%%%%%%%%%%%%%%%%%%%%%%%%%%%%%%%

This paper initiates the study of correlated randomly growing graphs and leaves open several problems. We end the introduction by discussing possible future directions. 

\begin{itemize}
\item \textbf{Estimating the correlation time $\tcorr$.} We have shown (in PA and UA trees) that the correlation time $\tcorr$ can be estimated with vanishing relative error as $\tcorr \to \infty$. It would be interesting to understand the limits of how well $\tcorr$ can be estimated. 
\item \textbf{Estimating the common subgraph $G_{\tcorr}$.} It is of great interest to estimate the common subgraph $G_{\tcorr}$ shared by the two correlated randomly growing graphs. This question can be formalized in several ways: for instance, we might want to find a large subgraph of $G_{\tcorr}$ or a small supergraph of $G_{\tcorr}$, with probability close to $1$. 
Recent work by Lugosi and Pereira~\cite{LP19} and Devroye and Reddad~\cite{DR19} (following work by Bubeck, Devroye, and Lugosi~\cite{BDL16}) 
has studied seed-finding algorithms for UA trees. 
We suspect that their results and the techniques they have developed will be useful for estimating $G_{\tcorr}$. 
\item \textbf{Other models of randomly growing graphs.} 
In our work we focus on PA and UA trees when studying specific models of randomly growing graphs. Our general result in Theorem~\ref{thm:positive_power_detection} says that correlation can be detected if~\eqref{eq:influence_seed} holds. This is a much weaker form of the influence of the seed than is established in Theorems~\ref{thm:PA_seed} and~\ref{thm:UA_seed} for PA and UA trees. 
Are there models of randomly growing graphs for which it is possible to show that~\eqref{eq:influence_seed} holds even if showing the analogue of Theorems~\ref{thm:PA_seed} and~\ref{thm:UA_seed} is currently out of reach? 
\item \textbf{Large amounts of correlation.} In our work we have focused on $\tcorr$ being fixed compared to the graph size $n$. What if $\tcorr$ is a function of $n$? This introduces much more correlation among the two graphs 
and it would be interesting to understand how much stronger results can be obtained. 
\item \textbf{Three or more correlated graphs.} The introduced model of correlated randomly growing graphs naturally extends to three or more correlated graphs. How do the questions of detection and estimation change in this setting? 
For instance, is it much easier to estimate the common subgraph $G_{\tcorr}$ if we have samples from many correlated graphs?
\end{itemize}

%%%%%%%%%%%%%%%%%%%%%%%%%%%%%%%%%%%%%%%%%%%%
\subsection{Outline} \label{sec:outline} %%%
%%%%%%%%%%%%%%%%%%%%%%%%%%%%%%%%%%%%%%%%%%%%

The rest of the paper is organized as follows. We start with proving Theorem~\ref{thm:positive_power_detection} in Section~\ref{sec:detection_seed}. 
We then present explicit algorithmic proofs of Corollaries~\ref{thm:detection_PA} and~\ref{thm:detection_UA} in Section~\ref{sec:detection_specific}. 
In Section~\ref{sec:detection_large_tcorr} we turn to detecting correlation with probability going to $1$ as $\tcorr \to \infty$ 
and prove Theorems~\ref{thm:high_prob_detection},~\ref{thm:high_prob_detection_PA}, and~\ref{thm:high_prob_detection_UA}. 
Finally, we turn to estimating $\tcorr$ as $\tcorr \to \infty$. 
We first provide an initial, coarse estimate of~$\tcorr$ in Section~\ref{sec:estimation_tcorr_coarse}; this section contains the main ideas of our estimators. 
However, further ideas are needed in order to obtain an estimator of $\tcorr$ which has vanishing relative error as $\tcorr \to \infty$: 
these, and a proof of Theorem~\ref{thm:estimation}, can be found in Section~\ref{sec:estimation_large_tcorr}. 

%%%%%%%%%%%%%%%%%%%%%%%%%%%%%%%%%%%%%%%%%%%%%%%%%%
\section{Detecting correlation when the seed has an influence} \label{sec:detection_seed} %%%
%%%%%%%%%%%%%%%%%%%%%%%%%%%%%%%%%%%%%%%%%%%%%%%%%%

In this section we prove Theorem~\ref{thm:positive_power_detection}. To abbreviate notation, in the following we denote by $\p_{0}$ the underlying probability measure when 
$\left( G_{n}^{1}, G_{n}^{2} \right) \sim \cG \left( n, S \right)^{\otimes 2}$ 
and by $\p_{\tcorr}$ the underlying probability measure when 
$\left( G_{n}^{1}, G_{n}^{2} \right) \sim \cCG \left(n, \tcorr, S \right)$. 
Furthermore, 
for a graph $H$  
we denote by $\p_{H}$ the probability measure on the sequence of randomly growing graphs $\left\{ G_{n} \right\}_{n \geq \left| H \right|}$ with seed $H$ and attachment rule $\cG$. 

\begin{proof}[Proof of Theorem~\ref{thm:positive_power_detection}] 
From~\eqref{eq:influence_seed} it follows that there exist $\delta > 0$ 
and a sequence $\left\{ \cE_{n} \right\}_{n \geq \tcorr}$ such that 
\begin{equation}\label{eq:TV_diff_delta}
\left| \p_{G} \left( G_{n} \in \cE_{n} \right) - \p_{G'} \left( G_{n} \in \cE_{n} \right) \right| \geq \delta 
\end{equation}
for every $n \geq \tcorr$. Turning now to a pair of graphs 
$\left( G_{n}^{1}, G_{n}^{2} \right)$, with $n \geq \tcorr$, we consider the event 
\[
\left\{ G_{n}^{1} \in \cE_{n} \right\} 
\cap 
\left\{ G_{n}^{2} \in \cE_{n} \right\}. 
\]
Under the null hypothesis $H_{0}$, the two graphs $G_{n}^{1}$ and $G_{n}^{2}$ are independent, and thus we have that 
\[
\p_{0} \left( G_{n}^{1} \in \cE_{n}, G_{n}^{2} \in \cE_{n} \right) 
= \p_{0} \left( G_{n}^{1} \in \cE_{n} \right) \p_{0} \left( G_{n}^{2} \in \cE_{n} \right) 
= \left( \p_{S} \left( G_{n} \in \cE_{n} \right) \right)^{2}.
\]
Note also that by conditioning on the graph at time $\tcorr$ 
and using the fact that the sequential attachment rule $\cG$ is Markov, 
we have that 
\[
\mu := \p_{S} \left( G_{n} \in \cE_{n} \right) 
= \sum_{H\, :\, \left| H \right| = \tcorr} \p_{H} \left( G_{n} \in \cE_{n} \right) \p_{S} \left( G_{\tcorr} = H \right),
\]
where the sum is over all graphs on $\tcorr$ vertices. 

Turning 
%now 
to the alternative hypothesis $H_{\tcorr}$, we can again condition on the graph at time $\tcorr$, and use the fact $G_{n}^{1}$ and $G_{n}^{2}$ are independent conditioned on the graph at time $\tcorr$. We thus obtain that 
\begin{multline*}
\p_{\tcorr} \left( G_{n}^{1} \in \cE_{n}, G_{n}^{2} \in \cE_{n} \right) \\
\begin{aligned}
&=  \sum_{H\, :\, \left| H \right| = \tcorr} 
\p_{\tcorr} \left( G_{n}^{1} \in \cE_{n}, G_{n}^{2} \in \cE_{n} \, \middle| \, G_{\tcorr}^{1} = G_{\tcorr}^{2} = H \right) 
\p_{S} \left( G_{\tcorr} = H \right) \\
&= \sum_{H\, :\, \left| H \right| = \tcorr} 
\p_{\tcorr} \left( G_{n}^{1} \in \cE_{n} \, \middle| \, G_{\tcorr}^{1} = G_{\tcorr}^{2} = H \right) 
\p_{\tcorr} \left( G_{n}^{2} \in \cE_{n} \, \middle| \, G_{\tcorr}^{1} = G_{\tcorr}^{2} = H \right) 
\p_{S} \left( G_{\tcorr} = H \right) \\
&= \sum_{H\, :\, \left| H \right| = \tcorr} 
\left( \p_{H} \left( G_{n} \in \cE_{n} \right) \right)^{2} \p_{S} \left( G_{\tcorr} = H \right).
\end{aligned}
\end{multline*}
Altogether, we have thus obtained that 
\begin{multline*}
\p_{\tcorr} \left( G_{n}^{1} \in \cE_{n}, G_{n}^{2} \in \cE_{n} \right)
- \p_{0} \left( G_{n}^{1} \in \cE_{n}, G_{n}^{2} \in \cE_{n} \right) \\
\begin{aligned}
&= \sum_{H\, :\, \left| H \right| = \tcorr} 
\left( \p_{H} \left( G_{n} \in \cE_{n} \right) \right)^{2} \p_{S} \left( G_{\tcorr} = H \right)
- \left( \sum_{H\, :\, \left| H \right| = \tcorr} 
\p_{H} \left( G_{n} \in \cE_{n} \right) \p_{S} \left( G_{\tcorr} = H \right) \right)^{2} \\
&= \sum_{H\, :\, \left| H \right| = \tcorr}
\p_{S} \left( G_{\tcorr} = H \right) \left( \p_{H} \left( G_{n} \in \cE_{n} \right) - \mu \right)^{2}.
\end{aligned}
\end{multline*}
Note that all terms in this sum are nonnegative. 
Dropping all terms except those corresponding to $G$ and $G'$, we have that 
\begin{multline*}
\p_{\tcorr} \left( G_{n}^{1} \in \cE_{n}, G_{n}^{2} \in \cE_{n} \right)
- \p_{0} \left( G_{n}^{1} \in \cE_{n}, G_{n}^{2} \in \cE_{n} \right) \\
\geq 
\p_{S} \left( G_{\tcorr} = G \right) \left( \p_{G} \left( G_{n} \in \cE_{n} \right) - \mu \right)^{2} 
+ 
\p_{S} \left( G_{\tcorr} = G' \right) \left( \p_{G'} \left( G_{n} \in \cE_{n} \right) - \mu \right)^{2}.
\end{multline*}
By the condition that $G, G' \in \Range \left( \cG, S \right)$,  
we have that 
$\p_{S} \left( G_{\tcorr} = G \right)$ 
and 
$\p_{S} \left( G_{\tcorr} = G' \right)$ 
are both strictly positive, and note that these are not a function of $n$. 
By~\eqref{eq:TV_diff_delta} it follows that at least one of 
$\p_{G} \left( G_{n} \in \cE_{n} \right)$ and $\p_{G'} \left( G_{n} \in \cE_{n} \right)$ 
must be outside of the interval $\left( \mu - \delta / 2, \mu + \delta / 2 \right)$, 
showing that 
\[
\left( \p_{G} \left( G_{n} \in \cE_{n} \right) - \mu \right)^{2} 
+ 
\left( \p_{G'} \left( G_{n} \in \cE_{n} \right) - \mu \right)^{2} 
\geq \delta^{2} / 4.
\]
Putting everything together, we have shown that 
\[
\p_{\tcorr} \left( G_{n}^{1} \in \cE_{n}, G_{n}^{2} \in \cE_{n} \right)
- \p_{0} \left( G_{n}^{1} \in \cE_{n}, G_{n}^{2} \in \cE_{n} \right) 
\geq \frac{\delta^{2}}{4} \min \left\{ \p_{S} \left( G_{\tcorr} = G \right), \p_{S} \left( G_{\tcorr} = G' \right) \right\} 
\]
for every $n \geq \tcorr$, 
which implies that 
\[
\lim_{n \to \infty} \TV \left( \cCG \left( n, \tcorr, S \right), \cG \left( n, S \right)^{\otimes 2} \right) 
\geq 
\frac{\delta^{2}}{4} \min \left\{ \p_{S} \left( G_{\tcorr} = G \right), \p_{S} \left( G_{\tcorr} = G' \right) \right\} 
> 0. \qedhere
\]
\end{proof}

%%%%%%%%%%%%%%%%%%%%%%%%%%%%%%%%%%%%%%%%%%%%%%%%%%
\section{Detecting correlation explicitly} \label{sec:detection_specific} %%%
%%%%%%%%%%%%%%%%%%%%%%%%%%%%%%%%%%%%%%%%%%%%%%%%%%

In this section we give alternative proofs to Corollaries~\ref{thm:detection_PA} and~\ref{thm:detection_UA} that are algorithmic: they explicitly specify (efficiently computable) statistics that detect correlation in PA and UA trees. 
We first prove Corollary~\ref{thm:detection_PA} in Section~\ref{sec:detection_PA} and then turn to proving Corollary~\ref{thm:detection_UA} in Section~\ref{sec:detection_UA}. 

%%%%%%%%%%%%%%%%%%%%%%%%%%%%%%%%%%%%%%%%%%%%%%%%%%
\subsection{Detecting correlation in correlated PA trees} \label{sec:detection_PA} %%%
%%%%%%%%%%%%%%%%%%%%%%%%%%%%%%%%%%%%%%%%%%%%%%%%%%

Inspired by~\cite{BMR15}, 
we will prove Corollary~\ref{thm:detection_PA} by studying the maximum degrees in the two trees. 
For a tree~$T$, 
let $d_{T} (v)$ denote the degree of vertex $v$ in $T$, 
and let $\Delta(T)$ denote the maximum degree in~$T$. 
We will show that the pair of maximum degrees 
$\left( \Delta \left( T_{1} \right), \Delta \left( T_{2} \right) \right)$ 
has a different distribution 
under $\left( T_{1}, T_{2} \right) \sim \CPA \left( n, \tcorr, S \right)$ 
than 
under $\left( T_{1}, T_{2} \right) \sim \PA \left( n, S \right)^{\otimes 2}$, 
even in the limit as $n \to \infty$. 

Our starting point is the following lemma from~\cite{BMR15}, 
which determines how the tail behavior\footnote{Throughout the paper we use standard asymptotic notation, for instance, $f(t) \sim g(t)$ as $t\to \infty$ if $\lim_{t\to\infty} f(t)/g(t) = 1$.} of the maximum degree in a PA tree depends on the initial seed $S$, in the limit as $n \to \infty$. 
\begin{lemma}[\cite{BMR15}]\label{lem:tail}
Let $S$ be a finite tree with at least two vertices. 
Define the quantity $m (S):= \left| \left\{ v \in V(S) : d_{S} (v) = \Delta \left(S \right) \right\} \right|$. 
Then 
\[
\lim_{n \to \infty} \p \left( \frac{\Delta \left( \PA \left( n, S \right) \right)}{\sqrt{n}}  > u \right) 
\sim 
m (S) c \left( \left| S \right|, \Delta \left( S \right) \right)
u^{1 - 2\left|S\right| + 2\Delta \left(S \right)} 
\exp \left( - u^{2} / 4 \right)
\]
as $u \to \infty$, 
where the constant $c$ is defined as 
\[
c(a,b) := \frac{\Gamma \left( 2a - 2 \right)}{2^{b-1} \Gamma \left( a - 1/2 \right) \Gamma \left( b \right)}.
\]
\end{lemma}

We first prove Corollary~\ref{thm:detection_PA} in the special case when 
the seed tree is $S = S_{2}$, the unique tree on two vertices. 
This is to simplify exposition and so that the main ideas are clear; we then later show what needs to be changed for a general seed tree $S$. 
To abbreviate notation, 
in the following we denote by $\p_{0}$ the underlying probability measure when 
$\left( T_{1}(n), T_{2}(n) \right) \sim \PA \left( n, S \right)^{\otimes 2}$ 
and by~$\p_{\tcorr}$ the underlying probability measure when 
$\left( T_{1}(n), T_{2}(n) \right) \sim \CPA \left( n, \tcorr, S \right)$.

\begin{proof}[Proof of Corollary~\ref{thm:detection_PA} when $S = S_{2}$.] 
Given two trees on $n$ vertices, $T_{1}(n)$ and $T_{2}(n)$, define the event
\begin{equation}\label{eq:Aun_def}
A_{u,n} := \left\{ \frac{\Delta \left( T_{1} \left( n \right) \right)}{\sqrt{n}} > u, \frac{\Delta \left( T_{2} \left( n \right) \right)}{\sqrt{n}} > u \right\}.
\end{equation}
Under $\p_{0}$, the trees $T_{1}(n)$ and $T_{2}(n)$ are independent and identically distributed, 
so the probability of this event factorizes: 
\begin{equation}\label{eq:p0_prob_factor}
\p_{0} \left( A_{u,n} \right) 
= \p_{0} \left( \frac{\Delta \left( T_{1} \left( n \right) \right)}{\sqrt{n}} > u \right) \p_{0} \left( \frac{\Delta \left( T_{2} \left( n \right) \right)}{\sqrt{n}} > u \right) 
= \left( \p_{0} \left( \frac{\Delta \left( T_{1} \left( n \right) \right)}{\sqrt{n}} > u \right) \right)^{2}.
\end{equation}
Now taking the limit as $n \to \infty$ and using Lemma~\ref{lem:tail} (together with the facts that $m(S_{2}) = 2$ and $c(2,1) = 2 / \sqrt{\pi}$), we obtain that 
\begin{equation}\label{eq:probAun_asymp_null}
\lim_{n \to \infty} \p_{0} \left( A_{u,n} \right) 
\sim 
\frac{16}{\pi} u^{-2} \exp \left( - u^{2} / 2 \right)
\end{equation}
as $u \to \infty$. 

Next, our goal is to understand the probability of $A_{u,n}$ under $\p_{\tcorr}$. 
For a tree $T$ on $\tcorr$ vertices, define the event 
\[
\cE \left( T \right) := \left\{ T_{1} \left( \tcorr \right) = T_{2} \left( \tcorr \right) = T \right\}.
\]
Observe that if 
$\left( T_{1}(n), T_{2}(n) \right) \sim \CPA \left( n, \tcorr, S \right)$ 
and $n > \tcorr$, 
then $T_{1}(n)$ and $T_{2}(n)$ are conditionally i.i.d.\ given the event $\cE \left( T \right)$; more specifically, they are both distributed according to $\PA \left( n, T \right)$. 
Since 
$\left( T_{1}(n), T_{2}(n) \right) \sim \CPA \left( n, \tcorr, S \right)$ 
implies that $T_{1} \left( \tcorr \right) = T_{2} \left( \tcorr \right)$, 
we can condition on the tree obtained at time $\tcorr$ in order to compute the probability 
$\p_{\tcorr} \left( A_{u,n} \right)$: 
\[
\p_{\tcorr} \left( A_{u,n} \right) 
= \sum_{T} \p_{\tcorr} \left( A_{u,n} \, \middle| \, \cE \left( T \right) \right) \p_{\tcorr} \left( \cE \left( T \right) \right)
= \sum_{T} \left( \p \left( \frac{\Delta \left( \PA \left( n, T \right) \right)}{\sqrt{n}} > u \right) \right)^{2} \p_{\tcorr} \left( \cE \left( T \right) \right),
\]
where the sum is over all trees $T$ on $\tcorr$ vertices. 
Taking the limit as $n \to \infty$ we obtain that 
\[
\lim_{n \to \infty} \p_{\tcorr} \left( A_{u,n} \right) 
= 
\sum_{T} \left( \lim_{n \to \infty} \p \left( \frac{\Delta \left( \PA \left( n, T \right) \right)}{\sqrt{n}} > u \right) \right)^{2} \p_{\tcorr} \left( \cE \left( T \right) \right).
\]
We are interested in the asymptotics of this expression as $u \to \infty$,   
which we can read off of Lemma~\ref{lem:tail}. 
Using the fact that every tree in the sum has $\tcorr$ vertices, we obtain that 
\begin{equation}\label{eq:sumT'_asymp}
\lim_{n \to \infty} \p_{\tcorr} \left( A_{u,n} \right) 
\sim 
\sum_{T} 
\left\{ m\left( T \right) c \left( \tcorr , \Delta \left( T \right) \right) \right\}^{2} 
\p_{\tcorr} \left( \cE \left( T \right) \right) 
u^{2-4\tcorr + 4 \Delta \left( T \right)}
\exp \left( - u^{2} / 2 \right)
\end{equation}
as $u \to \infty$. 
%where the sum is still over all trees $T'$ on $\tcorr$ vertices. 
Note that the $\exp( - u^{2} / 2)$ factor is common to all terms in the sum, 
but the polynomial factor in $u$ differs across the terms. 
When $T = S_{\tcorr}$, 
the star on $\tcorr$ vertices, 
we have that $\Delta \left( T \right) = \tcorr - 1$ 
and so the polynomial factor in $u$ is $u^{-2}$. 
Whenever $T \neq S_{\tcorr}$, 
we have that $\Delta \left( T  \right) \leq \tcorr - 2$ 
and so the polynomial factor in $u$ is $O(u^{-6})$ as $u \to \infty$. 
Therefore the terms corresponding to trees $T$ that are not a star 
are lower order (asymptotically as $u \to \infty$) compared to the term corresponding to $T = S_{\tcorr}$. 
In other words, the sum in~\eqref{eq:sumT'_asymp} is asymptotically equivalent to the term corresponding to $T = S_{\tcorr}$: 
\begin{equation}\label{eq:maxdegtail_tcorr}
\lim_{n \to \infty} \p_{\tcorr} \left( A_{u,n} \right) 
\sim 
\left\{ m\left( S_{\tcorr} \right) c \left( \tcorr , \tcorr - 1 \right) \right\}^{2} 
\p_{\tcorr} \left( \cE \left( S_{\tcorr} \right) \right) 
u^{-2} 
\exp \left( - u^{2} / 2 \right)
\end{equation}
as $u \to \infty$. 
Observe that 
$m \left( S_{\tcorr} \right) = 1$ 
whenever $\tcorr > 2$. 
From the fact that $\Gamma(z+1) = z \Gamma (z)$ 
it follows that 
$c(t+1,t) = 2 c(t,t-1)$. 
Thus we have that 
$c(\tcorr, \tcorr - 1) = 2^{\tcorr - 2} c(2,1) 
= 2^{\tcorr - 1} / \sqrt{\pi}$. 
Finally, turning to the probability 
$\p_{\tcorr} \left( \cE \left( S_{\tcorr} \right) \right)$, 
note that under $\p_{\tcorr}$ we have that 
$T_{1}(n) = T_{2}(n)$ for all $n \leq \tcorr$ and thus 
$\p_{\tcorr} \left( \cE \left( S_{\tcorr} \right) \right) 
= \p \left( \PA \left( \tcorr, S \right) = S_{\tcorr} \right)$. 
Since the star $S_{3}$ is the unique tree on $3$ vertices, 
we have that 
$\p_{\tcorr} \left( \cE \left( S_{\tcorr} \right) \right) = 1$ 
when $\tcorr = 3$. 
When $\tcorr > 3$, 
the only way that 
we can have $\PA \left( \tcorr, S_{3} \right) = S_{\tcorr}$ 
is if all vertices from time $4$ through $\tcorr$ attach to the center of the star in~$S_{3}$. Since at each of the $\tcorr - 3$ time steps the degree of the center of the star is equal to half of the sum of the degrees in the tree, 
this has probability $2^{-(\tcorr - 3)}$. 
Putting everything together we have thus computed the constant factor in~\eqref{eq:maxdegtail_tcorr} and obtained that 
\begin{equation}\label{eq:probAun_asymp_tcorr}
\lim_{n \to \infty} \p_{\tcorr} \left( A_{u,n} \right) 
\sim 
\frac{2^{\tcorr+1}}{\pi}
u^{-2} 
\exp \left( - u^{2} / 2 \right)
\end{equation}
as $u \to \infty$, for every $\tcorr \geq 3$. 
In particular, comparing the expressions in~\eqref{eq:probAun_asymp_null} and~\eqref{eq:probAun_asymp_tcorr}, 
we have that
\[
\lim_{n \to \infty}
\left\{ \p_{\tcorr} \left( A_{u,n} \right) 
- \p_{0} \left( A_{u,n} \right) \right\}
\sim 
\left( 2^{\tcorr - 3} - 1 \right)
\frac{16}{\pi}
u^{-2} 
\exp \left( - u^{2} / 2 \right)
\]
as $u \to \infty$. 
When $\tcorr > 3$ this quantity is positive for every $u > 0$, which concludes the proof. 
\end{proof}

\begin{proof}[Proof of Corollary~\ref{thm:detection_PA} for a general seed tree $S$.]
We assume in the following that $S \neq S_{2}$ and thus $\left| S \right| \geq 3$. Therefore the assumption that $\tcorr > \left| S \right|$ implies that $\tcorr > 3$. 
We again consider the event $A_{u,n}$ defined in~\eqref{eq:Aun_def}. 
The identity in~\eqref{eq:p0_prob_factor} holds again, 
and thus taking the limit as $n \to \infty$ and using Lemma~\ref{lem:tail} we obtain that 
\begin{equation}\label{eq:prob_tail_null}
\lim_{n\to\infty} \p_{0} \left( A_{u,n} \right) 
\sim 
\left\{ m (S) c \left( \left| S \right|, \Delta \left( S \right) \right) \right\}^{2}
u^{2 - 4\left|S\right| + 4\Delta \left(S \right)} 
\exp \left( - u^{2} / 2 \right)
\end{equation}
as $u \to \infty$. 

Next, our goal is to understand the probability of $A_{u,n}$ under $\p_{\tcorr}$. By the same arguments as in the case $S = S_{2}$, we have that~\eqref{eq:sumT'_asymp} holds. However, the subsequent analysis of this expression is different for general $S$. 

First, note that under $\p_{\tcorr}$ we have that $T_{1}(n) = T_{2} (n)$ for all $n \leq \tcorr$ and so 
$\p_{\tcorr} \left( \cE \left( T \right) \right) = \p \left( \PA \left( \tcorr, S \right) = T \right)$. 
Thus the sum in~\eqref{eq:sumT'_asymp} is only over trees $T$ for which this probability is positive (for general $S$, this is not every tree on $\tcorr$ vertices). 
Next, note that if 
$\p \left( \PA \left( \tcorr, S \right) = T \right) > 0$, 
then $\Delta \left( T \right) \leq \Delta \left( S \right) + \tcorr - \left| S \right|$, 
since the maximum degree can only increase by $1$ at each time step. 
Therefore 
$2 - 4\tcorr + 4 \Delta \left( T \right)
\leq 2 - 4 \left| S \right| + 4 \Delta \left( S \right)$ 
for every such tree $T$ 
and it follows that 
\[
\lim_{n \to \infty} \p_{\tcorr} \left( A_{u,n} \right) 
= O \left( u^{2 - 4 \left| S \right| + 4 \Delta \left( S \right)} \exp \left( - u^{2} / 2 \right) \right)
\]
as $u \to \infty$. 
This implies that the only terms that contribute to the sum in~\eqref{eq:sumT'_asymp} (asymptotically as $u \to \infty$) 
correspond to trees $T$ such that 
$\p \left( \PA \left( \tcorr, S \right) = T \right) > 0$ 
and 
$\Delta \left( T \right) = \Delta \left( S \right) + \tcorr - \left| S \right|$; 
the other terms are lower order (asymptotically as $u \to \infty$). 
In other words, we have shown that 
\begin{multline}\label{eq:sumT'_selected}
\lim_{n \to \infty} \p_{\tcorr} \left( A_{u,n} \right) \\
\sim 
\sum_{\substack{T\, :\, \p \left( \PA \left( \tcorr, S \right) = T \right) > 0, \\ \Delta \left( T \right) = \Delta \left( S \right) + \tcorr - \left| S \right|}}
\left\{ m\left( T \right) c \left( \tcorr , \Delta \left( T \right) \right) \right\}^{2} 
\p \left( \PA \left( \tcorr, S \right) = T \right) 
u^{2 - 4 \left| S \right| + 4 \Delta \left( S \right)} 
e^{-u^{2}/2}
\end{multline}
as $u \to \infty$. 
This expression is on the order of 
$u^{2 - 4 \left| S \right| + 4 \Delta \left( S \right)} 
\exp \left( - u^{2} / 2 \right)$, 
so what remains is to determine the constant.

First, from the definition of $c(a,b)$ and the fact that $\Gamma \left( z + 1 \right) = z \Gamma \left( z \right)$,  
we have that 
\[
c(t+1, s+1) = \frac{2(t-1)}{s} c(t,s).
\]
Iterating this expression we obtain that 
\begin{equation}\label{eq:c_expression}
c \left( \tcorr, \Delta \left( S \right) + \tcorr - \left| S \right| \right) 
= 2^{\tcorr - \left| S \right|} 
\left( \prod_{i=0}^{\tcorr - \left|S \right| - 1} \frac{\left| S \right| - 1 + i}{\Delta \left( S \right) + i} \right) 
c \left( \left| S \right|, \Delta \left( S \right) \right).
\end{equation}
Turning now to the other quantities in~\eqref{eq:sumT'_selected}, we have to understand for what trees $T$ do we have $\Delta \left( T \right) = \Delta \left( S \right) + \tcorr - \left| S \right|$. For this to happen, we must have that the maximum degree increases at every time step of the process, from time $\left| S \right|$ to time $\tcorr$. This happens if and only if at every time step of the process the incoming vertex attaches to a vertex of maximum degree. Initially, at time $\left|S \right|$, there are $m(S)$ vertices with degree equal to the maximum degree $\Delta(S)$, and the sum of the degrees is $2 \left( \left| S \right| - 1 \right)$. 
Therefore the probability that the maximum degree increases in the next time step is 
\[
\frac{m(S) \Delta(S)}{2 \left( \left| S \right| - 1 \right)}.
\]
Thereafter there is only a single vertex with maximum degree, whose degree is now $\Delta \left( S \right) + 1$, while the sum of the degrees is now $2 \left| S \right|$. Thus the probability that the maximum degree increases in the next time step is 
$\left( \Delta \left( S \right) + 1 \right)/ \left( 2 \left| S \right| \right)$. 
Continuing this argument recursively we
obtain that 
$m(T) = 1$ for any tree $T$ such that 
$\p \left( \PA \left( \tcorr, S \right) = T \right) > 0$ 
and 
$\Delta \left( T \right) = \Delta \left( S \right) + \tcorr - \left| S \right|$, 
and also that 
\begin{align}
\sum_{\substack{T\, :\, \p \left( \PA \left( \tcorr, S \right) = T \right) > 0, \\ \Delta \left( T \right) = \Delta \left( S \right) + \tcorr - \left| S \right|}}
\p \left( \PA \left( \tcorr, S \right) = T \right) 
&=
\frac{m(S) \Delta(S)}{2 \left( \left| S \right| - 1 \right)}
\prod_{i=1}^{\tcorr - \left|S\right|-1} \frac{\Delta \left( S \right) + i}{2 \left( \left|S \right| - 1 + i \right)} \notag \\
&= 
\frac{m(S)}{2^{\tcorr - \left|S\right|}} 
\prod_{i=0}^{\tcorr - \left|S\right|-1} \frac{\Delta \left( S \right) + i}{\left|S \right| - 1 + i}. \label{eq:sum_probs_T'}
\end{align}
Thus putting together~\eqref{eq:c_expression} and~\eqref{eq:sum_probs_T'} 
we obtain that 
\begin{multline*}
\sum_{\substack{T\, :\, \p \left( \PA \left( \tcorr, S \right) = T \right) > 0, \\ \Delta \left( T \right) = \Delta \left( S \right) + \tcorr - \left| S \right|}}
\left\{ m\left( T \right) c \left( \tcorr , \Delta \left( T \right) \right) \right\}^{2} 
\p \left( \PA \left( \tcorr, S \right) = T \right) \\
= 
2^{\tcorr - \left| S \right|} 
\left( \prod_{i=0}^{\tcorr - \left|S \right| - 1} \frac{\left| S \right| - 1 + i}{\Delta \left( S \right) + i} \right) 
m(S) 
\left( c \left( \left| S \right|, \Delta \left( S \right) \right) \right)^{2}.
\end{multline*}
Putting this expression back into~\eqref{eq:sumT'_selected} 
we thus have that 
\[
\lim_{n \to \infty} \p_{\tcorr} \left( A_{u,n} \right) 
\sim 
2^{\tcorr - \left| S \right|} 
\left( \prod_{i=0}^{\tcorr - \left|S \right| - 1} \frac{\left| S \right| - 1 + i}{\Delta \left( S \right) + i} \right) 
m(S) 
\left(c \left( \left| S \right|, \Delta \left( S \right) \right) \right)^{2}
u^{2 - 4 \left| S \right| + 4 \Delta \left( S \right)} 
\exp \left( - u^{2} / 2 \right)
\]
as $u \to \infty$. 
Comparing this expression with~\eqref{eq:prob_tail_null}, 
we obtain that 
\begin{multline}\label{eq:tail_prob_tcorr_final}
\lim_{n \to \infty}
\left\{ \p_{\tcorr} \left( A_{u,n} \right) 
- \p_{0} \left( A_{u,n} \right) \right\} \\
\sim 
\left\{ 
2^{\tcorr - \left| S \right|} 
\prod_{i=0}^{\tcorr - \left|S \right| - 1} \frac{\left| S \right| - 1 + i}{\Delta \left( S \right) + i} 
- m(S) 
\right\}
m(S) 
\left( c \left( \left| S \right|, \Delta \left( S \right) \right) \right)^{2}
u^{2 - 4 \left| S \right| + 4 \Delta \left( S \right)} 
\exp \left( - u^{2} / 2 \right)
\end{multline}
as $u \to \infty$. 
To conclude the proof what remains to be shown is that the expression in the curly brackets above is strictly positive. 
To see this, first note that $\Delta(S) \leq \left|S\right| - 1$, 
so all the fractions in the product are at least $1$. 
Dropping all but the first fraction (corresponding to $i=0$), 
and using that $\tcorr > \left|S\right|$, 
we have that 
\begin{equation}\label{eq:drop_factors}
2^{\tcorr - \left| S \right|} 
\prod_{i=0}^{\tcorr - \left|S \right| - 1} \frac{\left| S \right| - 1 + i}{\Delta \left( S \right) + i} 
\geq 
\frac{2 \left( \left|S \right| - 1 \right)}{\Delta(S)}.
\end{equation}
Note that 
$2 \left( \left|S \right| - 1 \right)$ is equal to the sum of the degrees of vertices in $S$, 
while $m(S) \Delta(S)$ is the equal to the sum of the degrees of vertices in $S$ whose degree is equal to the maximum degree $\Delta(S)$. 
Since $\left| S \right| \geq 3$, we know that not every vertex has degree equal to the maximum degree (since there are leaves and also $\Delta(S) > 1$). 
Therefore we must have that 
$m(S) \Delta(S) < \sum_{v \in V(S)} d_{S} (v) = 2 \left( \left|S \right| - 1 \right)$. 
This, combined with~\eqref{eq:drop_factors}, shows that 
the bracketed expression in~\eqref{eq:tail_prob_tcorr_final} is positive. 
\end{proof}

%%%%%%%%%%%%%%%%%%%%%%%%%%%%%%%%%%%%%%%%%%%%%%%%%%
\subsection{Detecting correlation in correlated UA trees} \label{sec:detection_UA} %%%
%%%%%%%%%%%%%%%%%%%%%%%%%%%%%%%%%%%%%%%%%%%%%%%%%%

Inspired by~\cite{BEMR17}, we prove Corollary~\ref{thm:detection_UA} by considering a statistic that measures global balancedness properties of a tree. 
For a tree $T$ and an edge $e \in E(T)$, 
let $T'$ and $T''$ be the two connected components of $T \setminus \left\{ e \right\}$. 
Define 
\[
h \left( T, e \right) 
:= \frac{\left| T' \right|^{2} \left| T'' \right|^{2}}{\left| T \right|^{4}}
\]
and also
\[
H \left( T \right) := \sum_{e \in E(T)} h(T,e).
\]
We have that $0 \leq h(T,e) \leq 1/16$, 
and for ``peripheral'' edges $e$, the quantity $h(T,e)$ is closer to $0$, 
while for more ``central'' edges $e$, the quantity $h(T,e)$ is closer to $1/16$. 
The statistic $H(T)$ thus measures 
the global balancedness properties of the tree $T$ in a particular way, 
and ``central'' edges contribute the most to this statistic. 
This statistic was used in~\cite{BEMR17} to show that 
uniform attachment started from the seed $P_{4}$ (the path on four vertices) 
is different from uniform attachment started from the seed $S_{4}$ (the star on four vertices); 
formally, $\lim_{n \to \infty} \TV \left( \UA \left( n, P_{4} \right), \UA \left( n, S_{4} \right) \right) > 0$. 

We prove Corollary~\ref{thm:detection_UA} by showing that 
the pair 
$\left( H \left( T_{1} \right), H \left( T_{2} \right) \right)$ 
has a different distribution under 
$\left( T_{1}, T_{2} \right) \sim \CUA \left( n, \tcorr, S \right)$ 
than under 
$\left( T_{1}, T_{2} \right) \sim \UA \left( n, S \right)^{\otimes 2}$, 
even in the limit as $n \to \infty$. 
In fact, 
we do this  by showing that the product 
$H\left( T_{1} \right) H \left( T_{2} \right)$ 
has a different distribution in the two settings. 
To abbreviate notation, 
in the following we will denote by $\p_{0}$ the underlying probability measure when 
$\left( T_{1}(n), T_{2}(n) \right) \sim \UA \left( n, S \right)^{\otimes 2}$ 
and by~$\p_{\tcorr}$ the underlying probability measure when 
$\left( T_{1}(n), T_{2}(n) \right) \sim \CUA \left( n, \tcorr, S \right)$. 
Likewise, $\E_{0}$, $\E_{\tcorr}$, $\Var_{0}$, and $\Var_{\tcorr}$ refer to expectations and variances under these measures. 

To simplify exposition and to highlight the main ideas, we first prove Corollary~\ref{thm:detection_UA} in the special case when $S = S_{1}$ and $\tcorr = 4$; we then later show what needs to be changed in the general setting.

\begin{proof}[Proof of Corollary~\ref{thm:detection_UA} when $S = S_{1}$ and $\tcorr = 4$] 
We start by defining two random variables, in order to abbreviate notation. 
Define 
\begin{align*}
X_{n} &:= H \left( T_{1} \left( n \right) \right) H \left( T_{2} \left( n \right) \right), 
\qquad \text{ where } \left( T_{1}(n), T_{2}(n) \right) \sim \UA \left( n, S_{1} \right)^{\otimes 2}, \\
Y_{n} &:= H \left( T_{1} \left( n \right) \right) H \left( T_{2} \left( n \right) \right), 
\qquad \text{ where } \left( T_{1}(n), T_{2}(n) \right) \sim \CUA \left( n, 4, S_{1} \right).
\end{align*}
By the Cauchy-Schwarz inequality (see~\cite[Section~3.2]{BEMR17} for details), we have that 
\[
\TV \left( \CUA \left( n, 4, S_{1} \right), \UA \left( n, S_{1} \right)^{\otimes 2} \right)
\geq 
\TV \left( X_{n}, Y_{n} \right) 
\geq 
\frac{\left( \E \left[ X_{n} \right] - \E \left[ Y_{n} \right] \right)^{2}}{2 \Var \left( X_{n} \right) + 2 \Var \left( Y_{n} \right) + \left( \E \left[ X_{n} \right] - \E \left[ Y_{n} \right] \right)^{2}}.
\]
Thus in order to prove the claim, 
it suffices to show the following two things: 
\begin{equation}\label{eq:expectations_differ_UA}
\liminf_{n \to \infty} \left| \E \left[ X_{n} \right] - \E \left[ Y_{n} \right] \right| > 0
\end{equation} 
and 
\begin{equation}\label{eq:variances_bounded_UA}
\limsup_{n \to \infty} \left\{ \Var \left( X_{n} \right) + \Var \left( Y_{n} \right) \right\} < \infty.
\end{equation}

In order to understand the expectations in~\eqref{eq:expectations_differ_UA}, we first understand the evolution of the pair of trees 
$\left( T_{1}(n), T_{2}(n) \right)$ 
under $\p_{0}$ and under $\p_{4}$. 
First, for $n \leq 3$ we have that $T_{1} (n) = T_{2} (n) = S_{n}$ under both models, since $S_{n}$ is the only tree on $n$ vertices for $n \leq 3$. 
There are two trees on four vertices: the path $P_{4}$ and the star $S_{4}$.  
We know that $\p \left( \UA \left( 4, S_{1} \right) = P_{4} \right) = 2/3$ 
and that $\p \left( \UA \left( 4, S_{1} \right) = S_{4} \right) = 1/3$. 
Thus under $\p_{0}$ the two trees $T_{1}(4)$ and $T_{2}(4)$ are i.i.d.\ with this marginal distribution. 
Under $\p_{4}$ we have that $T_{1}(4) = T_{2}(4)$, so 
$\p_{4} \left( T_{1}(4) = T_{2}(4) = P_{4} \right) = 2/3$ 
and 
$\p_{4} \left( T_{1}(4) = T_{2}(4) = S_{4} \right) = 1/3$. 
Conditioned on $\left( T_{1}(4), T_{2}(4) \right)$, 
the evolution of the two trees $\left\{ T_{1}(n) \right\}_{n \geq 4}$ 
and $\left\{ T_{2}(n) \right\}_{n \geq 4}$ 
is independent (and according to uniform attachment) 
under both $\p_{0}$ and $\p_{4}$. 

We are now ready to compute the expectations $\E \left[ X_{n} \right]$ and $\E \left[ Y_{n} \right]$. 
To abbreviate notation, we introduce two quantities for $n \geq 4$: 
$m_{P,n} := \E \left[ H \left( \UA \left( n, P_{4} \right) \right) \right]$ 
and 
$m_{S,n} := \E \left[ H \left( \UA \left( n, S_{4} \right) \right) \right]$. 
By the previous paragraph we thus have that 
\[
\E \left[ H \left( \UA \left( n, S_{1} \right) \right) \right] 
= \frac{2}{3} m_{P,n} + \frac{1}{3} m_{S,n}.
\]
Under $\p_{0}$ we have that $T_{1}(n)$ and $T_{2}(n)$ are i.i.d., so 
\[
\E \left[ X_{n} \right] 
= \E_{0} \left[ H \left( T_{1} \left( n \right) \right) \right] \E_{0} \left[ H \left( T_{2} \left( n \right) \right) \right] 
= \left( \frac{2}{3} m_{P,n} + \frac{1}{3} m_{S,n} \right)^{2} 
= \frac{4}{9} m_{P,n}^{2} + \frac{1}{9} m_{S,n}^{2} 
+ \frac{4}{9} m_{P,n} m_{S,n}.
\]
To compute $\E \left[ Y_{n} \right]$ we may condition on the value of $T_{1}(4) = T_{2}(4)$: 
\begin{align*}
\E \left[ Y_{n} \right] 
&= 
\frac{2}{3} \E_{4} \left[ H \left( T_{1} \left( n \right) \right) H \left( T_{2} \left( n \right) \right) \, \middle| \, T_{1}(4) = P_{4} \right] 
+ 
\frac{1}{3} \E_{4} \left[ H \left( T_{1} \left( n \right) \right) H \left( T_{2} \left( n \right) \right) \, \middle| \, T_{1}(4) = S_{4} \right] \\
&= \frac{2}{3} m_{P,n}^{2} + \frac{1}{3} m_{S,n}^{2}.
\end{align*}
Computing the difference of the previous two displays we obtain that
\[
\E \left[ X_{n} \right] - \E \left[ Y_{n} \right] 
= - \frac{2}{9} \left( m_{P,n} - m_{S,n} \right)^{2}.
\]
In~\cite[Section~2.2]{BEMR17} it was shown that $\lim_{n \to \infty} \left( m_{P,n} - m_{S,n} \right) = 1/70$.   
This implies that 
\[
\lim_{n \to \infty} \left| \E \left[ X_{n} \right] - \E \left[ Y_{n} \right]  \right| 
= \frac{1}{22050} > 0,
\]
which establishes~\eqref{eq:expectations_differ_UA}.

We now turn to bounding the variances. Under $\p_{0}$ we have that $T_{1}(n)$ and $T_{2}(n)$ are i.i.d., so 
\[
\Var \left( X_{n} \right) 
= \Var_{0} \left( H \left( T_{1} \left( n \right) \right) H \left( T_{2} \left( n \right) \right) \right) 
= \Var_{0} \left( H \left( T_{1} \left( n \right) \right) \right)^{2} 
+ 2 \E_{0} \left[ H \left( T_{1} \left( n \right) \right) \right]^{2} \Var_{0} \left( H \left( T_{1} \left( n \right) \right) \right).
\]
The analysis in~\cite[Section~2.2]{BEMR17} shows that 
\[
\limsup_{n\to\infty} \Var_{0} \left( H \left( T_{1} \left( n \right) \right) \right) < \infty 
\qquad 
\text{ and }
\qquad 
\limsup_{n\to\infty} \E_{0} \left[ H \left( T_{1} \left( n \right) \right) \right] < \infty,
\]
which thus implies that $\limsup_{n \to \infty} \Var \left( X_{n} \right) < \infty$. 
The analysis of $\Var \left( Y_{n} \right)$ is similar, by conditioning on the tree at time $\tcorr = 4$; we leave the details to the reader. 
\end{proof}

We now show what changes in the proof when $S = S_{1}$ and $\tcorr$ is arbitrary. 

\begin{proof}[Proof of Corollary~\ref{thm:detection_UA} when $S = S_{1}$ and $\tcorr > 4$.]
Define $X_{n}$ as before and $Y_{n}$ analogously (with $4$ replaced by $\tcorr$). Again we have to show that~\eqref{eq:expectations_differ_UA} and~\eqref{eq:variances_bounded_UA} hold. 
The method for showing~\eqref{eq:variances_bounded_UA} (i.e., for bounding the variances) is unchanged; we explain here what changes in showing~\eqref{eq:expectations_differ_UA}. 

For a tree $T$ on $\tcorr$ vertices, let 
$p_{T} := \p \left( \UA \left( \tcorr, S_{1} \right) = T \right)$ 
and let 
$m_{T,n} := \E \left[ H \left( \UA \left( n, T \right) \right) \right]$. 
Note that $p_{T} > 0$ for every tree $T$ on $\tcorr$ vertices. 
By the same arguments as before we have that 
\[
\E \left[ X_{n} \right]
= \left( \sum_{T} p_{T} m_{T,n} \right)^{2}, 
\qquad \qquad 
\E \left[ Y_{n} \right] 
= \sum_{T} p_{T} m_{T,n}^{2} 
\] 
for every $n \geq \tcorr$, 
where in both sums $T$ ranges over all trees on $\tcorr$ vertices. 
Thus by Cauchy-Schwarz it follows that $\E \left[ X_{n} \right] \leq \E \left[ Y_{n} \right]$ for every $n \geq \tcorr$. 
In order to show that~\eqref{eq:expectations_differ_UA} holds, it thus suffices to show that there exist trees $T$ and $T'$ on $\tcorr$ vertices such that 
$\lim_{n \to \infty} \left( m_{T,n} - m_{T', n} \right) \neq 0$. 

We choose $T = S_{\tcorr}$ (the star on $\tcorr$ vertices), and $T'$ to be the tree on $\tcorr$ vertices where one of the vertices has degree $\tcorr - 2$ (that is, this is the star on $\tcorr - 1$ vertices with an extra edge attached to one of the leaves). 
Computing the difference $m_{T,n} - m_{T',n}$ 
was done explicitly in~\cite{BEMR17} for the case $\tcorr = 4$ (see above for the result). To do this calculation for general $\tcorr \geq 4$, we first introduce some notation. For all $\alpha,\beta,n \in \mathbb{N}$, let $B_{\alpha,\beta,n}$ be a random variable such that $B_{\alpha, \beta, n} - \alpha$ has the beta-binomial distribution with parameters $(\alpha,\beta,n)$; that is, it is a random variable satisfying
\[
\p \left( B_{\alpha,\beta,n} = \alpha + k \right) = \frac{(k+\alpha-1)! (n-k+\beta - 1)! (\alpha + \beta - 1)!}{ (n+\alpha+\beta - 1)! (\alpha - 1)! (\beta - 1)! }  \binom{n}{k}, ~~ \forall k \in \left\{0, 1, \dots, n \right\}.
\]
The key observation is the following distributional identity:  
if $e \in E \left( S \right)$ is such that the two components of 
$S \setminus \left\{ e \right\}$ 
have size $\alpha$ and $\left| S \right| - \alpha$, then 
\[
h \left( \UA \left( n, S \right), e \right) 
\stackrel{d}{=} 
\frac{1}{n^4} B_{\alpha,\left|S\right| - \alpha,n-|S|}^{2} \left( n-B_{\alpha,\left| S \right| - \alpha,n-|S|} \right)^2.
\]
This is an immediate consequence of the characterization of $\left( B_{\alpha,\beta,n}, n + \left( \alpha + \beta \right) - B_{\alpha, \beta, n} \right)$ as the distribution of a classical P\'olya urn with replacement matrix $\left(\begin{smallmatrix} 1 & 0 \\ 0 & 1 \end{smallmatrix}\right)$ and starting state $(\alpha,\beta)$ after $n$ draws. 
It then follows (by the same arguments as in~\cite[Section~2.2]{BEMR17}) that 
\[
m_{T,n} - m_{T',n} 
= \frac{1}{n^{4}} \left( 
\E \left[ B_{1,\tcorr-1,n-\tcorr}^{2} \left( n - B_{1,\tcorr-1,n-\tcorr} \right)^{2} \right] 
- \E \left[ B_{2,\tcorr-2,n-\tcorr}^{2} \left( n - B_{2,\tcorr-2,n-\tcorr} \right)^{2} \right]
\right).
\]
What remains is a straightforward calculation using explicit formulae for the first four moments of the beta-binomial distribution, and we obtain that 
\[
m_{T,n} - m_{T',n} 
= - \frac{\left( \tcorr - 3 \right) \left( n + 1 \right) \left\{ 4 \left( \tcorr - 1 \right) n^{2} - \left( \tcorr^{2} - 15 \tcorr + 26 \right) n + \left( - \tcorr^{2} + 19 \tcorr - 30 \right) \right\}}{\tcorr \left( \tcorr + 1 \right) \left( \tcorr + 2 \right) \left( \tcorr + 3 \right) n^{3}}.
\]
Taking the limit as $n \to \infty$ we have that 
\[
\lim_{n \to \infty} \left( m_{T,n} - m_{T',n} \right) 
= - \frac{4 \left( \tcorr - 1 \right) \left( \tcorr - 3 \right)}{\tcorr \left( \tcorr + 1 \right) \left( \tcorr + 2 \right) \left( \tcorr + 3 \right)} 
\neq 0. \qedhere
\]
\end{proof}

Essentially the same proof works for any seed graph $S$ and any $\tcorr$ satisfying $\tcorr > \left| S \right|$ and $\tcorr > 3$. 
Again we have to show that there exist two trees $T$ and $T'$ on $\tcorr$ vertices such that 
$\lim_{n \to \infty} \left( m_{T,n} - m_{T',n} \right) \neq 0$ 
\emph{and also} $T, T' \in \Range \left( \UA, S \right)$. 
When $S = S_{1}$, the latter condition always holds, as $\Range \left( \UA, S \right)$ consists of all (finite) trees. 
For general $S$, the trees $T$ and $T'$ defined above may not be in $\Range \left( \UA, S \right)$. However, one can still always choose $T$ and $T'$ in $\Range \left( \UA, S \right)$ such that 
\[
m_{T,n} - m_{T',n} 
= \frac{1}{n^{4}} \left( 
\E \left[ B_{1,\tcorr-1,n-\tcorr}^{2} \left( n - B_{1,\tcorr-1,n-\tcorr} \right)^{2} \right] 
- \E \left[ B_{2,\tcorr-2,n-\tcorr}^{2} \left( n - B_{2,\tcorr-2,n-\tcorr} \right)^{2} \right]
\right)
\]
holds; we leave this as an exercise to the reader. 
One can then conclude as above. 

%%%%%%%%%%%%%%%%%%%%%%%%%%%%%%%%%%%%%%%%%%%%%%%%%%
\section{Detecting correlation with probability going to $1$ as $\tcorr \to \infty$} \label{sec:detection_large_tcorr} %%%
%%%%%%%%%%%%%%%%%%%%%%%%%%%%%%%%%%%%%%%%%%%%%%%%%%

In this section we focus on detecting correlation with probability going to $1$ as $\tcorr \to \infty$. 
We first prove Theorem~\ref{thm:high_prob_detection} in Section~\ref{sec:detection_large_tcorr_general} 
and then prove Theorems~\ref{thm:high_prob_detection_PA} and~\ref{thm:high_prob_detection_UA} in Section~\ref{sec:detection_large_tcorr_specific}.

%%%%%%%%%%%%%%%%%%%%%%%%%%%%%%%%%%%%%%%%%%%%%%%%%%
\subsection{A sufficient condition for Markov sequential attachment rules} \label{sec:detection_large_tcorr_general} %%%
%%%%%%%%%%%%%%%%%%%%%%%%%%%%%%%%%%%%%%%%%%%%%%%%%%

To abbreviate notation, 
in the following we denote by $\p_{0}$ the underlying probability measure 
when 
$\left\{ \left( G_{t}^{1}, G_{t}^{2} \right) \right\}_{t \geq \left| S \right|}$ 
are two independent sequences of randomly growing graphs with seed $S$ and attachment rule $\cG$. 
Similarly, we denote by $\p_{\tcorr}$ the underlying probability measure when the two graphs are correlated until time $\tcorr$.

\begin{proof}[Proof of Theorem~\ref{thm:high_prob_detection}]
We will show that for every $\delta > 0$ there exists $t' = t' \left( \delta \right)$ such that for every $\tcorr \geq t'$ we have that 
\begin{equation}\label{eq:delta_goal}
\lim_{n \to \infty} 
\TV \left( \cCG \left( n, \tcorr, S \right), \cG \left( n, S \right)^{\otimes 2} \right)
\geq 1 - \delta. 
\end{equation}
To this end, fix $\delta > 0$. 
Let $\left\{ \left( G_{t}^{1}, G_{t}^{2} \right) \right\}_{t \geq \left| S \right|}$ be two  sequences of randomly growing graphs with seed $S$ and attachment rule $\cG$, under either $\p_{0}$ or $\p_{\tcorr}$. 
Let $f_{\infty}^{1} := \lim_{t \to \infty} f \left( G_{t}^{1} \right)$ 
and $f_{\infty}^{2} := \lim_{t \to \infty} f \left( G_{t}^{2} \right)$; 
by our assumptions these limits exist almost surely, under both $\p_{0}$ and $\p_{\tcorr}$. 
Observe that
\[
\lim_{\eps \to 0} \p_{0} \left( \left| f_{\infty}^{1} - f_{\infty}^{2} \right| \leq \eps \right) 
= 
\p_{0} \left( f_{\infty}^{1} = f_{\infty}^{2} \right)
= 0, 
\]
the latter equality holding because $f_{\infty}^{1}$ and $f_{\infty}^{2}$ are i.i.d.\  absolutely continuous random variables under~$\p_{0}$. 
Thus fix $\eps > 0$ such that 
\begin{equation}\label{eq:eps_delta}
\p_{0} \left( \left| f_{\infty}^{1} - f_{\infty}^{2} \right| \leq \eps \right) 
\leq \delta / 2. 
\end{equation}
Turning to the measure $\p_{\tcorr}$, note that under $\p_{\tcorr}$ we have that $G_{\tcorr}^{1} = G_{\tcorr}^{2}$ almost surely, 
and hence $f \left( G_{\tcorr}^{1} \right) = f \left( G_{\tcorr}^{2} \right)$ almost surely as well. 
So by the triangle inequality we have, for any $n \geq \tcorr$, that 
\begin{align*}
\p_{\tcorr} \left( \left| f \left( G_{n}^{1} \right) - f \left( G_{n}^{2} \right) \right| > \eps \right) 
&\leq \p_{\tcorr} \left( \left| f \left( G_{n}^{1} \right) - f \left( G_{\tcorr}^{1} \right) \right| > \eps / 2 \right) + \p_{\tcorr} \left( \left| f \left( G_{n}^{2} \right) - f \left( G_{\tcorr}^{2} \right) \right| > \eps / 2 \right) \\
&= 2 \p_{\tcorr} \left( \left| f \left( G_{n}^{1} \right) - f \left( G_{\tcorr}^{1} \right) \right| > \eps / 2 \right) 
= 2 \p_{0} \left( \left| f \left( G_{n}^{1} \right) - f \left( G_{\tcorr}^{1} \right) \right| > \eps / 2 \right),
\end{align*}
where the first equality is due to symmetry and the second equality is because the marginal processes are the same under $\p_{0}$ and $\p_{\tcorr}$. 
Now we can bound from below the total variation distance in question by considering the event 
$\left\{ \left| f \left( G_{n}^{1} \right) - f \left( G_{n}^{2} \right) \right| \leq \eps \right\}$ 
under $\p_{0}$ and $\p_{\tcorr}$. For any $n \geq \tcorr$ we have that 
\begin{multline*}
\TV \left( \cCG \left( n, \tcorr, S \right), \cG \left( n, S \right)^{\otimes 2} \right) 
\geq 
\p_{\tcorr} \left( \left| f \left( G_{n}^{1} \right) - f \left( G_{n}^{2} \right) \right| \leq \eps \right) 
- \p_{0} \left( \left| f \left( G_{n}^{1} \right) - f \left( G_{n}^{2} \right) \right| \leq \eps \right) \\
\geq  
1 - 2 \p_{0} \left( \left| f \left( G_{n}^{1} \right) - f \left( G_{\tcorr}^{1} \right) \right| > \eps / 2 \right) - \p_{0} \left( \left| f \left( G_{n}^{1} \right) - f \left( G_{n}^{2} \right) \right| \leq \eps \right).
\end{multline*}
Taking limits as $n \to \infty$, we obtain that 
\[
\lim_{n \to \infty} \TV \left( \cCG \left( n, \tcorr, S \right), \cG \left( n, S \right)^{\otimes 2} \right) 
\geq 
1 - 2 \p_{0} \left( \left| f_{\infty}^{1} - f \left( G_{\tcorr}^{1} \right) \right| > \eps / 2 \right) 
- \p_{0} \left( \left| f_{\infty}^{1} - f_{\infty}^{2} \right| \leq \eps \right).
\]
Since $f \left(G_{\tcorr}^{1} \right) \to f_{\infty}^{1}$ almost surely as $\tcorr \to \infty$, 
we also have that 
$\p_{0} \left( \left| f_{\infty}^{1} - f \left( G_{\tcorr}^{1} \right) \right| > \eps / 2 \right) \to 0$ 
as $\tcorr \to \infty$. 
Thus there exists $t' = t' \left( \delta \right)$ 
such that 
$\p_{0} \left( \left| f_{\infty}^{1} - f \left( G_{\tcorr}^{1} \right) \right| > \eps / 2 \right) \leq \delta / 4$ 
for every $\tcorr \geq t'$. 
Combining this with~\eqref{eq:eps_delta} shows~\eqref{eq:delta_goal} and concludes the proof. 
\end{proof}

%%%%%%%%%%%%%%%%%%%%%%%%%%%%%%%%%%%%%%%%%%%%%%%%%%
\subsection{Applications to PA and UA trees} \label{sec:detection_large_tcorr_specific} %%%
%%%%%%%%%%%%%%%%%%%%%%%%%%%%%%%%%%%%%%%%%%%%%%%%%%

Here we show how Theorem~\ref{thm:high_prob_detection} can be applied to PA and UA trees, 
in order to prove Theorems~\ref{thm:high_prob_detection_PA} and~\ref{thm:high_prob_detection_UA}. 
In order to apply Theorem~\ref{thm:high_prob_detection}, 
we have to find a function $f$ such that 
\[
\lim_{t \to \infty} f \left( G_{t} \right) =: f_{\infty} 
\]
exists almost surely and that $f_{\infty}$ is an absolutely continuous random variable, 
where $\left\{ G_{t} \right\}_{t \geq \left| S \right|}$ is a sequence of randomly growing graphs with seed $S$ and attachment rule $\cG$, and where $\cG$ corresponds to  either $\PA$ or $\UA$ trees. 

We first argue that it is enough to show this for the special case when the seed is $S_{2}$, the unique tree on two vertices, as this implies the same for any seed tree $S$ on at least two vertices. 
Indeed, for a tree $S$ on at least two vertices, $\PA(n,S)$ has the same distribution as $\PA(n,S_{2})$ conditioned on $\PA \left( \left| S \right| , S_{2} \right) = S$ (an event which has positive probability), 
and therefore the function $f$ that works for the seed $S_{2}$ (i.e., which has the desired properties) also works when the seed is $S$. 
The same argument works for UA trees as well. 
More generally, suppose that $\cG$ is a Markov sequential attachment rule 
and we have a function $f$ satisfying the desired properties when the seed is $S'$. Then the same function $f$ also satisfies the desired properties whenever the seed $S$ satisfies $S \in \Range \left( \cG, S' \right)$. 
For PA and UA trees we simply use that $\Range \left( \PA, S_{2} \right) = \Range \left( \UA, S_{2} \right)$ consists of all finite trees on at least two vertices. 

Therefore in the following we may, and thus will, assume that the seed is $S = S_{2}$. 
We start with PA trees, for which considering the normalized maximum degree suffices. 

\begin{proof}[Proof of Theorem~\ref{thm:high_prob_detection_PA}]
For a graph $G$, define $f(G) := \Delta \left( G \right) / \sqrt{\left| G \right|}$, 
where recall that $\Delta \left( G \right)$ is the maximum degree in $G$. 
M\'ori showed that the limit $f_{\infty} := \lim_{n \to \infty} f \left( \PA \left( n, S_{2} \right) \right)$ exists almost surely, 
and moreover that the limit is almost surely positive, finite, and it has an absolutely continuous distribution~\cite[Theorem~3.1]{Mori:05}. 
Now applying Theorem~\ref{thm:high_prob_detection} yields the desired conclusion. 
\end{proof}

We next present a method that works equally well for both PA and UA trees, 
with only minor changes needed between the two cases. Accordingly, we present a unified proof for Theorems~\ref{thm:high_prob_detection_PA} and~\ref{thm:high_prob_detection_UA}, and throughout the proof we will always explain what differs for PA and UA trees. 
The proof is based on a notion of centrality in trees, which we detail below.

Given a tree $T$ and a distinguished vertex $v$ in the tree, let $(T,v)$ be the rooted tree with root $v$. For any other vertex $u$, $(T,v)_{u \downarrow}$ is the rooted subtree of $(T,v)$ whose root is $u$ and whose vertex set contains all vertices $w$ such that the unique path connecting $w$ and $v$ contains $u$. The {\it anti-centrality} of a vertex $v$ in a tree $T$ is defined as 
\[
\Psi_{T} (v) : =  \max\limits_{u \in \mathcal{N}_{v} (T)} \left| (T,v)_{u \downarrow} \right|,
\]
where 
$\mathcal{N}_{v} (T) := \left\{ u \in V \left( T \right) : (u,v) \in E \left( T \right) \right\}$ 
is the neighborhood of $v$ in $T$; 
see Figure~\ref{fig:centrality} for an illustration. 
Note that $\Psi_{T} \left( v \right)$ is efficiently computable (i.e., in $\poly\left( \left| T \right| \right)$ time, e.g., using a breadth first search (BFS) algorithm). 
A \emph{centroid} is a vertex that has minimum anti-centrality. 
Note that there can be multiple centroids, but only at most two (see, e.g.,~\cite[Lemma~2.1]{persistent_centrality}). 
If there is a unique centroid (which is often the case), then we refer to it as \emph{the} centroid. 
Properties of this centrality measure and of the corresponding centroid(s) have been widely studied, both for trees in general and also more specifically in a variety of sequentially-generated trees, including PA and UA trees (see, e.g.,~\cite{persistent_centrality} and the references therein). 
Centroids and centrality were also used as a key tool in root-finding algorithms in PA and UA trees~\cite{BDL16,LP19,DR19}.

\begin{figure}[t]
\centering
\includegraphics[width=0.38\textwidth]{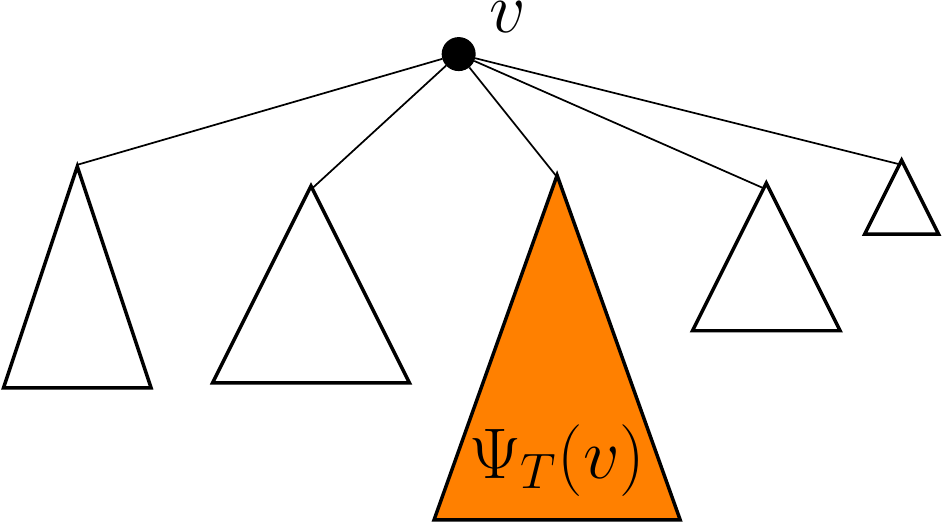}
\caption{The anti-centrality $\Psi_{T} \left( v \right)$ of vertex $v$ in tree $T$ is the size of the largest pendent subtree.}
\label{fig:centrality}
\end{figure}

In the following $\left\{ T_{n} \right\}_{n \geq 2}$ denotes a sequence of trees started from the seed $S_{2}$ and grown according to PA or UA. 
To abbreviate notation, we write 
$\Psi_{n} \left( v \right) := \Psi_{T_{n}} \left( v \right)$ 
for a vertex $v \in V \left( T_{n} \right)$. 
Recall that for a vertex $v$ in the tree $T_{n}$, we denote by $\tau(v)$ the {\it timestamp} of $v$. That is, $\tau(v) = k$ if $v$ is not in $T_{k-1}$ but is introduced in $T_{k}$. 
In the following when we refer to ``a fixed vertex $v$'', 
we mean that the timestamp $\tau(v)$ of $v$ is fixed (i.e., it does not change with $n$). 
The following theorem describes properties of the asymptotic behavior of the anti-centrality of a fixed vertex $v$ in PA and UA trees. 

 \newcommand{\thmcontinuouscentrality}{Let $\left\{ T_{n} \right\}_{n \geq 2}$ be a sequence of trees started from the seed $S_{2}$ and grown according to PA or UA. 
 Let $v$ be a fixed vertex. Then the limit 
 \[
 \Psi(v) := \lim_{n \to \infty} \frac{1}{n} \Psi_{n} \left( v \right)
 \]
 exists almost surely. 
 Furthermore, $\Psi(v)$ is an absolutely continuous random variable.}

\begin{theorem}
\label{thm:continuous_centrality}
\thmcontinuouscentrality
\end{theorem}

We refer to $\Psi(v)$ as the \emph{limiting anti-centrality} of $v$. 
We defer the proof of this theorem to Section~\ref{sec:subtree_distributions_anti-centrality}, where, 
in addition to Theorem~\ref{thm:continuous_centrality},  
we also prove a distributional representation of~$\Psi(v)$; see Theorem~\ref{thm:anti-centrality_representation}. %\textsection~\ref{sec:subtree_distributions}. 
The key insight behind the proof is that the evolution of the sizes of the subtrees around $v$ can be described in terms of  P\'{o}lya urn processes. The limits of these P\'{o}lya urn processes are absolutely continuous random variables, 
from which 
%. Due to the structure of the limiting random variables that describe the subtree sizes, 
we can show that $\Psi(v)$ is also an absolutely continuous random variable. The structure of $\Psi(v)$ is the same in both PA and UA trees (with only minor differences in the details), which allows us to develop techniques and proofs that simultaneously work for both models of random trees.

We are particularly interested in the anti-centrality of the centroid(s). 
Note that even if the tree $T$ has two centroids, 
the anti-centrality of the two centroids is equal, by definition. 
If $\theta(T)$ is a centroid of the tree $T$, then 
\[
\Psi_{T} \left( \theta(T) \right) 
= \min_{v \in T} \max\limits_{u \in \mathcal{N}_{v} (T)} \left| (T,v)_{u \downarrow} \right|.
\]
Turning to the sequence of trees $\left\{ T_{n} \right\}_{n \geq 2}$, 
let $\theta(n) := \theta \left( T_{n} \right)$ denote a centroid of $T_{n}$. 
Jog and Loh proved in~\cite{persistent_centrality}, for both PA and UA trees, 
that almost surely the centroid only changes finitely many times. 
That is, the limit 
$\theta := \lim_{n \to \infty} \theta \left( n \right)$ exists almost surely; 
we call $\theta$ the \emph{limiting centroid} of the sequence of trees $\left\{ T_{n} \right\}_{n \geq 2}$. 
Together with Theorem~\ref{thm:continuous_centrality} this implies the following corollary. 

\begin{corollary}
\label{cor:continuous_centrality}
Let $\left\{ T_{n} \right\}_{n \geq 2}$ be a sequence of trees started from the seed $S_{2}$ and grown according to PA or UA. 
Moreover, let $\theta(n) := \theta \left( T_{n} \right)$ denote a centroid of $T_{n}$. 
Then the limit 
\[
\lim_{n \to \infty} \frac{1}{n} \Psi_{n} \left( \theta \left( n \right) \right)
= \lim_{n \to \infty} \frac{1}{n} \min_{v \in T_{n}} \max\limits_{u \in \mathcal{N}_{v} (T_{n})} \left| (T_{n},v)_{u \downarrow} \right|
\]
exists almost surely and is an absolutely continuous random variable. 
\end{corollary}

\begin{proof}
By~\cite{persistent_centrality}, the centroid stabilizes almost surely, 
that is, the limiting centroid 
$\theta := \lim_{n \to \infty} \theta \left( n \right)$ exists almost surely. 
Let $v_{1}, v_{2}, \ldots, v_{n}$ denote the vertices in $T_{n}$, labeled in order of appearance; 
that is, $\tau \left( v_{k} \right) = k$ for $k > 2$ (and $v_{1}$ and $v_{2}$ are the two vertices in the initial tree $T_{2}$). 
Fix $k \geq 1$ and let $E_{k} := \left\{ \theta = v_{k} \right\}$ be the event that the limiting centroid is $v_{k}$. 
Let $E := \cup_{k \geq 1} E_{k}$ and note that $\p \left( E \right) = 1$. 
On the event $E_{k}$ we have that 
$\lim_{n \to \infty} \frac{1}{n} \Psi_{n} \left( \theta \left( n \right) \right) 
= \Psi \left( v_{k} \right)$, 
so altogether we have that 
$\lim_{n \to \infty} \frac{1}{n} \Psi_{n} \left( \theta \left( n \right) \right) 
= \Psi \left( \theta \right)$ 
and thus the limit exists almost surely. 
To see that the limit is absolutely continuous, let $F$ be a set with Lebesgue measure zero. Then 
\[
\p \left( \Psi \left( \theta \right) \in F \right) 
%= \p \left( \left\{ \Psi \left( \theta \right) \in F \right\} \cap E \right) 
= \sum_{k \geq 1} \p \left( \left\{ \Psi \left( \theta \right) \in F \right\} \cap E_{k} \right) 
= \sum_{k \geq 1} \p \left( \left\{ \Psi \left( v_{k} \right) \in F \right\} \cap E_{k} \right) 
\leq \sum_{k \geq 1} \p \left( \Psi \left( v_{k} \right) \in F \right) 
= 0,
\]
where in the second equality we used the definition of $E_{k}$ and in the last equality we used that $\Psi \left( v_{k} \right)$ is absolutely continuous for any fixed $k \geq 1$. 
\end{proof}

Corollary~\ref{cor:continuous_centrality} directly implies Theorems~\ref{thm:high_prob_detection_PA} and~\ref{thm:high_prob_detection_UA}, as follows. 

\begin{proof}[Proof of Theorems~\ref{thm:high_prob_detection_PA} and~\ref{thm:high_prob_detection_UA}]
For a tree $T$, define 
\[
f(T) := \frac{1}{\left| T \right|} \min_{v \in T} \max\limits_{u \in \mathcal{N}_{v} (T)} \left| (T,v)_{u \downarrow} \right|.
\]
By Corollary~\ref{cor:continuous_centrality}, 
the limit $f_{\infty} := f \left( T_{n} \right)$ exists almost surely and is absolutely continuous, for both PA and UA trees. 
Now applying Theorem~\ref{thm:high_prob_detection} yields the desired conclusion. 
\end{proof}

\subsection{The distribution of subtree sizes and anti-centrality}
\label{sec:subtree_distributions_anti-centrality}

In this section we derive the limiting distribution of the sizes of the subtrees around a fixed vertex~$v$, 
and using this we derive a distributional representation of the limiting anti-centrality $\Psi(v)$. 
Theorem~\ref{thm:continuous_centrality} then follows immediately. 
Before we state the main theorem of this section, we recall the definition of the timestamp $\tau(v)$ of $v$: $\tau(v) = k$ if $v$ is not in $T_{k-1}$ but is in $T_{k}$. 
In particular, we use the convention that the timestamp of both vertices in the initial tree $T_{2}$ is $2$. 

\begin{theorem}\label{thm:anti-centrality_representation}
Let $\left\{ T_{n} \right\}_{n \geq 2}$ be a sequence of trees started from the seed $S_{2}$ and grown according to PA or UA. 
Let $v$ be a fixed vertex. 
Let $\left\{ \varphi_{k} \right\}_{k \geq 0}$ be mutually independent random variables, all of them having a beta distribution, with parameters as follows: 
\[
\varphi_{0} \sim 
\begin{cases} 
\Beta \left( \tau(v) - 1, 1 \right) &\text{ for } \UA, \\
\Beta \left( \tau(v) - \frac{3}{2}, \frac{1}{2} \right) &\text{ for } \PA, 
\end{cases}
\]
and for $k \geq 1$, let 
\[
\varphi_{k} \sim 
\begin{cases} 
\Beta \left( 1, 1 \right) &\text{ for } \UA, \\
\Beta \left( \frac{1}{2}, \frac{k+1}{2} \right) &\text{ for } \PA. 
\end{cases}
\]
We then define the random variables $\left\{ \psi_{\ell} \right\}_{\ell \geq 0}$ as follows: $\psi_{0} := \varphi_{0}$, and for $\ell \geq 1$ let 
\[
\psi_{\ell} := \varphi_{\ell} \prod_{i=0}^{\ell-1} \left( 1 - \varphi_{i} \right).
\]
The limiting anti-centrality $\Psi(v)$ of $v$ exists almost surely and has the following distributional representation: 
\begin{equation}\label{eq:Psi_rep}
\Psi(v) \stackrel{d}{=} \max_{\ell \geq 0} \psi_{\ell}.
\end{equation}
\end{theorem}

In this representation $\psi_{\ell}$ is the asymptotic normalized size of the $\ell$th subtree around $v$; here counting starts at $\ell = 0$ and subtrees are ordered according to their first appearance around~$v$. 
Similar representations---of various limiting quantities using a sequence of independent (beta) random variables---are common in the study of preferential attachment, uniform attachment, and related random graph models (see, e.g.,~\cite{BeBoChSa:14,RB17}).

\begin{proof}
We first prove the claim for UA trees. 
Note that $v$ is a leaf in $T_{\tau(v)}$. Let $u_{0}$ denote the neighbor of $v$ in $T_{\tau(v)}$ and let $e_{0}$ denote the edge connecting $v$ and $u_{0}$. For $n \geq \tau(v)$, the edge $e_{0}$ partitions $T_{n}$ into two subtrees: 
$\left( T_{n}, v \right)_{u_{0} \downarrow}$ and $T_{n} \setminus \left( T_{n}, v \right)_{u_{0} \downarrow}$. 
When a new vertex joins the tree, it attaches to an existing vertex uniformly at random. Therefore, the probability of the new vertex joining either one of these two subtrees is proportional to their size. Thus the evolution of the pair of subtree sizes, 
$\left( \left| \left( T_{n}, v \right)_{u_{0} \downarrow} \right|, n - \left| \left( T_{n}, v \right)_{u_{0} \downarrow} \right| \right)$, 
follows a classical P\'olya urn. 
Initially, at time $n = \tau(v)$, the pair of subtree sizes is $\left( \tau(v) - 1, 1 \right)$. 
Therefore, by classical results on P\'olya urns (see, e.g.,~\cite[Section~4.5 and Example~4.7]{RB17}), 
the limit 
\begin{equation}\label{eq:varphi0}
\varphi_{0} := \lim_{n \to \infty} \frac{1}{n} \left| \left( T_{n}, v \right)_{u_{0} \downarrow} \right| 
\end{equation}
exists almost surely and $\varphi_{0} \sim \Beta \left( \tau(v) - 1, 1 \right)$. 

Next, let $u_{1}$ denote the first vertex that attaches to $v$ with $\tau(u_{1}) > \tau(v)$, and let $e_{1}$ denote the edge connecting $v$ and $u_{1}$. (Note that almost surely $\tau(u_{1}) < \infty$.) 
For $n \geq \tau \left( u_{1} \right)$, the edges $e_{0}$ and $e_{1}$ partition the tree $T_{n}$ into three subtrees: 
$\left( T_{n}, v \right)_{u_{0} \downarrow}$, 
$\left( T_{n}, v \right)_{u_{1} \downarrow}$, 
and $T_{n} \setminus \left(  \left( T_{n}, v \right)_{u_{0} \downarrow} \cup \left( T_{n}, v \right)_{u_{1} \downarrow} \right)$. 
When a new vertex joins the tree, it attaches to an existing vertex uniformly at random. We can view this as a multi-stage process as follows. 
First, the vertex decides whether it will join the subtree 
$\left( T_{n}, v \right)_{u_{0} \downarrow}$ 
or the subtree 
$T_{n} \setminus \left( T_{n}, v \right)_{u_{0} \downarrow}$; 
it does so by flipping a coin, with the probability of choosing either option being proportional to the size of the respective subtree. 
Next, if the vertex decides to join the subtree 
$T_{n} \setminus \left( T_{n}, v \right)_{u_{0} \downarrow}$, 
it then chooses whether to join the subtree 
$\left( T_{n}, v \right)_{u_{1} \downarrow}$ 
or the subtree 
$T_{n} \setminus \left(  \left( T_{n}, v \right)_{u_{0} \downarrow} \cup \left( T_{n}, v \right)_{u_{1} \downarrow} \right)$; 
it again does so by flipping a coin, with the probability of choosing either option being proportional to the size of the respective subtree. 
This second coin flip is independent of the first coin flip. 
Finally, once the vertex has decided which of the three subtrees to join, it attaches to a vertex chosen uniformly at random from the given subtree. 

From this construction it is immediate that, when viewed at the times when the new vertex joins the subtree 
$T_{n} \setminus \left( T_{n}, v \right)_{u_{0} \downarrow}$, 
the pair 
\begin{equation}\label{eq:u1_pair}
\left( \left| \left( T_{n}, v \right)_{u_{1} \downarrow} \right|, 
n - \left| \left( T_{n}, v \right)_{u_{0} \downarrow} \right| - \left| \left( T_{n}, v \right)_{u_{1} \downarrow} \right| \right)
\end{equation}
evolves as a classical P\'olya urn started from $(1,1)$. Thus the limit 
\begin{equation}\label{eq:varphi1}
\varphi_{1} := \lim_{n \to \infty} \frac{\left| \left( T_{n}, v \right)_{u_{1} \downarrow} \right|}{n - \left| \left( T_{n}, v \right)_{u_{0} \downarrow} \right|} 
\end{equation}
exists almost surely and $\varphi_{1} \sim \Beta \left( 1, 1 \right)$ (in other words, $\varphi_{1}$ is uniform on the interval $[0,1]$). 
Moreover, the evolution of the P\'olya urn describing the pair in~\eqref{eq:u1_pair} is independent of the process that determines the times at which the subtree 
$T_{n} \setminus \left( T_{n}, v \right)_{u_{0} \downarrow}$ increases, 
which means that $\varphi_{1}$ and $\varphi_{0}$ are independent. 
Putting together~\eqref{eq:varphi0} and~\eqref{eq:varphi1}, we obtain that 
\[
\lim_{n \to \infty} \frac{1}{n} \left| \left( T_{n}, v \right)_{u_{1} \downarrow} \right| 
= \lim_{n \to \infty}  \frac{\left| \left( T_{n}, v \right)_{u_{1} \downarrow} \right|}{n - \left| \left( T_{n}, v \right)_{u_{0} \downarrow} \right|} \frac{n - \left| \left( T_{n}, v \right)_{u_{0} \downarrow} \right|}{n}
= \varphi_{1} \left( 1 - \varphi_{0} \right) 
= \psi_{1} 
\]
almost surely. 

We can then iterate this argument. 
For $\ell \geq 2$, let $u_{\ell}$ denote the $\ell$th vertex to attach to $v$. 
The random variables $\varphi_{2}, \varphi_{3}, \ldots$ can be defined inductively by the limit 
\[
\varphi_{\ell} := \lim_{n \to \infty} \frac{\left| \left( T_{n}, v \right)_{u_{\ell} \downarrow} \right|}{n - \sum_{i = 0}^{\ell - 1} \left| \left( T_{n}, v \right)_{u_{i} \downarrow} \right|}; 
\]
the same argument as above shows that this limit exists almost surely, $\varphi_{\ell} \sim \Beta \left( 1, 1 \right)$ for every $\ell \geq 1$, 
and that $\varphi_{\ell}$ is independent of $\varphi_{0}, \varphi_{1}, \ldots, \varphi_{\ell - 1}$. Subsequently, this implies by induction that 
\[
\lim_{n \to \infty} \frac{1}{n} \left| \left( T_{n}, v \right)_{u_{\ell} \downarrow} \right| 
= \lim_{n \to \infty}  \frac{\left| \left( T_{n}, v \right)_{u_{\ell} \downarrow} \right|}{n - \sum_{i=0}^{\ell-1} \left| \left( T_{n}, v \right)_{u_{i} \downarrow} \right|} \frac{n - \sum_{i=0}^{\ell-1} \left| \left( T_{n}, v \right)_{u_{i} \downarrow} \right|}{n}
= \varphi_{\ell} \prod_{i=0}^{\ell-1} \left( 1 - \varphi_{i} \right) 
= \psi_{\ell} 
\]
almost surely. 
We have thus shown that the asymptotic normalized size of the $\ell$th subtree around~$v$ is given by $\psi_{\ell}$. 
What remains is to understand how the subtree sizes of these fixed neighbors of~$v$ relate to the anti-centrality of $v$. 

Define the event $E_{k} := \left\{ \varphi_{k} > 1 / 2 \right\}$ and let $E := \cup_{k \geq 1} E_{k}$. The events $\left\{ E_{k} \right\}_{k \geq 1}$ are mutually independent and $\p \left( E_{k} \right) = 1/2$ for every $k \geq 1$. Therefore $\p \left( E \right) = 1$. 
Since $E_{k}$ holds if and only if $\varphi_{k} > 1 - \varphi_{k}$, 
the event $E_{k}$ is equivalent to the event that 
\[
\lim_{n \to \infty} \frac{1}{n} \left| \left( T_{n}, v \right)_{u_{k} \downarrow} \right| 
> 
\lim_{n \to \infty} \frac{1}{n} \left( n - \sum_{\ell = 0}^{k} \left| \left( T_{n}, v \right)_{u_{\ell} \downarrow} \right| \right) 
\]
holds. 
Thus on the event $E_{k}$ we have, for all $n$ large enough, that 
\[
\left| \left( T_{n}, v \right)_{u_{k} \downarrow} \right| 
> 
n - \sum_{\ell = 0}^{k} \left| \left( T_{n}, v \right)_{u_{\ell} \downarrow} \right|.
\]
Since for any $u \in \cN_{v} \left( T_{n} \right) \setminus \left\{ u_{0}, u_{1}, \ldots, u_{k} \right\}$ we have that 
\[
\left| \left( T_{n}, v \right)_{u \downarrow} \right| 
\leq 
n - \sum_{\ell = 0}^{k} \left| \left( T_{n}, v \right)_{u_{\ell} \downarrow} \right|, 
\]
it follows that 
\[
\Psi_{n} \left( v \right) = \max_{\ell \in \left\{ 0, 1, \ldots, k \right\}} \left| \left( T_{n}, v \right)_{u_{\ell} \downarrow} \right| 
\]
for all $n$ large enough, on the event $E_{k}$. 
Thus dividing by $n$ and taking limits, we have that, on the event $E_{k}$, the limit 
$\Psi(v) := \lim_{n \to \infty} \frac{1}{n} \Psi_{n} \left( v \right)$ exists and moreover 
\[
\Psi(v) = \max_{\ell \in \left\{ 0, 1, \ldots, k \right\}} \psi_{\ell}. 
\]
Consequently, on the event $E$, the limit 
$\Psi(v) := \lim_{n \to \infty} \frac{1}{n} \Psi_{n} \left( v \right)$ exists and moreover 
$\Psi(v) = \max_{\ell \geq 0} \psi_{\ell}$. 
Since $E$ holds almost surely, this concludes the proof of~\eqref{eq:Psi_rep} for UA trees. 

For PA trees the arguments are similar, so we only explain the differences.  
In PA, when a new vertex joins the tree, it attaches to an existing vertex with probability proportional to its degree. 
Thus, if we partition the tree into finitely many subtrees, 
the probability that the new vertex joins a particular subtree is proportional to the sum of the degrees of the vertices in the subtree. 
Moreover, when a vertex joins a particular subtree, it increases the sum of the degrees in the subtree by $2$, due to the new edge. For more details, see~\cite[Section~4.5 and Example~4.11]{RB17}.

Thus there are two differences in the analysis of subtrees above: 
(1) the quantity associated with a subtree that we analyze is now the sum of the degrees of the vertices in the subtree (instead of the number of vertices in the subtree), 
and (2) the P\'olya urns that arise have replacement matrix 
$\left(\begin{smallmatrix} 2 & 0 \\ 0 & 2 \end{smallmatrix}\right)$ 
(see~\cite[Section~4.5]{RB17}). 
The first change also means that the initial conditions of the appropriate P\'olya urns are different. 
Specifically, 
the limiting random variable $\varphi_{0}$ 
arises from a P\'olya urn with 
replacement matrix 
$\left(\begin{smallmatrix} 2 & 0 \\ 0 & 2 \end{smallmatrix}\right)$ 
and initial condition 
$\left( 2 \tau \left( v \right) - 3, 1 \right)$, 
which is why $\varphi_{0} \sim \Beta \left( \tau(v) - \frac{3}{2}, \frac{1}{2} \right)$. 
For $k \geq 1$, 
the limiting random variable $\varphi_{k}$ arises from a P\'olya urn with 
replacement matrix 
$\left(\begin{smallmatrix} 2 & 0 \\ 0 & 2 \end{smallmatrix}\right)$ 
and initial condition 
$\left( 1, k+1 \right)$, 
which is why $\varphi_{k} \sim \Beta \left( \frac{1}{2}, \frac{k+1}{2} \right)$.

There is one more subtle point here: 
we are interested in the asymptotic behavior of the sizes of various subtrees (that is, the number of vertices in the subtrees), 
but the analysis concerns the sum of the degrees of the vertices in the subtrees. 
However, 
the map $x \mapsto 2x - 1$ 
takes the number of vertices in a subtree 
to the sum of the degrees of the vertices in the subtree 
(this uses the fact that we are considering subtrees where there is exactly one edge exiting the subtree). 
The normalization factor also differs by essentially a factor of $2$: 
it is $n$ when the considering the number of vertices 
and $2n -2$ when considering the sum of the degrees. 
Thus after normalization the quantity that we care about (subtree size) 
is asymptotically the same as the quantity that we analyze (sum of the degrees in a subtree).

With these changes we have thus determined that the asymptotic normalized size of the $\ell$th subtree around $v$ is given by $\psi_{\ell}$ for PA trees. 
What remains is to show~\eqref{eq:Psi_rep} for PA trees. 
Since the random variables $\left\{ \varphi_{k} \right\}_{k \geq 1}$ are no longer i.i.d. uniform on $[0,1]$ (as in the case of UA trees), a different argument is needed here. For $k \geq 1$ define the event 
\[
E_{k} := \left\{ \max_{\ell \in \left\{ 0, 1, \ldots, k \right\}} \psi_{\ell} > 1 - \sum_{\ell = 0}^{k} \psi_{\ell} \right\}
\]
and let 
$E := \cup_{k \geq 0} E_{k}$. 
An analogous argument as above shows that 
on the event $E_{k}$ we have that 
$\Psi(v) := \lim_{n \to \infty} \frac{1}{n} \Psi_{n} \left( v \right)$ 
exists and moreover 
$\Psi(v) = \max_{\ell \in \left\{ 0, 1, \ldots, k \right\}} \psi_{\ell}$. 
Thus on the event~$E$ we have that 
$\Psi(v) := \lim_{n \to \infty} \frac{1}{n} \Psi_{n} \left( v \right)$ 
exists and moreover 
$\Psi(v) = \max_{\ell \geq 0} \psi_{\ell}$. 
What remains to show is that $\p \left( E \right) = 1$, 
which is equivalent to showing that $\lim_{k \to \infty} \p \left( E_{k} \right) = 1$, since $\left\{ E_{k} \right\}_{k \geq 0}$ is an increasing sequence of events. This, in turn, follows from the fact that 
$1 - \sum_{\ell = 0}^{k} \psi_{\ell} \to 0$ in probability as $k \to \infty$. 
To see that this convergence in probability holds, 
first observe that 
$1 - \sum_{\ell = 0}^{k} \psi_{\ell} 
= \prod_{\ell = 0}^{k} \left( 1 - \varphi_{k} \right)$. 
Then by independence we have that 
\[
\E \left[ 1 - \sum_{\ell = 0}^{k} \psi_{\ell} \right] 
= \prod_{\ell = 0}^{k} \E \left[ 1 - \varphi_{k} \right]
= \frac{1}{2\tau(v) - 2} \prod_{\ell = 1}^{k} \frac{\ell + 1}{\ell + 2} 
= \frac{1}{\left( \tau(v) - 1 \right) \left( k + 2 \right)},
\]
which goes to $0$ as $k \to \infty$. 
The conclusion then follows from Markov's inequality. 
\end{proof} 

Theorem~\ref{thm:anti-centrality_representation} directly implies Theorem~\ref{thm:continuous_centrality}, as we now show. 

\begin{proof}[Proof of Theorem~\ref{thm:continuous_centrality}]
By Theorem~\ref{thm:anti-centrality_representation}, 
the limiting anti-centrality $\Psi(v)$ exists almost surely. 
Moreover, it satisfies the distributional representation given in~\eqref{eq:Psi_rep}. 
That is, it is the maximum of countably many absolutely continuous random variables. 
As such, it is absolutely continuous as well. Indeed, if $F$ is a set with Lebesgue measure zero, then 
\[
\p \left( \Psi(v) \in F \right) 
= \p \left( \max_{\ell \geq 0} \psi_{\ell} \in F \right) 
\leq \sum_{\ell \geq 0} \p \left( \psi_{\ell} \in F \right) 
= 0. \qedhere
\]
\end{proof}

%%%%%%%%%%%%%%%%%%%%%%%%%%%%%%%%%%%%%%%%%%%%%%%%%%
\section{An initial, coarse estimate of $\tcorr$} \label{sec:estimation_tcorr_coarse} %%%
%%%%%%%%%%%%%%%%%%%%%%%%%%%%%%%%%%%%%%%%%%%%%%%%%%

We now turn to the problem of estimating $\tcorr$. 
The estimator that we use to prove Theorem~\ref{thm:estimation} is somewhat involved, so in this section we first study a simpler estimator. The guarantees we prove for this simpler estimator are weaker than those in Theorem~\ref{thm:estimation} (see Theorem~\ref{thm:estimation_coarse} below), but studying this simpler estimator highlights some of the key ideas that also go into the more involved estimator studied subsequently in Section~\ref{sec:estimation_large_tcorr}. Moreover, as we shall see in Section~\ref{sec:estimation_large_tcorr}, our estimator for $\tcorr$ that achieves vanishing relative error needs as input an initial, coarse estimate of~$\tcorr$---and the simple estimator studied in this section provides this.

In this section we will thus prove the following result. 

\begin{theorem}[A coarse estimate of $\tcorr$ in PA and UA trees]\label{thm:estimation_coarse} 
Let $S = S_{2}$ be the unique tree on two vertices 
and let $\left( T_{n}^{1}, T_{n}^{2} \right) \sim \CPA \left( n, \tcorr, S \right)$. 
There exists an estimator 
$\wh{t}_{n} \equiv \wh{t} \left( T_{n}^{1}, T_{n}^{2} \right)$, 
computable in polynomial time, such that 
\[
\lim_{\tcorr \to \infty} \liminf_{n \to \infty} 
\p \left( \frac{\tcorr}{\log \tcorr}
\leq \wh{t}_{n} 
\leq \tcorr \log \tcorr \right)
= 1. 
\]
The same result also holds when 
$\left( T_{n}^{1}, T_{n}^{2} \right) \sim \CUA \left( n, \tcorr, S \right)$. 
\end{theorem}

We now describe the estimator used to prove Theorem~\ref{thm:estimation_coarse}. 
Recall all the notation introduced in Sections~\ref{sec:detection_large_tcorr_specific} and~\ref{sec:subtree_distributions_anti-centrality}, which we will use here. 
Moreover, for anything introduced previously in these sections, if we add a superscript $i$ to it (where $i \in \left\{1,2\right\}$), this means that it is the appropriate object in the tree $T_{n}^{i}$. 
For instance, $\theta^{1} \left( n \right)$ and $\theta^{2} \left( n \right)$ are the centroids in $T_{n}^{1}$ and $T_{n}^{2}$, respectively. 

The main idea is to consider the minimum anti-centrality in the two trees $T_{n}^{1}$ and $T_{n}^{2}$. In other words, we consider the sizes of the largest pendent subtrees of the two centroids. 
The heuristic, which we will make precise, is as follows. 
If $\tcorr$ is large, then the centroids in $T_{n}^{1}$ and $T_{n}^{2}$ correspond to the same vertex, with probability close to $1$. 
If this is the case, then the sizes of the largest pendent subtrees of the centroids should be similar, and their difference should concentrate on some function of $n$ and $\tcorr$---which should be a function of only $\tcorr$ in the limit as~$n\to\infty$. 
Estimating this function and inverting it then allows us to estimate $\tcorr$. 
See Figure~\ref{fig:coarse_estimate} for an illustration.

Thus we define, for $i \in \left\{ 1, 2 \right\}$, the random variable 
\begin{equation}\label{eq:X}
X_{n}^{i} 
:= \frac{1}{n} \min_{v \in T_{n}^{i}} \max_{u \in \cN_{v} \left( T_{n}^{i} \right)} \left| \left( T_{n}^{i}, v \right)_{u \downarrow} \right| 
= \frac{1}{n} \Psi_{T_{n}^{i}} \left( \theta^{i} \left( n \right) \right).
\end{equation}
Now define 
\begin{equation}\label{eq:Y}
Y_{n} := \frac{\left( X_{n}^{1} - X_{n}^{2} \right)^{2}}{2 X_{n}^{1} \left( 1 - X_{n}^{1} \right)}.
\end{equation}
As we shall see, 
$Y_{n}$ is concentrated around $1/\tcorr$, 
so we can define the estimator $\wh{t}_{n} := 1/ Y_{n}$. 
Theorem~\ref{thm:estimation_coarse} then follows immediately from the following result. 
\begin{theorem}\label{thm:estimation_coarse_inverse} 
Let $S = S_{2}$ be the unique tree on two vertices 
and let $\left( T_{n}^{1}, T_{n}^{2} \right) \sim \CPA \left( n, \tcorr, S \right)$.  
Define $Y_{n}$ via~\eqref{eq:X} and~\eqref{eq:Y}. 
We have that 
\[
\lim_{\tcorr \to \infty} \liminf_{n \to \infty} 
\p \left( \frac{1}{\tcorr \log \tcorr}
\leq Y_{n} 
\leq \frac{\log \tcorr}{\tcorr} \right)
= 1. 
\]
The same result also holds when 
$\left( T_{n}^{1}, T_{n}^{2} \right) \sim \CUA \left( n, \tcorr, S \right)$. 
\end{theorem}

\begin{figure}[t]
\centering
\includegraphics[width=0.8\textwidth]{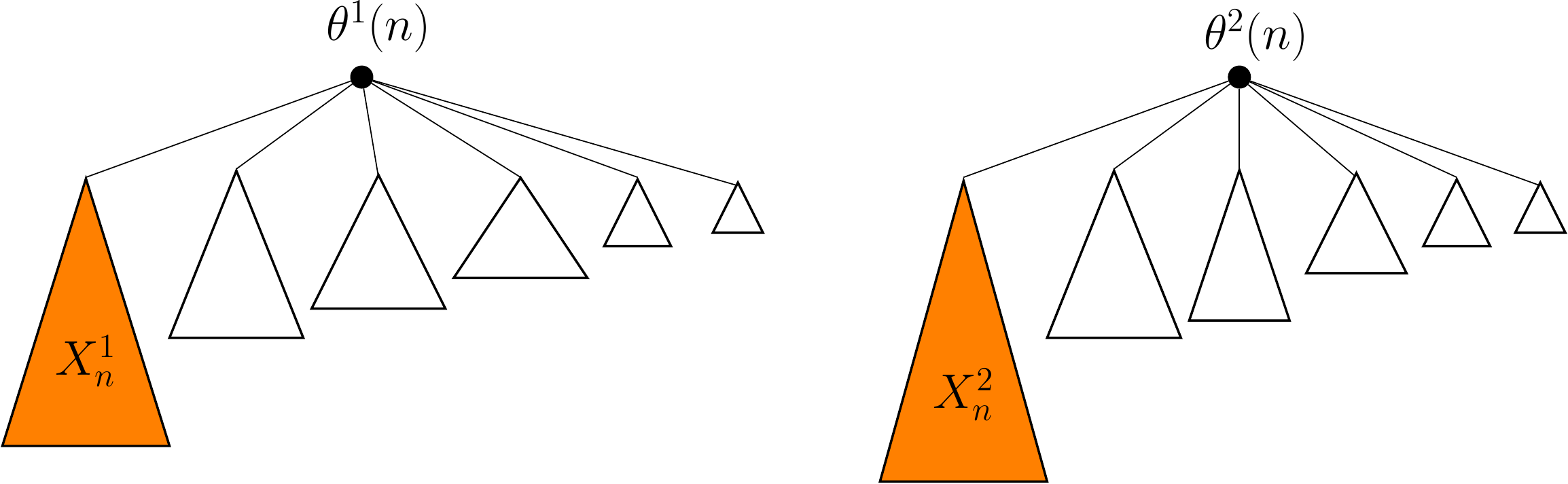}
\caption{The pendent subtrees of the centroids $\theta^{1}(n)$ and $\theta^{2}(n)$ in $T_{n}^{1}$ and $T_{n}^{2}$, respectively, are ordered in decreasing order. The estimator studied in Section~\ref{sec:estimation_tcorr_coarse} is a function of the (normalized) sizes of the largest pendent subtrees, $X_{n}^{1}$ and $X_{n}^{2}$.}
\label{fig:coarse_estimate}
\end{figure}

In the remainder of this section we prove this theorem. 
We start in Section~\ref{sec:coarse_preliminaries} with some preliminaries: 
specifically, we define a couple of ``nice'' events on the space of sequences of growing trees, 
on which we will obtain bounds for $Y_{n}$. 
We prove a first moment estimate for $Y_{n}$ in Section~\ref{sec:coarse_first_moment}. We then prove Theorem~\ref{thm:estimation_coarse_inverse} in Section~\ref{sec:coarse_conclusion}, using the fact that the previously defined ``nice'' events have probability close to~$1$. Finally, we prove this latter fact in Section~\ref{sec:coarse_nice_event_proof}.

\subsection{Preliminaries}\label{sec:coarse_preliminaries}

We start by introducing some notation on labeling vertices. 
Let $\left\{ T_{n} \right\}_{n \geq 2}$ be a growing sequence of trees started from the seed $S=S_{2}$, where at each step we add a single new node and a new edge. 
We denote the vertices of $T_{n}$ by $v_{1}, v_{2}, \ldots, v_{n}$, 
where $v_{1}$ and $v_{2}$ are the two initial vertices in~$S$, and for $k \geq 3$, $v_{k}$ is the unique vertex with timestamp $k$. 
As before, we write $\Psi_{n} (v) := \Psi_{T_{n}} (v)$ for a vertex $v \in V(T_{n})$. 
We write $\wt{v}_{i,n} (1)$ for the neighbor of $v_{i}$ that is the root of the largest subtree of $\left( T_{n}, v_{i} \right)$ (assuming that there is a unique largest subtree). 
With this notation we have that 
$\Psi_{n} \left( v_{i} \right) = \left| \left( T_{n}, v_{i} \right)_{\wt{v}_{i,n} (1) \downarrow} \right|$. 
More generally, for any $k \geq 1$ we write $\wt{v}_{i,n} (k)$ for the neighbor of $v_{i}$ that is the root of the $k$th largest subtree of $\left( T_{n}, v_{i} \right)$ (assuming that there is a unique such vertex). 
Finally, we write $\wt{\theta}_{n} \left( 1 \right)$ for the neighbor of the centroid $\theta(n)$ that is the root of the largest subtree of $\left( T_{n}, \theta \left( n \right) \right)$ (assuming that the centroid is unique and that there is a unique largest subtree).

We are now ready to define what we mean by the ``nice'' event on the space of sequences of growing trees. 

\begin{definition}[The event $\cA$]\label{def:nice_event}
Given a sequence of trees $\left\{ T_{n} \right\}_{n \geq \tcorr}$, we say that the event $\cA$ holds if and only if the following three properties all hold: 
\begin{enumerate}[label=(A\arabic*)]
\item\label{prop:nice1}%[(A1)]\label{prop:nice1} 
The centroid $\theta(n)$ is unique for all $n \geq \tcorr$ 
and $\theta \left( n \right) = \theta \left( \tcorr \right)$ for all $n \geq \tcorr$. 
\item\label{prop:nice2} %[(A2)]\label{prop:nice2} 
The vertex $\wt{\theta}_{n} \left( 1 \right)$ is uniquely defined for all $n \geq \tcorr$ and $\wt{\theta}_{n} \left( 1 \right) = \wt{\theta}_{\tcorr} \left( 1 \right)$ for all $n \geq \tcorr$. 
\item\label{prop:nice3}%[(A3)]\label{prop:nice3} 
For all $n \geq \tcorr$ we have that 
\begin{equation}\label{eq:A3}
\left| \frac{1}{n} \Psi_{n} \left( \theta \left( n \right) \right) 
- \frac{1}{\tcorr} \Psi_{\tcorr} \left( \theta \left( \tcorr \right) \right) \right| 
\leq \frac{1}{\tcorr^{1/3}} \min \left\{ \frac{1}{\tcorr} \Psi_{\tcorr} \left( \theta \left( \tcorr \right) \right), 1 - \frac{1}{\tcorr} \Psi_{\tcorr} \left( \theta \left( \tcorr \right) \right) \right\}.
\end{equation}
\end{enumerate}
\end{definition}

The exponent $1/3$ in~\eqref{eq:A3} is chosen for simplicity; any positive constant that is less than $1/2$ is a good choice for everything that follows. 
Furthermore, we always have that $\Psi_{\tcorr} \left( \theta \left( \tcorr \right) \right) \leq \tcorr / 2$---this is a known property of tree centroids (see, e.g.,~\cite[Lemma~2.1]{persistent_centrality})---so the minimum in~\eqref{eq:A3} is always attained by the first term; 
we include the second term in the definition just for clarity. 
Given a sequence of trees $\left\{ T_{n} \right\}_{n \geq 2}$, 
we say that the event $\cA$ holds if and only if it holds for the subsequence $\left\{ T_{n} \right\}_{n \geq \tcorr}$. 
The event $\cA$ clearly depends on $\tcorr$, but we choose to omit $\tcorr$ from the notation in order to keep notation lighter. 
The following lemma shows that for PA and UA trees the event $\cA$ holds with probability close to $1$ when $\tcorr$ is large.

\begin{lemma}\label{lem:nice_event}
Let $\left\{ T_{n} \right\}_{n \geq 2}$ be a sequence of trees started from the seed $S$ and grown according to PA or UA. 
There exists a finite constant $C$ such that 
for every $\tcorr \geq 2$ we have that 
\begin{equation}\label{eq:notA}
\p \left( \cA^{c} \right) \leq \frac{C}{\log \tcorr},
\end{equation}
where $\cA^{c}$ denotes the complement of $\cA$. 
\end{lemma}
The proof of Lemma~\ref{lem:nice_event} is deferred to Section~\ref{sec:coarse_nice_event_proof}. 

The intuition behind defining $\cA$ in this way is as follows. 
On the event $\cA$, both the centroid and the largest subtree of the centroid do not change locations within the tree for $n \geq \tcorr$. 
Hence, by conditioning on the tree at time $\tcorr$, 
studying $\Psi_{n} \left( \theta \left( n \right) \right)$ essentially amounts to understanding the growth of a \emph{fixed} subtree that is present in $T_{\tcorr}$. 
Since the sizes of fixed subtrees grow according to P\'olya urn processes (in PA and UA trees), their distributions are very well understood. 

We are interested in a pair of correlated randomly growing (PA or UA) trees 
$\left\{ \left( T_{n}^{1}, T_{n}^{2} \right) \right\}_{n \geq 2}$. 
Let $\cA^{1}$ and $\cA^{2}$ denote the ``nice'' events corresponding to 
$\left\{ T_{n}^{1} \right\}_{n \geq 2}$ 
and 
$\left\{ T_{n}^{2} \right\}_{n \geq 2}$. 
Since $T_{n}^{1} = T_{n}^{2}$ for all $n \leq \tcorr$, 
we have, in particular, that 
$\theta^{1} \left( \tcorr \right) = \theta^{2} \left( \tcorr \right) =: \theta \left( \tcorr \right)$ 
and also that 
$\wt{\theta}_{\tcorr}^{1} \left( 1 \right) = \wt{\theta}_{\tcorr}^{2} \left( 1 \right) =: \wt{\theta}_{\tcorr} \left( 1 \right)$. 
A key observation is that on the event $\cA^{1} \cap \cA^{2}$ 
we have that 
$\theta^{1}(n) = \theta^{2}(n) = \theta \left( \tcorr \right)$ 
for all $n \geq \tcorr$ 
and that 
$\wt{\theta}_{n}^{1} \left( 1 \right) = \wt{\theta}_{n}^{2} \left( 1 \right) = \wt{\theta}_{\tcorr} \left( 1 \right)$ 
for all $n \geq \tcorr$, 
which implies that on the event $\cA^{1} \cap \cA^{2}$ we have that 
$X_{n}^{i} = \frac{1}{n} \left| \left( T_{n}^{i}, \theta \left( \tcorr \right) \right)_{\wt{\theta}_{\tcorr} \left( 1 \right) \downarrow} \right|$ for $i \in \left\{ 1, 2 \right\}$ and all $n \geq \tcorr$. 
Thus in order to understand the behavior of the statistic $Y_{n}$ on the event $\cA^{1} \cap \cA^{2}$, it suffices to condition on the tree at time $\tcorr$ and then analyze the behavior of fixed subtrees. We do this next. 

Condition now on the tree $T_{\tcorr}^{1} = T_{\tcorr}^{2} =: T_{\tcorr}$; that is, assume that $T_{\tcorr}$ is given. 
To abbreviate notation, we write 
$\theta := \theta \left( \tcorr \right)$ 
and 
$\wt{\theta} \left( 1 \right) := \wt{\theta}_{\tcorr} \left( 1 \right)$; 
importantly, note that these are now \emph{fixed} vertices (i.e., they do not change with $n$). 
Define the random variables 
\[
Z_{n}^{i} := \frac{1}{n} \left| \left( T_{n}^{i}, \theta \right)_{\wt{\theta} \left( 1 \right) \downarrow} \right|
\]
for $i \in \left\{ 1, 2 \right\}$ and $n \geq \tcorr$. 
As observed above, on the event $\cA^{i}$ we have that $X_{n}^{i} = Z_{n}^{i}$ for $n \geq \tcorr$. 

In UA trees, the evolution of 
$\left( n Z_{n}^{i}, n - n Z_{n}^{i} \right)$ for $n \geq \tcorr$ 
follows a classical P\'olya urn 
with initial condition 
$\left( \Psi_{\tcorr} \left( \theta \right), \tcorr - \Psi_{\tcorr} \left( \theta \right) \right)$, 
for $i \in \left\{ 1, 2 \right\}$. 
Moreover, the P\'olya urns for $i = 1$ and $i = 2$ are independent (recall that we are conditioning on $T_{\tcorr}$, so this is conditional independence given $T_{\tcorr}$). 
In PA trees, the evolution of 
$\left( 2n Z_{n}^{i} - 1, 2n - 2n Z_{n}^{i} - 1 \right)$ for $n \geq \tcorr$ 
follows a P\'olya urn with replacement matrix 
$\left(\begin{smallmatrix} 2 & 0 \\ 0 & 2 \end{smallmatrix}\right)$ 
and initial condition 
$\left( 2 \Psi_{\tcorr} \left( \theta \right) - 1, 2 \tcorr - 2 \Psi_{\tcorr} \left( \theta \right) - 1 \right)$, 
for $i \in \left\{ 1, 2 \right\}$. 
Moreover, the P\'olya urns for $i = 1$ and $i = 2$ are independent (again, this is conditional independence given $T_{\tcorr}$). 

Thus by classical results on P\'olya urns it follows that the limiting random variables 
\[
Z^{i} := \lim_{n \to \infty} Z_{n}^{i} 
\]
exist almost surely for $i \in \left\{ 1, 2 \right\}$, 
for both PA and UA trees. 
Moreover, $Z^{1}$ and $Z^{2}$ are i.i.d.\ (again, this is conditional independence given $T_{\tcorr}$) 
beta random variables, with parameters given as follows: 
\begin{equation}\label{eq:Z}
Z \sim 
\begin{cases} 
\Beta \left( \Psi_{\tcorr} \left( \theta \right), \tcorr - \Psi_{\tcorr} \left( \theta \right) \right) &\text{ for } \UA, \\
\Beta \left( \Psi_{\tcorr} \left( \theta \right) - \frac{1}{2}, \tcorr - \Psi_{\tcorr} \left( \theta \right) - \frac{1}{2} \right) &\text{ for } \PA. 
\end{cases}
\end{equation}
Here $Z$ is a random variable with the same distribution as $Z^{1}$ and $Z^{2}$. 

From~\eqref{eq:Z} it is clear that the quantity $\Psi_{\tcorr} \left( \theta \right)$ plays an important role in the distribution of~$Z$. 
%We always have that $\Psi_{\tcorr} \left( \theta \right) \leq \tcorr / 2$; this is a known property of tree centroids (see, e.g.,~\cite[Lemma~2.1]{persistent_centrality}). 
%As mentioned above, 
We always have that $\Psi_{\tcorr} \left( \theta \right) \leq \tcorr / 2$. 
Typically $\Psi_{\tcorr} \left( \theta \right)$ is on the order $\tcorr$, but with some small probability it can be of smaller order. 
The following definition and lemma quantify this. 

\begin{definition}[The event $\cB$]\label{def:B}
Let $\cB$ denote the following event: 
\[
\cB := \left\{ \frac{\tcorr}{\sqrt{\log \tcorr}} \leq \Psi_{\tcorr} \left( \theta \left( \tcorr \right) \right) \leq \frac{\tcorr}{2} \right\}.
\]
\end{definition}

The event $\cB$ clearly depends on $\tcorr$, but we choose to omit $\tcorr$ from the notation in order to keep notation lighter. Also, as mentioned above, the bound 
$\Psi_{\tcorr} \left( \theta \right) \leq \tcorr / 2$ always holds, but we still include it in the definition of $\cB$ just for clarity. 

\begin{lemma}\label{lem:B}
Let $\left\{ T_{n} \right\}_{n \geq 2}$ be a sequence of trees started from the seed $S$ and grown according to PA or UA. 
There exists a finite constant $C$ such that 
for every $\tcorr \geq 2$ we have that 
\begin{equation}\label{eq:notB}
\p \left( \cB^{c} \right) \leq \frac{C}{\log^{1/4} \left( \tcorr \right)},
\end{equation}
where $\cB^{c}$ denotes the complement of $\cB$. 
\end{lemma}

The bound in~\eqref{eq:notB} can be improved to $C / \sqrt{\log \tcorr}$ for UA trees, but we choose to have a unified theorem for PA and UA trees for simplicity. 
The proof of Lemma~\ref{lem:B} is deferred to Section~\ref{sec:coarse_nice_event_proof}.

\subsection{First moment estimate}\label{sec:coarse_first_moment}

In this subsection we prove the following first moment estimate. 
\begin{lemma}\label{lem:coarse_first_moment} 
Let $\left( T_{n}^{1}, T_{n}^{1} \right) \sim \CPA \left( n, \tcorr, S \right)$. 
For all $\tcorr$ large enough we have that 
\begin{equation}\label{eq:coarse_first_moment_UB}
\limsup_{n \to \infty} \E \left[ Y_{n} \mathbf{1}_{\cA^{1} \cap \cA^{2}} \right] 
\leq 
\frac{1 + 3 \tcorr^{-1/3}}{\tcorr}.
\end{equation}
The same bound holds also when $\left( T_{n}^{1}, T_{n}^{1} \right) \sim \CUA \left( n, \tcorr, S \right)$.
\end{lemma} 

We note that a matching lower bound (of the form $(1-o(1))/\tcorr$ as $\tcorr \to \infty$) also holds, but since we will not use that direction, we do not give details here.

\begin{proof}
We condition on the tree $T_{\tcorr}$ at time $\tcorr$; by the tower rule we have that 
\begin{equation}\label{eq:coarse_UB_tower_rule}
\E \left[ Y_{n} \mathbf{1}_{\cA^{1} \cap \cA^{2}} \right] 
= \E \left[ \E \left[ Y_{n} \mathbf{1}_{\cA^{1} \cap \cA^{2}} \, \middle| \, T_{\tcorr} \right] \right].
\end{equation}
Now given $T_{\tcorr}$, 
observe that property~\ref{prop:nice3} in Definition~\ref{def:nice_event} 
implies that on the event $\cA^{1} \cap \cA^{2}$ we have that 
\[
X_{n}^{1} \left( 1 - X_{n}^{1} \right) 
\geq \tfrac{1}{\tcorr} \Psi_{\tcorr} \left( \theta \right) 
\left( 1 - \tfrac{1}{\tcorr} \Psi_{\tcorr} \left( \theta \right) \right) 
\left( 1 - \tcorr^{-1/3} \right)^{2}
\]
for $n \geq \tcorr$. Plugging this inequality into the definition of $Y_{n}$ we obtain that 
\begin{align*}
\E \left[ Y_{n} \mathbf{1}_{\cA^{1} \cap \cA^{2}} \, \middle| \, T_{\tcorr} \right] 
&\leq \frac{\E \left[ \left( X_{n}^{1} - X_{n}^{2} \right)^{2} \mathbf{1}_{\cA^{1} \cap \cA^{2}} \, \middle| \, T_{\tcorr} \right]}{2 \cdot \tfrac{1}{\tcorr} \Psi_{\tcorr} \left( \theta \right) 
\left( 1 - \tfrac{1}{\tcorr} \Psi_{\tcorr} \left( \theta \right) \right) 
\left( 1 - \tcorr^{-1/3} \right)^{2}} \\
&\leq \frac{\E \left[ \left( Z_{n}^{1} - Z_{n}^{2} \right)^{2} \, \middle| \, T_{\tcorr} \right]}{2 \cdot \tfrac{1}{\tcorr} \Psi_{\tcorr} \left( \theta \right) 
\left( 1 - \tfrac{1}{\tcorr} \Psi_{\tcorr} \left( \theta \right) \right) 
\left( 1 - \tcorr^{-1/3} \right)^{2}},
\end{align*}
where the second inequality follows by observing that on the event $\cA^{1} \cap \cA^{2}$ we have that $X_{n}^{i} = Z_{n}^{i}$ for $i \in \left\{ 1, 2 \right\}$, 
and then removing the indicator to get an upper bound. 
Taking the limit as $n \to \infty$ and applying the bounded convergence theorem 
we obtain that 
\begin{equation}\label{eq:coarse_first_moment_lim_UB}
\limsup_{n \to \infty} \E \left[ Y_{n} \mathbf{1}_{\cA^{1} \cap \cA^{2}} \, \middle| \, T_{\tcorr} \right] 
\leq 
\frac{\E \left[ \left( Z^{1} - Z^{2} \right)^{2} \, \middle| \, T_{\tcorr} \right]}{2 \cdot \tfrac{1}{\tcorr} \Psi_{\tcorr} \left( \theta \right) 
\left( 1 - \tfrac{1}{\tcorr} \Psi_{\tcorr} \left( \theta \right) \right) 
\left( 1 - \tcorr^{-1/3} \right)^{2}}.
\end{equation}
Now using conditional independence, the limiting conditional distribution obtained in~\eqref{eq:Z}, and 
plugging in the variance of the beta distribution, 
we have that 
\[
\E \left[ \left( Z^{1} - Z^{2} \right)^{2} \, \middle| \, T_{\tcorr} \right] 
= 2 \Var \left( Z \, \middle| \, T_{\tcorr} \right) 
= 
\begin{cases} 
\frac{2 \Psi_{\tcorr} \left( \theta \right) \left( \tcorr - \Psi_{\tcorr} \left( \theta \right) \right)}{\tcorr^{2} \left( \tcorr + 1 \right)} 
&\text{ for } \UA, \\
\frac{2 \left( \Psi_{\tcorr} \left( \theta \right) - 1/2 \right) \left( \tcorr - \Psi_{\tcorr} \left( \theta \right) - 1/2 \right)}{\left( \tcorr - 1 \right)^{2} \tcorr} 
&\text{ for } \PA. 
\end{cases}
\]
Plugging these formulas into~\eqref{eq:coarse_first_moment_lim_UB}, we obtain, for both PA and UA trees, that 
\[
\limsup_{n \to \infty} \E \left[ Y_{n} \mathbf{1}_{\cA^{1} \cap \cA^{2}} \, \middle| \, T_{\tcorr} \right] 
\leq 
\left( 1 + \tfrac{1}{\tcorr - 1} \right)^{2} \left( 1 - \tcorr^{-1/3} \right)^{-2} \frac{1}{\tcorr} 
\leq \frac{1 + 3 \tcorr^{-1/3}}{\tcorr},
\]
where the second inequality holds for all $\tcorr$ large enough. 
Since this holds for any tree $T_{\tcorr}$, taking an expectation and using~\eqref{eq:coarse_UB_tower_rule} we arrive at~\eqref{eq:coarse_first_moment_UB}. 
\end{proof}

\subsection{Putting everything together: proof of Theorem~\ref{thm:estimation_coarse_inverse}}\label{sec:coarse_conclusion}

\begin{proof}[Proof of Theorem~\ref{thm:estimation_coarse_inverse}]
We start with the upper bound, which is a consequence of 
%the upper bound~\eqref{eq:coarse_first_moment_UB} in 
Lemma~\ref{lem:coarse_first_moment} and Markov's inequality. First, by a union bound we have that 
\[
\p \left( Y_{n} \geq \frac{\log \tcorr}{\tcorr} \right) 
\leq \p \left( \left( \cA^{1} \cap \cA^{2} \right)^{c} \right) 
+ \p \left( \left\{ Y_{n} \geq \frac{\log \tcorr}{\tcorr} \right\} \cap \cA^{1} \cap \cA^{2} \right).
\]
By a union bound and Lemma~\ref{lem:nice_event} 
we have that the first term is at most $C / \log \tcorr$ for some constant~$C$, 
and so it remains to deal with the second term. 
By Markov's inequality we have that 
\[
\p \left( \left\{ Y_{n} \geq \frac{\log \tcorr}{\tcorr} \right\} \cap \cA^{1} \cap \cA^{2} \right) 
\leq 
\p \left( Y_{n} \mathbf{1}_{\cA^{1} \cap \cA^{2}} \geq \frac{\log \tcorr}{\tcorr} \right) 
\leq 
\frac{\tcorr}{\log \tcorr} \E \left[ Y_{n} \mathbf{1}_{\cA^{1} \cap \cA^{2}} \right].
\]
By~\eqref{eq:coarse_first_moment_UB} we thus have that 
\[
\limsup_{n \to \infty} \p \left( Y_{n} \geq \frac{\log \tcorr}{\tcorr} \right) 
\leq 
\frac{C+2}{\log \tcorr}
\]
for all $\tcorr$ large enough. This expression goes to zero as $\tcorr \to \infty$, which concludes the proof of the upper bound. 

We now turn to the lower bound. To abbreviate notation, 
we introduce $\delta_{\tcorr} := \left( \tcorr \log \tcorr \right)^{-1/2}$. 
Our goal is to show that 
\[
\lim_{\tcorr \to \infty} \limsup_{n \to \infty} 
\p \left( Y_{n} \leq \delta_{\tcorr}^{2} \right) = 0.
\]
Since $Y_{n} \leq \delta_{\tcorr}^{2}$ implies that 
$\left| X_{n}^{1} - X_{n}^{2} \right| \leq \delta_{\tcorr}$, 
we have that 
\[
\p \left( Y_{n} \leq \delta_{\tcorr}^{2} \right) 
\leq \p \left( \left| X_{n}^{1} - X_{n}^{2} \right| \leq \delta_{\tcorr} \right).
\]
By a union bound we have that 
\[
\p \left( \left| X_{n}^{1} - X_{n}^{2} \right| \leq \delta_{\tcorr} \right) 
\leq 
\p \left( \left( \cA^{1} \right)^{c} \right) 
+ \p \left( \left( \cA^{2} \right)^{c} \right) 
+ \p \left( \cB^{c} \right) 
+ \p \left( \left\{ \left| X_{n}^{1} - X_{n}^{2} \right| \leq \delta_{\tcorr} \right\} \cap \cA^{1} \cap \cA^{2} \cap \cB \right).
\]
By Lemmas~\ref{lem:nice_event} and~\ref{lem:B}, 
there exists a finite constant $C$ such that 
the first three terms in the display above are bounded above by $C / \log^{1/4} \left( \tcorr \right)$. 
Since this goes to zero as $\tcorr \to \infty$, what remains is to bound the last term in the display above. 
To do this, we first condition on the tree $T_{\tcorr}$. 
By the tower rule, using also the fact that the event $\cB$ is measurable with respect to $T_{\tcorr}$, we have that 
\begin{equation}\label{eq:prob_X1_X2_tower}
\p \left( \left\{ \left| X_{n}^{1} - X_{n}^{2} \right| \leq \delta_{\tcorr} \right\} \cap \cA^{1} \cap \cA^{2} \cap \cB \right) 
= 
\E \left[ \E \left[ \mathbf{1}_{\left\{ \left| X_{n}^{1} - X_{n}^{2} \right| \leq \delta_{\tcorr} \right\}} \mathbf{1}_{\cA^{1} \cap \cA^{2}} \, \middle| \, T_{\tcorr} \right] \mathbf{1}_{\cB} \right].
\end{equation}
We now fix $T_{\tcorr}$ and study the conditional expectation 
$\E \left[ \mathbf{1}_{\left\{ \left| X_{n}^{1} - X_{n}^{2} \right| \leq \delta_{\tcorr} \right\}} \mathbf{1}_{\cA^{1} \cap \cA^{2}} \, \middle| \, T_{\tcorr} \right]$. 
Recall that on the event $\cA^{1} \cap \cA^{2}$ we have that $X_{n}^{i} = Z_{n}^{i}$ for $i \in \left\{ 1, 2 \right\}$ and $n \geq \tcorr$. 
Therefore by the bounded convergence theorem we have that 
\begin{align}
\limsup_{n \to \infty} \E \left[ \mathbf{1}_{\left\{ \left| X_{n}^{1} - X_{n}^{2} \right| \leq \delta_{\tcorr} \right\}} \mathbf{1}_{\cA^{1} \cap \cA^{2}} \, \middle| \, T_{\tcorr} \right] 
&= 
\E \left[ \mathbf{1}_{\left\{ \left| Z^{1} - Z^{2} \right| \leq \delta_{\tcorr} \right\}} \mathbf{1}_{\cA^{1} \cap \cA^{2}} \, \middle| \, T_{\tcorr} \right] \notag \\
&\leq \E \left[ \mathbf{1}_{\left\{ \left| Z^{1} - Z^{2} \right| \leq \delta_{\tcorr} \right\}} \, \middle| \, T_{\tcorr} \right], \label{eq:cond_prob_Z1_Z2}
\end{align}
where the inequality follows by dropping the second indicator. 
For notational convenience, and in order to treat the cases of PA and UA trees simultaneously, we introduce 
\begin{equation}\label{eq:ab}
\left( a, b \right) := 
\begin{cases} 
\left( \Psi_{\tcorr} \left( \theta \right), \tcorr - \Psi_{\tcorr} \left( \theta \right) \right) &\text{ for } \UA, \\
\left( \Psi_{\tcorr} \left( \theta \right) - \frac{1}{2}, \tcorr - \Psi_{\tcorr} \left( \theta \right) - \frac{1}{2} \right) &\text{ for } \PA. 
\end{cases}
\end{equation}
Recall from~\eqref{eq:Z} that, conditioned on $T_{\tcorr}$, the random variables $Z^{1}$ and $Z^{2}$ are i.i.d.\ $\Beta \left( a, b \right)$ random variables. 
To bound the expression in~\eqref{eq:cond_prob_Z1_Z2}, we first condition on $Z^{1}$. By the tower rule, we have that 
\[
\E \left[ \mathbf{1}_{\left\{ \left| Z^{1} - Z^{2} \right| \leq \delta_{\tcorr} \right\}} \, \middle| \, T_{\tcorr} \right] 
= 
\E \left[ \E \left[ \mathbf{1}_{\left\{ \left| Z^{1} - Z^{2} \right| \leq \delta_{\tcorr} \right\}} \, \middle| \, Z^{1}, T_{\tcorr} \right] \, \middle| \, T_{\tcorr} \right].
\]
Conditioned on $Z^{1}$ and $T_{\tcorr}$, we have that $Z^{2} \sim \Beta \left( a, b \right)$, so we can compute this conditional expectation explicitly: 
\begin{equation}\label{eq:beta_density}
\E \left[ \mathbf{1}_{\left\{ \left| Z^{1} - Z^{2} \right| \leq \delta_{\tcorr} \right\}} \, \middle| \, Z^{1}, T_{\tcorr} \right] 
= 
\frac{1}{B\left( a, b\right)} \int_{\left( Z^{1} - \delta_{\tcorr} \right) \vee 0}^{\left( Z^{1} + \delta_{\tcorr} \right) \wedge 1} x^{a-1} \left( 1 - x \right)^{b-1} dx,
\end{equation}
where $B(a,b) = \Gamma \left( a \right) \Gamma \left( b \right) / \Gamma \left( a + b \right)$ is the beta function. 
Recall from~\eqref{eq:prob_X1_X2_tower} that we only care about bounding this expression when the event $\cB$ holds. 
From the definition of $\cB$, and also the definitions of $a$ and $b$ (see~\eqref{eq:ab}), it follows that if $\cB$ holds, then $a, b > 2$ for all $\tcorr$ large enough. 
We know that if $a,b > 1$, then the mode of the $\Beta(a,b)$ distribution is at $\tfrac{a-1}{a+b-2}$. Plugging this into~\eqref{eq:beta_density}, we obtain, for all $\tcorr$ large enough, that 
\begin{equation}\label{eq:beta_density_bound}
\E \left[ \mathbf{1}_{\left\{ \left| Z^{1} - Z^{2} \right| \leq \delta_{\tcorr} \right\}} \, \middle| \, Z^{1}, T_{\tcorr} \right] \mathbf{1}_{\cB}
\leq \frac{2\delta_{\tcorr}}{B \left( a, b \right)} 
\left( \frac{a-1}{a+b-2} \right)^{a-1} \left( \frac{b-1}{a+b-2} \right)^{b-1} \mathbf{1}_{\cB}.
\end{equation}
Now using the standard inequalities 
$\sqrt{2\pi} n^{n+1/2} e^{n} \leq n! \leq e n^{n+1/2} e^{n}$, 
which hold for all $n \geq 1$, we have that 
\[
B(a,b) = \frac{(a-1)! (b-1)!}{(a+b-1)!} 
\geq 
2\pi 
\left( \frac{a-1}{a+b-1} \right)^{a-1} \left( \frac{b-1}{a+b-1} \right)^{b-1}
\frac{\sqrt{\left( a - 1 \right) \left( b - 1 \right)}}{\left( a + b - 1 \right)^{3/2}}.
\]
Therefore 
\begin{align*}
\frac{1}{B \left( a, b \right)} 
\left( \frac{a-1}{a+b-2} \right)^{a-1} \left( \frac{b-1}{a+b-2} \right)^{b-1}
&\leq \frac{1}{2\pi} \frac{\left( a + b - 1 \right)^{3/2}}{\sqrt{\left( a - 1 \right) \left( b - 1 \right)}} \left( 1 + \frac{1}{a+b-2} \right)^{a+b-2} \\
&\leq \frac{e}{2\pi} \frac{\left( a + b - 1 \right)^{3/2}}{\sqrt{\left( a - 1 \right) \left( b - 1 \right)}}.
\end{align*}
Plugging this back into~\eqref{eq:beta_density_bound}, we obtain, for all $\tcorr$ large enough, that 
\[
\E \left[ \mathbf{1}_{\left\{ \left| Z^{1} - Z^{2} \right| \leq \delta_{\tcorr} \right\}} \, \middle| \, Z^{1}, T_{\tcorr} \right] \mathbf{1}_{\cB} 
\leq C \delta_{\tcorr} \frac{\left( a + b \right)^{3/2}}{\sqrt{ab}} \mathbf{1}_{\cB}
\]
for some constant $C$. 
From~\eqref{eq:ab} we have that $a+b \leq \tcorr$. 
We also have that $b \geq \tcorr / 2 - 1/2$. 
Furthermore, on the event $\cB$ we have that $a \geq \tcorr / \sqrt{\log \tcorr} - 1/2$. 
Altogether these imply that
\[
\frac{\left( a + b \right)^{3/2}}{\sqrt{ab}} \mathbf{1}_{\cB} 
\leq C' \tcorr^{1/2} \log^{1/4} \left( \tcorr \right).
\]
for some constant $C'$ and all $\tcorr$ large enough. 
Plugging this back into the previous display and using the definition of $\delta_{\tcorr}$ we obtain that 
\[
\E \left[ \mathbf{1}_{\left\{ \left| Z^{1} - Z^{2} \right| \leq \delta_{\tcorr} \right\}} \, \middle| \, Z^{1}, T_{\tcorr} \right] \mathbf{1}_{\cB} 
\leq 
\frac{C''}{\log^{1/4} \left( \tcorr \right)} 
\]
for some constant $C''$ and all $\tcorr$ large enough. 
Now taking an expectation over $Z^{1}$ and using~\eqref{eq:prob_X1_X2_tower} and~\eqref{eq:cond_prob_Z1_Z2}, 
we finally obtain that 
\[
\limsup_{n \to \infty} 
\p \left( \left\{ \left| X_{n}^{1} - X_{n}^{2} \right| \leq \delta_{\tcorr} \right\} \cap \cA^{1} \cap \cA^{2} \cap \cB \right) 
\leq 
\frac{C''}{\log^{1/4} \left( \tcorr \right)} 
\]
for all $\tcorr$ large enough. 
This expression goes to zero as $\tcorr \to \infty$, which concludes the proof. 
\end{proof}

\subsection{Proofs of remaining lemmas}
\label{sec:coarse_nice_event_proof}

In this subsection we prove Lemmas~\ref{lem:nice_event} and~\ref{lem:B}, proofs that we have deferred until now. 

\subsubsection{Proof of Lemma~\ref{lem:B}}

We start with the proof of Lemma~\ref{lem:B}, which is relatively short. 

\begin{proof}[Proof of Lemma~\ref{lem:B}]
First, by a union bound we have that 
\begin{equation}\label{eq:B_union_bound}
\p \left( \cB^{c} \right) 
= \p \left( \Psi_{\tcorr} \left( \theta \left( \tcorr \right) \right) < \frac{\tcorr}{\sqrt{\log \tcorr}} \right) 
\leq \sum_{i=1}^{\tcorr} \p \left( \Psi_{\tcorr} \left( v_{i} \right) < \frac{\tcorr}{\sqrt{\log \tcorr}} \right).
\end{equation}
Noting that the term for $i = 1$ is equal to the term for $i = 2$, 
we now fix $i \geq 2$. 
Note that $v_{i}$ is introduced in $T_{i}$. 
Let $w$ denote the neighbor of $v_{i}$ in $T_{i}$. 
By definition we have that 
\[
\Psi_{\tcorr} \left( v_{i} \right) 
= \max_{u \in \cN_{v_{i}} \left( T_{\tcorr} \right)} \left| \left( T_{\tcorr}, v_{i} \right)_{u \downarrow} \right| 
\geq 
\left| \left( T_{\tcorr}, v_{i} \right)_{w \downarrow} \right|
\]
and so---introducing 
$M_{n} := \frac{1}{n} \left| \left( T_{n}, v_{i} \right)_{w \downarrow} \right|$ 
for $n \geq i$ 
in order to abbreviate notation---we have that
\[
\p \left( \Psi_{\tcorr} \left( v_{i} \right) < \frac{\tcorr}{\sqrt{\log \tcorr}} \right) 
\leq 
\p \left( M_{\tcorr} \leq \frac{1}{\sqrt{\log \tcorr}} \right).
\]
This latter probability can be understood using P\'olya urn and martingale arguments. The proofs for PA and UA trees are similar, and we start with UA trees. 
For UA trees, the evolution of the pair 
$ \left( n M_{n}, n - n M_{n} \right)$ 
for $n \geq i$ 
follows a classical P\'olya urn with initial condition $\left( i - 1, 1 \right)$. 
By standard results on P\'olya urns we have that 
$\left\{ M_{n} \right\}_{n \geq i}$ is a martingale,  
the limit $M_{\infty} := \lim_{n\to\infty} M_{n}$ exists almost surely,  
and $M_{\infty} \sim \Beta \left( i - 1, 1 \right)$. 
By this latter property we have that 
\begin{equation}\label{eq:beta_prob}
\p \left( M_{\infty} \leq z \right) = z^{i-1} 
\end{equation}
for all $z \in \left( 0, 1 \right)$. 
Since $\left\{ M_{n} \right\}_{n \geq i}$ is a nonnegative martingale, we also have that 
\begin{equation*}\label{eq:mg_arg}
\p \left( M_{\infty} \leq 2 z \, \middle| \, M_{n} \leq z \right) \geq 1/2 %\frac{1}{2} 
\end{equation*}
for all $z \geq 0$ and $n \geq i$, 
which implies that 
$\p \left( M_{n} \leq z \right) \leq 2 \p \left( M_{\infty} \leq 2 z \right)$. 
%Combining~\eqref{eq:beta_prob} and~\eqref{eq:mg_arg} 
Thus using~\eqref{eq:beta_prob} 
we have that 
%\[
%\p \left( M_{\tcorr} \leq \frac{1}{\sqrt{\log \tcorr}} \right) 
%\leq \frac{\p \left( M_{\infty} \leq \frac{2}{\sqrt{\log \tcorr}} \right)}{\p \left( M_{\infty} \leq \frac{2}{\sqrt{\log \tcorr}} \, \middle| \, M_{\tcorr} \leq \frac{1}{\sqrt{\log \tcorr}} \right)} 
%\leq 2 \left( \frac{2}{\sqrt{\log \tcorr}} \right)^{i-1}
%\]
\[
\p \left( M_{\tcorr} \leq \frac{1}{\sqrt{\log \tcorr}} \right) 
\leq 2 \left( \frac{2}{\sqrt{\log \tcorr}} \right)^{i-1}
\]
for all $\tcorr$ large enough. 
% (specifically, for all $\tcorr$ such that $2 / \sqrt{\log \tcorr} < 1$). 
Plugging this bound back into~\eqref{eq:B_union_bound} and noting that the geometric sum is on the same order as the largest term, we obtain that 
\[
\p \left( \cB^{c} \right) \leq \frac{12}{\sqrt{\log \tcorr}} 
\]
for all $\tcorr$ large enough.

Turning now to PA trees, the evolution of the pair 
$\left( 2n M_{n} - 1, 2n - 2n M_{n} - 1 \right)$ 
for $n \geq i$ 
follows a P\'olya urn with replacement matrix 
$\left(\begin{smallmatrix} 2 & 0 \\ 0 & 2 \end{smallmatrix}\right)$ 
and initial condition 
$\left( 2i - 3, 1 \right)$. 
Define $\wt{M}_{n} := \left( 2n M_{n} - 1 \right) / \left( 2n - 2 \right)$. 
The process 
$\left\{ \wt{M}_{n} \right\}_{n \geq i}$ 
is a bounded martingale and hence its limit as $n \to \infty$ exists almost surely. 
Since 
\[
\wt{M}_{n}
= M_{n} + \frac{1}{n-1} M_{n} - \frac{1}{2n-2} 
\]
and $M_{n} \in \left[ 0, 1 \right]$, 
the limit of the martingale equals the limit of $M_{n}$; 
that is, 
$M_{\infty} := \lim_{n \to \infty} M_{n} 
= \lim_{n \to \infty} \wt{M}_{n}$ 
exists almost surely. 
Furthermore, by standard results on P\'olya urns we know that 
$M_{\infty} \sim \Beta \left( i - 3/2, 1/2 \right)$. 
By this latter property, and using the bound $\left( 1 - z \right)^{-1/2} \leq \sqrt{2}$ for $z \in \left( 0, 1/2 \right)$ in the density function of the beta distribution, we have that 
\[
\p \left( M_{\infty} \leq z \right) 
\leq \frac{\sqrt{2}}{\left( i - \frac{3}{2} \right) B \left( i - \frac{3}{2}, \frac{1}{2} \right)} z^{i - 3/2}
\]
for all $z \in \left( 0, 1/2 \right)$. 
We can further bound this quantity using properties of the Gamma function. 
Specifically, we use the following identities: 
$\Gamma \left( z + 1 \right) = z \Gamma \left( z \right)$, 
for a positive integer $n$ we have that $\Gamma \left( n + 1 \right) = n!$ 
and also that $\Gamma \left( n + 1/2 \right) = \frac{(2n)! \sqrt{\pi}}{4^{n} n!}$, 
and finally that $\Gamma \left( 1 / 2 \right) = \sqrt{\pi}$. Using these we have that 
\begin{equation}\label{eq:beta}
\left( i - \frac{3}{2} \right) B \left( i - \frac{3}{2}, \frac{1}{2} \right) 
=  \frac{\left( i - \frac{3}{2} \right) \Gamma \left( i - \frac{3}{2} \right) \Gamma \left( \frac{1}{2} \right)}{\Gamma \left( i - 1 \right)} 
%= \frac{\sqrt{\pi}}{\left( i - 2 \right)!} \cdot \frac{\left( 2i - 2 \right)! \sqrt{\pi}}{4^{i-1} \left( i - 1 \right)!}
= \pi \left( i - 1 \right) \binom{2i-2}{i-1} 4^{-i + 1} 
\geq 4^{-i+1}.
\end{equation}
Plugging this back into the previous display we obtain that 
\[
\p \left( M_{\infty} \leq z \right) \leq 4 \left( 4z \right)^{i-3/2} 
\]
for all $z \in \left( 0, 1/2 \right)$. 
Using the fact that $\wt{M}_{n} \leq 2 M_{n}$, together with the same martingale argument as before, we have that 
\[
\p \left( M_{n} \leq z \right) 
\leq 
\p \left( \wt{M}_{n} \leq 2 z \right) 
\leq 
2 \p \left( M_{\infty} \leq 4 z \right).
\]
The previous two displays combined imply that 
$\p \left( M_{n} \leq z \right) \leq 8 \left( 16 z \right)^{i-3/2}$ for all $z \in \left( 0, 1/8 \right)$ and $n \geq i$. We have thus obtained that 
\[
\p \left( M_{\tcorr} \leq \frac{1}{\sqrt{\log \tcorr}} \right) 
\leq 
8 \left( \frac{16}{\sqrt{\log \tcorr}} \right)^{i-3/2}
\]
for all $\tcorr$ large enough. 
Plugging this bound back into~\eqref{eq:B_union_bound} and noting that the geometric sum is on the same order as the largest term, we obtain the desired bound~\eqref{eq:notB}. 
\end{proof}

\subsubsection{Proof of Lemma~\ref{lem:nice_event}}

We now turn to the proof of Lemma~\ref{lem:nice_event}, which is more involved. 
We start by stating and proving a few auxiliary lemmas that we will use. 

The following lemma gives us an exponential bound on the probability that a vertex of large timestamp ever becomes the centroid. This was proved in~\cite{persistent_centrality}; see their Lemmas~A.1 and~3.1.
\begin{lemma}\label{lem:persistent_centroid}
Consider a sequence of PA or UA trees started from the seed $S = S_{2}$. 
For all $t$ large enough we have that 
\[
\p \left( v_{t+1} \text{ becomes at least as central as } \theta \left( t \right) \text{ at some future time} \right) 
\leq \frac{P(t/2)}{2^{t/2}},
\]
where $P$ is a fixed polynomial. 
\end{lemma}

The following lemma is useful in studying the relative (anti-)centralities of two vertices by examining the growth of specific subtrees. 

\begin{lemma}\label{lem:from_centralities_to_subtrees}
Let $\left\{ T_{n} \right\}_{n \geq 2}$ be a sequence of growing trees (such as PA or UA trees), where at every time step a single vertex is added to the tree, together with a single edge. Let $v_{1}, v_{2}, v_{3}, \ldots$ denote the vertices in order of appearance. Fix $t$ and let $i$ and $j$ be distinct positive integers such that $i, j \leq t$. 
Suppose that 
\begin{equation}\label{eq:assumption_begin}
\Psi_{t} \left( v_{i} \right) > \Psi_{t} \left( v_{j} \right) 
\end{equation}
and that there exists $N > t$ such that 
\begin{equation}\label{eq:assumption_end}
\Psi_{N} \left( v_{i} \right) \leq \Psi_{N} \left( v_{j} \right). 
\end{equation}
Then there must exist $M$ such that $t < M \leq N$ and 
\[
\left| \left( T_{M}, v_{i} \right)_{v_{j} \downarrow} \right| 
= 
\left| \left( T_{M}, v_{j} \right)_{v_{i} \downarrow} \right|.
\]
\end{lemma}

\begin{proof} 
We start with some notation. Fix $n \geq t$ and consider the tree $T_{n}$. 
Let $a_{1}, a_{2}, a_{3}, \ldots$ denote the sizes of the pendent subtrees of $v_{i}$, excluding the subtree that contains $v_{j}$. 
Similarly, let $b_{1}, b_{2}, b_{3}, \ldots$ denote the sizes of the pendent subtrees of $v_{j}$, excluding the subtree that contains $v_{i}$. 
Finally, let $c$ denote the number of vertices that are ``in between'' $v_{i}$ and $v_{j}$; 
that is, $c$ is the number of vertices $u$ such that the path from $u$ to $v_{i}$ does not contain $v_{j}$ and the path from $u$ to $v_{j}$ does not contain $v_{i}$. 
Note that with this notation we have the following equalities: 
\begin{align*}
\Psi_{n} \left( v_{i} \right) &= \max \left\{ c + 1 + \sum_{\ell} b_{\ell}, a_{1}, a_{2}, a_{3}, \ldots \right\}, \\
\Psi_{n} \left( v_{j} \right) &= \max \left\{ c + 1 + \sum_{k} a_{k}, b_{1}, b_{2}, b_{3}, \ldots \right\}, 
\end{align*}
and also 
\begin{equation}\label{eq:subtrees_sum}
\left| \left( T_{n}, v_{i} \right)_{v_{j} \downarrow} \right| = 1 + \sum_{\ell} b_{\ell}, 
\qquad \qquad 
\left| \left( T_{n}, v_{j} \right)_{v_{i} \downarrow} \right| = 1 + \sum_{k} a_{k}.
\end{equation}
We now claim that if 
$\Psi_{n} \left( v_{i} \right) > \Psi_{n} \left( v_{j } \right)$, 
then 
$\sum_{\ell} b_{\ell} > \sum_{k} a_{k}$. 
We prove this by contradiction; suppose that 
$\sum_{\ell} b_{\ell} \leq \sum_{k} a_{k}$. 
Then 
$\Psi_{n} \left( v_{j} \right) \geq c + 1 + \sum_{k} a_{k} \geq c + 1 + \sum_{\ell} b_{\ell}$, 
so $\Psi_{n} \left( v_{i} \right) > \Psi_{n} \left( v_{j } \right)$ 
implies that 
$\Psi_{n} \left( v_{i} \right) = a_{k'}$ for some $k'$. 
But then $\Psi_{n} \left( v_{j} \right) \geq c + 1 + \sum_{k} a_{k} \geq 1 + a_{k'} > \Psi_{n} \left( v_{i} \right)$, which is a contradiction. 
The same argument shows that if 
$\Psi_{n} \left( v_{i} \right) \geq \Psi_{n} \left( v_{j } \right)$, 
then 
$\sum_{\ell} b_{\ell} \geq \sum_{k} a_{k}$. 
As a corollary, we have that 
if 
$\Psi_{n} \left( v_{i} \right) = \Psi_{n} \left( v_{j } \right)$, 
then 
$\sum_{\ell} b_{\ell} = \sum_{k} a_{k}$.

Altogether, using~\eqref{eq:subtrees_sum}, we have shown that 
\begin{equation}\label{eq:signs_matching}
\sgn \left( \Psi_{n} \left( v_{i} \right) - \Psi_{n} \left( v_{j} \right) \right) 
= 
\sgn \left( \left| \left( T_{n}, v_{i} \right)_{v_{j} \downarrow} \right| - \left| \left( T_{n}, v_{j} \right)_{v_{i} \downarrow} \right| \right),
\end{equation}
where $\sgn$ is the sign function: $\sgn(x) = - 1$ if $x < 0$, $\sgn(x) = 0$ if $x=0$, and $\sgn(x) = 1$ if $x > 0$. 
Observe also that the integer-valued quantity 
\[
f(n) := \left| \left( T_{n}, v_{i} \right)_{v_{j} \downarrow} \right| - \left| \left( T_{n}, v_{j} \right)_{v_{i} \downarrow} \right|
\]
changes by $1$, $0$, or $-1$ as $n$ increases by one. 
The assumption~\eqref{eq:assumption_begin}, together with~\eqref{eq:signs_matching}, implies that $f(t) > 0$. 
The assumption~\eqref{eq:assumption_end}, together with~\eqref{eq:signs_matching}, implies that $f(N) \leq 0$. 
Therefore, by the previous observation, there must exist $M \in \left\{ t+1, \ldots, N \right\}$ such that $f(M) = 0$. 
\end{proof}

The following lemma gives concentration bounds for P\'olya urns. 

\begin{lemma}\label{lem:Polya_concentration}
Let $\left\{ \left( A_{n}, B_{n} \right) \right\}_{n \geq 0}$ be a stochastic process with a deterministic initial condition satisfying $A_{0}, B_{0} \geq 1$, 
and let $k := A_{0} + B_{0}$. 

If $\left\{ \left( A_{n}, B_{n} \right) \right\}_{n \geq 0}$ evolves as a classical P\'olya urn, then for any $\eps > 0$ we have that 
\[
\p \left( \exists\, n \geq 0 : \left| \frac{A_{n}}{k+n} - \frac{A_{0}}{k} \right| \geq \eps \right) 
\leq 
2 \exp \left( - k \eps^{2} / 2 \right).
\]

If $\left\{ \left( 2 A_{n} - 1, 2 B_{n} - 1 \right) \right\}_{n \geq 0}$ evolves as a P\'olya urn with replacement matrix 
$\left(\begin{smallmatrix} 2 & 0 \\ 0 & 2 \end{smallmatrix}\right)$, 
then for any $\eps \geq 2 / (k-1)$ we have that 
\[
\p \left( \exists\, n \geq 0 : \left| \frac{A_{n}}{k+n} - \frac{A_{0}}{k} \right| \geq \eps \right) 
\leq 
2 \exp \left( - (k-1) \eps^{2} / 8 \right).
\]
\end{lemma}

\begin{proof}
We start with the first claim. Defining $M_{n} := A_{n} / (k+n)$, we have that $\left\{ M_{n} \right\}_{n \geq 0}$ is a martingale. 
The martingale differences satisfy 
$\left| M_{n} - M_{n-1} \right| \leq 1 / (k+n)$ 
for every $n \geq 1$. 
Therefore by the maximal version of Azuma's inequality we have for every $\eps > 0$ that 
\[
\p \left( \sup_{n \geq 0} \left| M_{n} - M_{0} \right| \geq \eps \right) 
\leq 
2 \exp \left( - \frac{\eps^{2}}{2 \sum_{n \geq 1} \left( k+n \right)^{-2}} \right).
\] 
The claim follows from the fact that 
$\sum_{n \geq 1} \left( k+n \right)^{-2} \leq 1/k$.

Turning to the second claim, first note that again $A_{n} + B_{n} = k + n$ for every $n \geq 0$. 
Define $\wt{M}_{n} := \left( 2 A_{n} - 1 \right) / \left( 2 A_{n} + 2 B_{n} - 2 \right) 
= \left( 2 A_{n} - 1 \right) / \left( 2 \left( k + n - 1 \right) \right)$ 
and observe that the process $\{ \wt{M}_{n} \}_{n \geq 0}$ is a martingale. Furthermore, the martingale differences satisfy 
$\left| \wt{M}_{n} - \wt{M}_{n-1} \right| \leq 1 / (k+n-1)$ 
for every $n \geq 1$. Therefore by the same argument as above we have for every $\eps > 0$ that 
\[
\p \left( \sup_{n \geq 0} \left| \wt{M}_{n} - \wt{M}_{0} \right| \geq \eps \right) 
\leq 
2 \exp \left( - \left( k - 1 \right)\eps^{2} / 2 \right).
\]
Now observe that 
$\left| \left( M_{n} - M_{0} \right) - \left( \wt{M}_{n} - \wt{M}_{0} \right) \right| \leq 1 / \left( k - 1 \right)$, so by the triangle inequality we have that 
\[
\p \left( \sup_{n \geq 0} \left| M_{n} - M_{0} \right| \geq \eps \right) 
\leq
\p \left( \sup_{n \geq 0} \left| \wt{M}_{n} - \wt{M}_{0} \right| \geq \eps - 1/ \left( k - 1 \right)\right) 
\leq
\p \left( \sup_{n \geq 0} \left| \wt{M}_{n} - \wt{M}_{0} \right| \geq \eps / 2 \right) 
\]
for any $\eps \geq 2 / (k-1)$. 
The result follows by putting the previous two displays together. 
\end{proof}

Finally, the following lemma gives a tail bound for degrees in PA and UA trees. 

\begin{lemma}\label{lem:deg_tail_bound}
Let $\left\{ T_{n} \right\}_{n \geq 2}$ be a sequence of trees started from the seed $S = S_{2}$ and grown according to PA or UA. Let $v_{1}, v_{2}, v_{3}, \ldots$ denote the vertices in order of appearance. Let $d_{n} (v)$ denote the degree of $v$ in $T_{n}$. 
There exists a positive constant $c$ such that for every $1 \leq i \leq n$ we have that 
\[
\p \left( d_{n} \left( v_{i} \right) \geq \sqrt{n} \log^{2} \left( n \right) \right) 
\leq 
\exp \left( - c \log^{3} \left( n \right) \right). 
\]
\end{lemma}

\begin{proof} 
The vertex $v_{3}$ attaches to either $v_{1}$ or $v_{2}$; without loss of generality, assume that it attaches to $v_{1}$, that is, $v_{1}$ has degree $2$ in $T_{3}$. 
For both PA and UA trees, $d_{n}(v_{1})$ stochastically dominates $d_{n} \left( v_{i} \right)$ for $1 < i \leq n$, so it suffices to prove the claim for $v_{1}$. 
Furthermore, the random variable $d_{n}(v_{1})$ in a PA tree stochastically dominates the random variable $d_{n}(v_{1})$ in a UA tree, hence it suffices to prove the claim for PA trees. 

For $n \geq 3$ let $M_{n} := d_{n}(v_{1}) / \sqrt{n-2}$. Observe that 
\[
\E \left[ d_{n+1} \left( v_{1} \right) \, \middle| \, d_{n} \left( v_{1} \right) \right] 
= 
\left( 1 + \frac{1}{2n-2} \right) d_{n} \left( v_{1} \right).
\]
Since $\left( 1 + 1 / (2n-2) \right) / \sqrt{n-1} \leq 1 / \sqrt{n-2}$ for every $n\geq 3$, 
it follows that $\left\{ M_{n} \right\}_{n \geq 3}$ is a supermartingale. 
Also, 
% the martingale differences satisfy 
$\left| M_{n} - M_{n-1} \right| \leq 1 / \sqrt{n-1}$. 
Thus by Azuma's inequality for supermartingales, noting that $M_{3} = 2$, we have  for every $\lambda > 0$ that 
\[
\p \left( \frac{d_{n} \left( v_{1} \right)}{\sqrt{n-2}}  - 2 \geq \lambda \right) 
\leq 
\exp \left( - \frac{\lambda^{2}}{2 \sum_{i=4}^{n} 1 / (i-1)} \right) 
\leq 
\exp \left( - \frac{\lambda^{2}}{2 \log n} \right).
\]
Plugging in 
$\lambda = \log^{2} \left( n \right)$ yields the desired claim. 
\end{proof}

We are now ready to prove Lemma~\ref{lem:nice_event}. 

\begin{proof}[Proof of Lemma~\ref{lem:nice_event}]
We divide the proof into six steps. In the following we informally call a vertex an ``early'' vertex if its timestamp is at most $\log \tcorr$. 

\textbf{Step 1:} \emph{The centroid is an early vertex.} 

For a fixed $i \geq 1$, 
let $\cA_{1} \left( i \right)$ denote the event that $v_{i}$ never becomes a centroid during the whole process; that is, the event that $v_{i}$ is not a centroid in $T_{s}$ for any $s \geq i$. Define 
\[
\cA_{1} := \bigcap_{i > \log \tcorr} \cA_{1} \left( i \right). 
\]
An immediate consequence of Lemma~\ref{lem:persistent_centroid} is that 
$\p \left( \cA_{1} \left( i \right)^{c} \right) \leq \exp \left( - i / 3 \right)$ for all $i$ large enough. 
So by a union bound we have, for all $\tcorr$ large enough that 
\[
\p \left( \cA_{1}^{c} \right) 
\leq \sum_{i > \log \tcorr} \p \left( \cA_{1} \left( i \right)^{c} \right) 
\leq \sum_{i > \log \tcorr} e^{-i/3} 
\leq \frac{4}{\tcorr^{1/3}}. 
\]
%for all $\tcorr$ large enough.

\textbf{Step 2:} \emph{Early subtrees are large in $T_{\tcorr}$.} 

This is an important intermediate step towards the overarching goal of characterizing the centroid. 
Specifically, 
the consequence of early subtree sizes being large is that then many of the random variables we will consider in future steps will be ``stable'' in timesteps $t \geq \tcorr$. 

For distinct positive integers $i, j \leq \log \tcorr$, we will show that subtrees of the form 
$\left( T_{\tcorr}, v_{i} \right)_{v_{j} \downarrow}$ are large. 
Formally, for distinct positive integers $i, j \leq \log \tcorr$, define the event 
\[
\cE_{2} \left( i, j \right) := \left\{ 
\left| \left( T_{\tcorr}, v_{i} \right)_{v_{j} \downarrow} \right| 
\geq \frac{\tcorr}{\log^{7} \left( \tcorr \right)} 
\right\}.
\]
We proceed by bounding the probability of the complement of 
$\cE_{2} \left( i, j \right)$, using arguments similar to those found in the proof of Lemma~\ref{lem:B}. Since the details are repetitive, we only give the final bounds and leave the details to the reader. 

%First, consider the special case when $i = 1$ and $j = 2$. 
%We start with UA trees. Let $U \sim \Uni[0,1]$. 
%We then have that 
%\[
%\p \left( \cE_{2} \left( 1, 2 \right)^{c} \right) 
%\leq 2 \p \left( U \leq \frac{2}{\log^{6} \left( \tcorr \right)} \right) 
%= \frac{4}{\log^{6} \left( \tcorr \right)}.
%\]
%Turning to PA trees, let $\varphi \sim \Beta \left( 1/2, 1/2 \right)$. 
%Using the fact that $\p \left( \varphi \leq z \right) \leq \sqrt{z}$ for $z \in \left(0,1/2 \right)$, 
%we have that 
%\[
%\p \left( \cE_{2} \left( 1, 2 \right)^{c} \right) 
%\leq 2 \p \left( \varphi \leq \frac{4}{\log^{6} \left( \tcorr \right)} \right) 
%\leq \frac{4}{\log^{3} \left( \tcorr \right)} 
%\]
%for all $\tcorr$ large enough. 
%
%We now turn to general distinct positive integers $i$ and $j$. 
%Noting that for $j \geq 3$ we have that 
%$\p \left( \cE_{2} \left( 1, j \right) \right) 
%= \p \left( \cE_{2} \left( 2, j \right) \right)$ 
%due to symmetry, 
%assume in the following that 
%$2 \leq i < j \leq \log \tcorr$. 

Assume in the following that 
$1 \leq i < j \leq \log \tcorr$. 
We start with UA trees. 
Let $\varphi_{j} \sim \Beta \left( 1, j - 1 \right)$. 
Then, by combining P\'olya urn and martingale arguments as in the proof of Lemma~\ref{lem:B}, 
we have for every $z \in [0,1]$ that 
\[
\max \left\{ \p \left( \frac{1}{\tcorr} \left| \left( T_{\tcorr}, v_{i} \right)_{v_{j} \downarrow} \right| \leq z \right), \p \left( \frac{1}{\tcorr} \left| \left( T_{\tcorr}, v_{j} \right)_{v_{i} \downarrow} \right| \leq z \right) \right\} 
\leq 
2 \p \left( \varphi_{j} \leq 2 z \right).
\]
For every $z \in [0,1]$ we have that 
\[
\p \left( \varphi_{j} \leq z \right) = (j-1) \int_{0}^{z} \left( 1 - x \right)^{j-2} dx 
\leq \left( j - 1 \right) z. 
\]
Combining the previous two displays and using the fact that 
$j \leq \log \tcorr$, we have that 
\[
\max \left\{ \p \left( \cE_{2} \left( i, j \right)^{c} \right), \p \left( \cE_{2} \left( j, i \right)^{c} \right) \right\} 
\leq \frac{4j}{\log^{7} \left( \tcorr \right)} 
\leq \frac{4}{\log^{6} \left( \tcorr \right)}. 
\]

Turning now to PA trees, let $\varphi_{j}' \sim \Beta \left( 1/2, j - 3/2 \right)$. 
Then, again by combining P\'olya urn and martingale arguments as in the proof of Lemma~\ref{lem:B}, 
we have for every $z \in [0,1]$ that 
\begin{equation}\label{eq:step2_PA}
\max \left\{ \p \left( \frac{1}{\tcorr} \left| \left( T_{\tcorr}, v_{i} \right)_{v_{j} \downarrow} \right| \leq z \right), \p \left( \frac{1}{\tcorr} \left| \left( T_{\tcorr}, v_{j} \right)_{v_{i} \downarrow} \right| \leq z \right) \right\} 
\leq 
2 \p \left( \varphi_{j}' \leq 4 z \right).
\end{equation}
We have that 
\[
\p \left( \varphi_{j}' \leq z \right) 
= \frac{1}{B\left(\frac{1}{2}, j - \frac{3}{2} \right)} \int_{0}^{z} x^{-1/2} \left( 1 - x \right)^{j - 5/2} dx 
\leq \frac{2\sqrt{2}}{B\left(\frac{1}{2}, j - \frac{3}{2} \right)} \sqrt{z}, 
\]
where the inequality holds for every $z \in \left( 0, 1/2 \right)$. 
From~\eqref{eq:beta} and the symmetry of the beta function we have that 
\begin{equation}\label{eq:beta_function_lb}
B \left( \frac{1}{2}, j - \frac{3}{2} \right) 
= \pi \frac{j-1}{j-\frac{3}{2}} 
\binom{2j-2}{j-1} 
4^{-j+1} 
\geq 
\frac{1}{\sqrt{j-1}},
\end{equation}
where the inequality follows by using the bound $\binom{2n}{n} \geq 4^{n} / \sqrt{4n}$ which holds for all $n \geq 1$. 
Combining the two previous displays 
we have obtained that 
$\p \left( \varphi_{j}' \leq z \right) \leq 2 \sqrt{2 j z}$ 
for all $z \in \left( 0 , 1 / 2 \right)$. 
Plugging this back into~\eqref{eq:step2_PA} and using the fact that $j \leq \log \tcorr$, we have, for all $\tcorr$ large enough, that 
\[
\max \left\{ \p \left( \cE_{2} \left( i, j \right)^{c} \right), \p \left( \cE_{2} \left( j, i \right)^{c} \right) \right\} 
\leq \frac{8\sqrt{2} \sqrt{j}}{\log^{7/2} \left( \tcorr \right)} 
\leq \frac{12}{\log^{3} \left( \tcorr \right)}
\]
%for all $\tcorr$ large enough. 

Altogether we have shown in this step that for all distinct positive integers $i, j \leq \log \tcorr$, and for both PA and UA trees, we have, for all $\tcorr$ large enough, that 
\[
\p \left( \cE_{2} \left( i, j \right)^{c} \right) 
\leq \frac{12}{\log^{3} \left( \tcorr \right)}. 
\]

\textbf{Step 3:} \emph{The anti-centrality rankings for the early vertices are stable.} 

Using Step~2, we will now show that the relative anti-centrality of any pair of early vertices is ``stable'' (with probability close to $1$); that is, it does not change after a certain time. 
More specifically, we will show, for distinct positive integers $i, j \leq \log \tcorr$, 
that 
if $\Psi_{\tcorr} \left( v_{i} \right) > \Psi_{\tcorr} \left( v_{j} \right)$, 
then $\Psi_{t} \left( v_{i} \right) > \Psi_{t} \left( v_{j} \right)$ for every $t \geq \tcorr$, 
with probability close to $1$ (and similarly if the inequality goes the other way). 
We thus define the events 
\[
\cA_{3} \left( i, j \right) := \left\{ \forall\, t \geq \tcorr : \left( \Psi_{\tcorr} \left( v_{i} \right) - \Psi_{\tcorr} \left( v_{j} \right) \right) \left( \Psi_{t} \left( v_{i} \right) - \Psi_{t} \left( v_{j} \right) \right)  >  0 \right\}
\]
for distinct positive integers $i, j \leq \log \tcorr$, 
and also 
\[
\cA_{3} := \bigcap_{\substack{1 \leq i, j \leq \log \tcorr \\ i \neq j}} \cA_{3} \left( i, j \right).
\]

By Lemma~\ref{lem:from_centralities_to_subtrees}, if we wish to compare 
$\Psi_{t} \left( v_{i} \right)$ and $\Psi_{t} \left( v_{j} \right)$, 
it suffices to compare the sizes of the subtrees 
$\left( T_{t}, v_{i} \right)_{v_{j} \downarrow}$ and $\left( T_{t}, v_{j} \right)_{v_{i} \downarrow}$. This motivates defining the event 
\[
\cE_{3} \left( i, j \right) := \left\{ \forall\, t \geq \tcorr : \left| \frac{\left| \left( T_{t}, v_{i} \right)_{v_{j} \downarrow} \right|}{\left| \left( T_{t}, v_{i} \right)_{v_{j} \downarrow} \right| + \left| \left( T_{t}, v_{j} \right)_{v_{i} \downarrow} \right|} - \frac{1}{2} \right| > \frac{1}{\log^{3} \left( \tcorr \right)} \right\} \cap \cE_{2} \left( i, j \right) \cap \cE_{2} \left( j, i \right) 
\]
for distinct positive integers $i,j \leq \log \tcorr$. 
We claim that, for all $\tcorr$ large enough, if $\cE_{3} \left( i, j \right)$ holds, then $\cA_{3} \left( i, j \right)$ must also hold. 
To see this, first note that on $\cE_{2} \left( i, j \right) \cap \cE_{2} \left( j, i \right)$ we have that 
\begin{equation}\label{eq:E2_LB}
\left| \left( T_{t}, v_{i} \right)_{v_{j} \downarrow} \right| + \left| \left( T_{t}, v_{j} \right)_{v_{i} \downarrow} \right| 
\geq 
\left| \left( T_{\tcorr}, v_{i} \right)_{v_{j} \downarrow} \right| + \left| \left( T_{\tcorr}, v_{j} \right)_{v_{i} \downarrow} \right| 
\geq \frac{\tcorr}{\log^{7} \left( \tcorr \right)}.
\end{equation}
Since the quantity 
$\left| \left( T_{t}, v_{i} \right)_{v_{j} \downarrow} \right|$ can change by at most $1$ at a time, the display above implies that the ratio 
\begin{equation}\label{eq:subtree_sizes_ratio}
\frac{\left| \left( T_{t}, v_{i} \right)_{v_{j} \downarrow} \right|}{\left| \left( T_{t}, v_{i} \right)_{v_{j} \downarrow} \right| + \left| \left( T_{t}, v_{j} \right)_{v_{i} \downarrow} \right|}
\end{equation}
can only change by at most $\log^{7} \left( \tcorr \right) / \tcorr$ 
at each time step. Since this is smaller than $1 / \log^{3} \left( \tcorr \right)$ for all $\tcorr$ large enough, 
the event $\cE_{3} \left( i, j \right)$ thus implies, for all $\tcorr$ large enough, that the ratio in~\eqref{eq:subtree_sizes_ratio} is either strictly greater than $1/2$ for all $t \geq \tcorr$ or strictly smaller than $1/2$ for all $t \geq \tcorr$. 
In light of Lemma~\ref{lem:from_centralities_to_subtrees}, this implies that $\cA_{3} \left(i, j \right)$ holds for all $\tcorr$ large enough. 

In the remainder of this step we thus focus on bounding the probability of $\cE_{3} \left( i, j \right)$. 
Since $\cE_{3}(i,j) = \cE_{3}(j,i)$, we may, and thus will, assume in the following that $1 \leq i < j \leq \log \tcorr$. 
To abbreviate notation, 
we introduce $J := \left| \left( T_{j}, v_{j} \right)_{v_{i} \downarrow} \right|$, 
and note that $1 \leq J \leq j-1$. 
We first give the proof for UA trees and subsequently explain what changes for PA trees. 

Conditioned on $T_{j}$, the pair 
\[
\left( \left| \left( T_{t}, v_{i} \right)_{v_{j} \downarrow} \right|, \left| \left( T_{t}, v_{j} \right)_{v_{i} \downarrow} \right| \right),
\]
when viewed at times when one of the coordinates increases, 
evolves as a classical P\'olya urn started from $\left( 1, J \right)$. 
Therefore, conditioned on $T_{j}$, 
the limit 
\begin{equation}\label{eq:phi_ij}
\varphi_{i,j} := \lim_{t\to\infty} \frac{\left| \left( T_{t}, v_{i} \right)_{v_{j} \downarrow} \right|}{\left| \left( T_{t}, v_{i} \right)_{v_{j} \downarrow} \right| + \left| \left( T_{t}, v_{j} \right)_{v_{i} \downarrow} \right|}
\end{equation}
exists almost surely, 
and moreover $\varphi_{i,j} \sim \Beta \left( 1, J \right)$. 
Since this holds for every tree $T_{j}$ on $j$ vertices, 
the limiting random variable $\varphi_{i,j}$ exists almost surely unconditionally (and its distribution is a mixture of beta distributions). 
Plugging in the density of the $\Beta \left( 1, J \right)$ distribution we have, for all $\tcorr$ large enough, that 
\[
\p \left( \left| \varphi_{i,j} - \frac{1}{2} \right| \leq \frac{2}{\log^{3} \left( \tcorr \right)} \, \middle| \, T_{j} \right)
= J \int_{\frac{1}{2} - 2 / \log^{3} \left( \tcorr \right)}^{\frac{1}{2} + 2 / \log^{3} \left( \tcorr \right)} \left( 1 - x \right)^{J-1} dx 
\leq J \left( \frac{2}{3} \right)^{J-1} \frac{4}{\log^{3} \left( \tcorr \right)} 
\leq \frac{6}{\log^{3} \left( \tcorr \right)},
\]
where we used that $J \left( 2 / 3 \right)^{J-1} \leq 4/3$ for every positive integer $J$. 
Taking an expectation over $T_{j}$ we obtain that 
\begin{equation}\label{eq:close_to_half_prob}
\p \left( \left| \varphi_{i,j} - \frac{1}{2} \right| \leq \frac{2}{\log^{3} \left( \tcorr \right)} \right) 
\leq \frac{6}{\log^{3} \left( \tcorr \right)}
\end{equation}
for all $\tcorr$ large enough. 
We can now bound the probability of $\cE_{3} \left( i, j \right)^{c}$: 
\[
\p \left( \cE_{3} \left( i, j \right)^{c} \right) 
\leq 
\p \left( \left| \varphi_{i,j} - \frac{1}{2} \right| \leq \frac{2}{\log^{3} \left( \tcorr \right)} \right) 
+ 
\p \left( \cE_{3} \left( i, j \right)^{c} \bigcap \left\{ \left| \varphi_{i,j} - \frac{1}{2} \right| > \frac{2}{\log^{3} \left( \tcorr \right)} \right\} \right).
\]
By~\eqref{eq:close_to_half_prob} the first term above is at most $6/\log^{3} \left( \tcorr \right)$ for all $\tcorr$ large enough, so what remains is to bound the second term. 
To do this, we introduce the event 
\[
\cE' := \left\{ \exists\, t \geq \tcorr :  \left| \frac{\left| \left( T_{t}, v_{i} \right)_{v_{j} \downarrow} \right|}{\left| \left( T_{t}, v_{i} \right)_{v_{j} \downarrow} \right| + \left| \left( T_{t}, v_{j} \right)_{v_{i} \downarrow} \right|} - \varphi_{i,j} \right| \geq \frac{1}{\log^{3} \left( \tcorr \right)} \right\}.
\]
By the triangle inequality and a union bound we have that 
\[
\p \left( \cE_{3} \left( i, j \right)^{c} \bigcap \left\{ \left| \varphi_{i,j} - \frac{1}{2} \right| > \frac{2}{\log^{3} \left( \tcorr \right)} \right\} \right) 
\leq 
\p \left( \cE_{2} \left( i, j \right)^{c} \right) + \p \left( \cE_{2} \left( j, i \right)^{c} \right) 
+ \p \left( \cE' \cap \cE_{2} \left( i, j \right) \cap \cE_{2} \left( i, j \right) \right).
\]
The first two terms in the display above are bounded above by $C/\log^{3} \left( \tcorr \right)$ for some finite $C$, by Step 2. It thus remains to bound the third term. 
To do this, we condition on the tree $T_{\tcorr}$. 
By the tower rule, noting that $\cE_{2} \left( i, j \right)$ and $\cE_{2} \left( i, j \right)$ are measurable with respect to $T_{\tcorr}$, we have that 
\begin{equation}\label{eq:E'_tower}
\p \left( \cE' \cap \cE_{2} \left( i, j \right) \cap \cE_{2} \left( i, j \right) \right) 
= 
\E \left[ \p \left( \cE' \, \middle| \, T_{\tcorr} \right) \mathbf{1}_{\cE_{2} \left( i, j \right) \cap \cE_{2} \left( j, i \right)} \right]. 
\end{equation}
Now if $\cE'$ holds then there exists $t \geq \tcorr$ such that 
\[
\left| \frac{\left| \left( T_{t}, v_{i} \right)_{v_{j} \downarrow} \right|}{\left| \left( T_{t}, v_{i} \right)_{v_{j} \downarrow} \right| + \left| \left( T_{t}, v_{j} \right)_{v_{i} \downarrow} \right|} - \frac{\left| \left( T_{\tcorr}, v_{i} \right)_{v_{j} \downarrow} \right|}{\left| \left( T_{\tcorr}, v_{i} \right)_{v_{j} \downarrow} \right| + \left| \left( T_{\tcorr}, v_{j} \right)_{v_{i} \downarrow} \right|} \right| \geq \frac{1}{2\log^{3} \left( \tcorr \right)}. 
\]
Therefore, by Lemma~\ref{lem:Polya_concentration}, we have that 
\[
\p \left( \cE' \, \middle| \, T_{\tcorr} \right) 
\leq 
2 \exp \left(  - \frac{\left| \left( T_{\tcorr}, v_{i} \right)_{v_{j} \downarrow} \right| + \left| \left( T_{\tcorr}, v_{j} \right)_{v_{i} \downarrow} \right|}{8 \log^{6} \left( \tcorr \right)}  \right).
\]
By~\eqref{eq:E2_LB} this implies that 
\[
\p \left( \cE' \, \middle| \, T_{\tcorr} \right) \mathbf{1}_{\cE_{2} \left( i, j \right) \cap \cE_{2} \left( j, i \right)} 
\leq 
2 \exp \left( - \tfrac{1}{8} \tcorr \log^{-13} \left( \tcorr \right) \right)
\]
and so by~\eqref{eq:E'_tower} we have that 
\[
\p \left( \cE' \cap \cE_{2} \left( i, j \right) \cap \cE_{2} \left( i, j \right) \right) 
\leq 
2 \exp \left( - \tfrac{1}{8} \tcorr \log^{-13} \left( \tcorr \right) \right).
\]
Putting everything together we have thus shown for UA trees that 
\[
\p \left( \cE_{3} \left( i, j \right)^{c} \right) 
\leq 
\frac{C}{\log^{3} \left( \tcorr \right)}
\]
for some finite constant $C$ and all $\tcorr \geq 2$. 

The proof for PA trees is similar, so we only highlight the minor changes. First, conditioned on~$T_{j}$, the pair 
\[
\left( 2 \left| \left( T_{t}, v_{i} \right)_{v_{j} \downarrow} \right| - 1, 2 \left| \left( T_{t}, v_{j} \right)_{v_{i} \downarrow} \right| - 1 \right),
\]
when viewed at times when one of the coordinates increases, 
evolves as a P\'olya urn 
with replacement matrix 
$\left(\begin{smallmatrix} 2 & 0 \\ 0 & 2 \end{smallmatrix}\right)$, 
started from $\left( 1, 2J - 1 \right)$. 
This implies that, conditioned on $T_{j}$, we have that $\varphi_{i,j} \sim \Beta \left( 1/2, J - 1/2 \right)$. 
The probability estimate with the beta distribution follows similarly, 
resulting in the inequality in~\eqref{eq:close_to_half_prob}, with the constant $6$ replaced with a larger finite constant. 
The rest of proof is unchanged, except when Lemma~\ref{lem:Polya_concentration} is applied, then the constant in the exponent changes.

We have thus shown, for both PA and UA trees, that 
\[
\p \left( \cA_{3}^{c} \right) 
\leq 
\sum_{\substack{1 \leq i, j \leq \log \tcorr \\ i \neq j}} \p \left( \cA_{3} \left( i,j \right)^{c} \right) 
\leq 
\sum_{\substack{1 \leq i, j \leq \log \tcorr \\ i \neq j}} \p \left( \cE_{3} \left( i,j \right)^{c} \right)
\leq 
\sum_{\substack{1 \leq i, j \leq \log \tcorr \\ i \neq j}} \frac{C}{\log^{3} \left( \tcorr \right)} 
\leq \frac{C}{\log \tcorr} 
\]
for some finite constant $C$ and all $\tcorr \geq 2$. 

\emph{Brief recap.} We briefly pause to recap what we have proved so far. 
Observe that on the event $\cA_{1} \cap \cA_{3}$ we have that property~\ref{prop:nice1} of Definition~\ref{def:nice_event} holds. In Steps~1 and~3 above we proved that 
$\p \left( \left( \cA_{1} \cap \cA_{3} \right)^{c} \right) 
\leq
\p \left( \cA_{1}^{c} \right) + \p \left( \cA_{3}^{c} \right) 
\leq 
C / \log \tcorr$ 
for some finite constant $C$ and all $\tcorr \geq 2$. 
What remains is to deal with properties~\ref{prop:nice2} and~\ref{prop:nice3} of Definition~\ref{def:nice_event}.

\textbf{Step 4:} \emph{The root of the largest pendent subtree of the centroid is an early vertex.} 

Recall the definition of $\wt{v}_{i,t} \left( 1 \right)$ from Section~\ref{sec:coarse_preliminaries}: 
$\wt{v}_{i,t} \left( 1 \right)$ is the neighbor of $v_{i}$ that is the root of the largest subtree of $\left( T_{t}, v_{i} \right)$ (assuming that there is a unique largest subtree; if the largest subtree is not unique, let $\wt{v}_{i,t} \left( 1 \right)$ denote a neighbor of $v_{i}$ that is the root of a largest subtree of $\left( T_{t}, v_{i} \right)$). 
For $i \leq \log \tcorr$, define the event 
\[
\cA_{4} \left( i \right) := \left\{ 
\forall\, t \geq \tcorr  
\text{ the timestamp of } \wt{v}_{i,t} \left( 1 \right) \text{ is at most } \log \tcorr 
\right\}.
\]
Since $\wt{v}_{i,t} \left( 1 \right)$ may not be uniquely defined, the definition of $\cA_{4} \left( i \right)$ needs some clarification: 
in the definition of $\cA_{4} \left( i \right)$ it is understood that, 
if $\wt{v}_{i,t} \left( 1 \right)$ is not uniquely defined, 
then \emph{every} vertex that can be chosen as $\wt{v}_{i,t} \left( 1 \right)$ has timestamp at most $\log \tcorr$. 
In other words, $\cA_{4} \left( i \right)$ is the event that no neighbor of $v_{i}$ with timestamp greater than $\log \tcorr$ is the root of a largest subtree of $\left( T_{t}, v_{i} \right)$, for all $t \geq \tcorr$. 
Define also 
\[
\cA_{4} := \bigcap_{1 \leq i \leq \log \tcorr} \cA_{4} \left( i \right).
\]
Our goal in Step~4 is to bound $\p \left( \cA_{4}^{c} \right)$. 

To abbreviate notation, in the following we let $s := \log \tcorr$ and fix $i \leq s$. 
For any $t \geq s$ we define two subtrees. 
First, let $T_{t}' \left( i \right) := \left( T_{t}, v_{i} \right)_{\wt{v}_{i,s}(1) \downarrow}$; 
here if $\wt{v}_{i,s}(1)$ is not uniquely defined, then we fix a particular choice for the remainder of the argument. 
We also define $T_{t}'' \left( i \right)$ to be the subtree of $T_{t}$ rooted at $v_{i}$ that contains all subtrees of $\left( T_{t}, v_{i} \right)$ formed after time $s$. 
In particular, we have that 
$\left| T_{s}' \left( i \right) \right| = \Psi_{s} \left( v_{i} \right)$ 
and 
$\left| T_{s}'' \left( i \right) \right| = 1$. 
Now define the event 
\[
\cE_{4} \left( i \right) := \left\{ 
\forall\, t \geq s : \frac{\left| T_{t}'' \left( i \right) \right|}{\left| T_{t}'' \left( i \right) \right| + \left| T_{t}' \left( i \right) \right|} < \frac{1}{2}
\right\}.
\]
If $\cE_{4} \left( i \right)$ holds, then 
$\left| T_{t}'' \left( i \right) \right| < \left| T_{t}' \left( i \right) \right|$ 
for all $t \geq \tcorr$, 
which implies that no subtree of $v_{i}$ born after time $s$ ever becomes as large as the subtree rooted at $\wt{v}_{i,s}(1)$. 
Therefore if $\cE_{4} \left( i \right)$ holds, then $\cA_{4} \left( i \right)$ must also hold. 
Thus 
$\p \left( \cA_{4} \left( i \right)^{c} \right) \leq \p \left( \cE_{4} \left( i \right)^{c} \right)$, 
and in the following we bound this latter probability. 

Consider first the case of UA trees. %; we will subsequently explain what changes for PA trees. 
Conditioned on $T_{s}$, the pair 
$
\left( \left| T_{t}'' \left( i \right) \right|, \left| T_{t}' \left( i \right) \right| \right), 
$
when viewed at times when one of the coordinates increases, evolves as a classical P\'olya urn started from 
$\left( 1, \Psi_{s} \left( v_{i} \right) \right)$. 
Therefore Lemma~\ref{lem:Polya_concentration} implies that 
\begin{equation}\label{eq:concentration_step4}
\p \left( \exists\, t \geq s : \frac{\left| T_{t}'' \left( i \right) \right|}{\left| T_{t}'' \left( i \right) \right| + \left| T_{t}' \left( i \right) \right|} \geq \frac{1}{1+\Psi_{s} \left( v_{i} \right)} + \lambda \, \middle| \, T_{s} \right) 
\leq 
\exp \left( - \frac{\lambda^{2}}{2} \Psi_{s} \left( v_{i} \right) \right) 
\end{equation}
for every $\lambda > 0$. 
For PA trees a similar argument shows that~\eqref{eq:concentration_step4} holds with a different constant in the exponent, and for all $\lambda \geq 2 / \Psi_{s} \left( v_{i} \right)$. 

Recalling that $d_{n}(v)$ denotes the degree of $v$ in $T_{n}$, define the event 
\[
\cE' \left( i \right) := \left\{ d_{s} \left( v_{i} \right) < \sqrt{s} \log^{2} \left( s \right) \right\}.
\]
By Lemma~\ref{lem:deg_tail_bound} we have, for both PA and UA trees, that 
\begin{equation}\label{eq:deg_bound_s}
\p \left( \cE' \left( i \right)^{c} \right) 
\leq 
\exp \left( - c \log^{3} \left( s \right) \right) 
= 
\exp \left( - c \left( \log \log \tcorr \right)^{3} \right) 
\end{equation}
for some positive constant $c$. 
On the event $\cE' \left( i \right)$ we have that 
\[
\Psi_{s} \left( v_{i} \right) 
= \left| \left( T_{s}, v_{i} \right)_{\wt{v}_{i,s} \left( 1 \right) \downarrow} \right| 
\geq \frac{s-1}{\sqrt{s} \log^{2} \left( s \right)} 
%\geq \frac{\sqrt{s}}{2 \log^{2} \left( s \right)} 
\geq \log^{1/3} \left( \tcorr \right),
\]
where the second inequality holds for all $\tcorr$ large enough. 
Here the first inequality follows from the pigeonhole principle: 
there are $s-1$ vertices in the rooted subtree $\left( T_{s}, v_{i} \right)$ apart from $v_{i}$, 
and there are at most $\sqrt{s} \log^{2} \left( s \right)$ subtrees, 
so at least one of them has at least 
$\left( s - 1 \right) / \left( \sqrt{s} \log^{2} \left( s \right) \right)$ 
vertices. 

Combining this argument with the inequality~\eqref{eq:concentration_step4}, we have, for all $\tcorr$ large enough, that 
\begin{equation}\label{eq:conc_bound_nice_set_step4}
\p \left( \cE_{4} \left( i \right)^{c} \, \middle| \, T_{s} \right) \mathbf{1}_{\cE' \left( i \right)} 
\leq 
\exp \left( - c \log^{1/3} \left( \tcorr \right) \right) 
\end{equation}
for some positive constant $c$, and both PA and UA trees. 
Putting together~\eqref{eq:deg_bound_s} and~\eqref{eq:conc_bound_nice_set_step4} 
we thus have that 
\begin{align*}
\p \left( \cE_{4} \left( i \right)^{c} \right) 
&= \E \left[ \p \left( \cE_{4} \left( i \right)^{c} \, \middle| \, T_{s} \right) \right] 
\leq \E \left[ \p \left( \cE_{4} \left( i \right)^{c} \, \middle| \, T_{s} \right) \mathbf{1}_{\cE' \left( i \right)} \right] 
+ \p \left( \cE' \left( i \right)^{c} \right) \\
&\leq 
\exp \left( - c \log^{1/3} \left( \tcorr \right) \right) 
+ \exp \left( - c \left( \log \log \tcorr \right)^{3} \right) 
\leq 2 \exp \left( - c \left( \log \log \tcorr \right)^{3} \right) 
\end{align*}
for some positive constant $c$ and all $\tcorr$ large enough. 
Finally, by a union bound we have that 
\[
\p \left( \cA_{4}^{c} \right) 
\leq \sum_{i=1}^{\log \tcorr} \p \left( \cA_{4} \left( i \right)^{c} \right) 
\leq \sum_{i=1}^{\log \tcorr} \p \left( \cE_{4} \left( i \right)^{c} \right) 
\leq 2 \log \left( \tcorr \right) \exp \left( - c \left( \log \log \tcorr \right)^{3} \right) 
\]
for some positive constant $c$ and all $\tcorr$ large enough. 
This is at most $1/ \log \tcorr$ for all $\tcorr$ large enough. 

\textbf{Step 5:} \emph{Early subtree rankings are stable.} 

For $i$ satisfying $1 \leq i \leq \log \tcorr$, 
let $\cA_{5} \left( i \right)$ denote the event that 
for every pair of neighbors $u_{1}$, $u_{2}$ of $v_{i}$ that are early vertices (that is, have timestamp at most $\log \tcorr$), 
we either have that 
$\left| \left( T_{t}, v_{i} \right)_{u_{1}\downarrow} \right| > \left| \left( T_{t}, v_{i} \right)_{u_{2}\downarrow} \right|$ 
for all $t \geq \tcorr$ 
or that 
$\left| \left( T_{t}, v_{i} \right)_{u_{1}\downarrow} \right| < \left| \left( T_{t}, v_{i} \right)_{u_{2}\downarrow} \right|$ 
for all $t \geq \tcorr$. 
In other words, 
the pairwise rankings of early subtrees of $v_{i}$ do not change after time $\tcorr$. 
Define also 
$\cA_{5} := \cap_{1 \leq i \leq \log \tcorr} \cA_{5} \left( i \right)$. 

Observe that, since $u_{1}$ and $u_{2}$ are neighbors of $v_{i}$, 
we have that 
$\left( T_{t}, v_{i} \right)_{u_{1}\downarrow} = \left( T_{t}, u_{2} \right)_{u_{1}\downarrow}$ 
and that 
$\left( T_{t}, v_{i} \right)_{u_{2}\downarrow} = \left( T_{t}, u_{1} \right)_{u_{2}\downarrow}$. 
Let $k,\ell \leq \log \tcorr$ be distinct positive integers and recall from Step~3 that, 
for all $\tcorr$ large enough,  
on the event $\cE_{3} \left( k, \ell \right)$ 
we either have that 
$\left| \left( T_{t}, v_{k} \right)_{v_{\ell}\downarrow} \right| > \left| \left( T_{t}, v_{\ell} \right)_{v_{k}\downarrow} \right|$ 
for all $t \geq \tcorr$ 
or that 
$\left| \left( T_{t}, v_{k} \right)_{v_{\ell}\downarrow} \right| < \left| \left( T_{t}, v_{\ell} \right)_{v_{k}\downarrow} \right|$ 
for all $t \geq \tcorr$. 
Putting the previous two sentences together we have that 
\[
\bigcap_{\substack{1 \leq i, j \leq \log \tcorr \\ i \neq j}} \cE_{3} \left( i, j \right) 
\subseteq 
\cA_{5} 
\]
for all $\tcorr$ large enough. Consequently, by Step~3 we have, for some finite constant $C$ and all $\tcorr$ large enough, that 
\[
\p \left( \cA_{5}^{c} \right) 
\leq 
\sum_{\substack{1 \leq i, j \leq \log \tcorr \\ i \neq j}} \p \left( \cE_{3} \left( i,j \right)^{c} \right)
\leq 
\sum_{\substack{1 \leq i, j \leq \log \tcorr \\ i \neq j}} \frac{C}{\log^{3} \left( \tcorr \right)} 
\leq \frac{C}{\log \tcorr}. 
\]

Finally, observe that on the event 
$\cA_{1} \cap \cA_{4} \cap \cA_{5}$ 
we have that property~\ref{prop:nice2} of Definition~\ref{def:nice_event} holds. 
Furthermore, we have shown that 
$\p \left( \left( \cA_{1} \cap \cA_{4} \cap \cA_{5} \right)^{c} \right) 
\leq \p \left( \cA_{1}^{c} \right) + \p \left( \cA_{4}^{c} \right) + \p \left( \cA_{5}^{c} \right) 
\leq C/ \log \tcorr$ 
for some finite constant $C$ and all $\tcorr \geq 2$.

\textbf{Step 6:} \emph{Concentration for early subtrees.} 

It remains to deal with property~\ref{prop:nice3} of Definition~\ref{def:nice_event}. So far we have shown that on the event 
$\cA_{1} \cap \cA_{3} \cap \cA_{4} \cap \cA_{5}$ 
we have that 
properties~\ref{prop:nice1} and~\ref{prop:nice2} 
of Definition~\ref{def:nice_event} hold, 
and moreover that $\theta \left( \tcorr \right)$ and $\wt{\theta}_{\tcorr} \left( 1 \right)$ are both early vertices. 
In light of this we define 
the events 
\[
\cA_{6} \left( i, j \right) := \left\{ 
\forall\, t \geq \tcorr : 
\left| \frac{1}{t} \left| \left( T_{t}, v_{i} \right)_{v_{j} \downarrow} \right| 
- \frac{1}{\tcorr} \left| \left( T_{\tcorr}, v_{i} \right)_{v_{j} \downarrow} \right| 
\right| 
\leq 
\frac{1}{\tcorr^{1/3}} \cdot \frac{1}{\tcorr} \left| \left( T_{\tcorr}, v_{i} \right)_{v_{j} \downarrow} \right|
\right\}
\]
for distinct positive integers $i, j \leq \log \tcorr$, and also 
\[
\cA_{6} := \bigcap_{\substack{1 \leq i, j \leq \log \tcorr \\ i \neq j}} \cA_{6} \left( i, j \right).
\]
Observe that on the event 
$\cA_{1} \cap \cA_{3} \cap \cA_{4} \cap \cA_{5} \cap \cA_{6}$ 
we have that property~\ref{prop:nice3} of Definition~\ref{def:nice_event} holds. 
Thus to conclude the proof what remains to be shown is that 
$\p \left( \cA_{6}^{c} \right) \leq C / \log \tcorr$ 
for some finite constant $C$ and all $\tcorr \geq 2$. 

Fix distinct positive integers $i,j \leq \log \tcorr$. By arguments similar to those in Step~3, in particular using Lemma~\ref{lem:Polya_concentration}, we have that 
\[
\p \left( \cA_{6} \left( i, j \right)^{c} \, \middle| \, T_{\tcorr} \right) 
\leq 
2 \exp \left( - c \tcorr \left( \frac{1}{\tcorr^{1/3}} \cdot \frac{1}{\tcorr} \left| \left( T_{\tcorr}, v_{i} \right)_{v_{j} \downarrow} \right| \right)^{2} \right) 
= 
2 \exp \left( - c \tcorr^{-5/3} \left| \left( T_{\tcorr}, v_{i} \right)_{v_{j} \downarrow} \right|^{2} \right)
\]
for some positive constant $c$ and all $\tcorr$ large enough. Recalling the definition of $\cE_{2} \left( i, j \right)$ we thus have that 
\[
\p \left( \cA_{6} \left( i, j \right)^{c} \, \middle| \, T_{\tcorr} \right) \mathbf{1}_{\cE_{2} \left( i, j \right)} 
\leq 
2 \exp \left( - c \tcorr^{1/3} \log^{-14} \left( \tcorr \right) \right)
\]
for all $\tcorr$ large enough. Using Step~2 we thus have that 
\[
\p \left( \cA_{6} \left( i, j \right)^{c} \right) 
\leq 
\E \left[ \p \left( \cA_{6} \left( i, j \right)^{c} \, \middle| \, T_{\tcorr} \right) \mathbf{1}_{\cE_{2} \left( i, j \right)} \right]
+ 
\p \left( \cE_{2} \left( i, j \right)^{c} \right) 
\leq 
2 \exp \left( - c \tcorr^{1/3} \log^{-14} \left( \tcorr \right) \right) 
+ \frac{12}{\log^{3} \left( \tcorr \right)}
\]
for all $\tcorr$ large enough. 
The conclusion follows by a union bound. 
\end{proof}

%%%%%%%%%%%%%%%%%%%%%%%%%%%%%%%%%%%%%%%%%%%%%%%%%%
\section{Estimating $\tcorr$ with vanishing relative error as $\tcorr \to \infty$} \label{sec:estimation_large_tcorr} %%%
%%%%%%%%%%%%%%%%%%%%%%%%%%%%%%%%%%%%%%%%%%%%%%%%%%

In this section we prove Theorem~\ref{thm:estimation}. To do this, we build on the ideas and the estimator introduced in Section~\ref{sec:estimation_tcorr_coarse}, which provided an initial, coarse estimate of $\tcorr$. 
The key additional idea compared to Section~\ref{sec:estimation_tcorr_coarse} is to average, over many subtrees, statistics similar to $Y_{n}$; see Figure~\ref{fig:finer_estimate} for an illustration. 
We start by defining precisely the estimator used to prove Theorem~\ref{thm:estimation}.  

For a tree $T_{n}$ on $n$ vertices, let $T_{n} \left( k \right)$ denote the $k$th largest subtree of the rooted tree $\left( T_{n}, \theta \left( n \right) \right)$ (with ties broken arbitrarily), 
with the root of this subtree denoted by $\wt{\theta}_{n} \left( k \right)$. 
In particular, with this notation we have that $\Psi_{n} \left( \theta \left( n \right) \right) = \left| T_{n} \left( 1 \right) \right|$. 
As before, for anything defined for a tree $T_{n}$, if we add a superscript $i$ to it (where $i \in \left\{ 1, 2 \right\}$), this means that it is the appropriate object in the tree~$T_{n}^{i}$. 
For $i \in \left\{ 1, 2 \right\}$ and $k \geq 1$, define the normalized subtree size 
\[
X_{n}^{i} \left( k \right) := \frac{1}{n} \left| T_{n}^{i} \left( k \right) \right|;
\]
see Figure~\ref{fig:finer_estimate} for an illustration. 
Now define 
\[
Y_{n} \left( k \right) := \frac{\left( X_{n}^{1} \left( k \right) - X_{n}^{2} \left( k \right) \right)^{2}}{2 X_{n}^{1} \left( k \right) \left( 1 - X_{n}^{1} \left( k \right) \right)}
\]
and note that $Y_{n} \left( 1 \right) \equiv Y_{n}$. 
For any $k \geq 1$ define 
\[
S_{n} \left( k \right) := \frac{1}{k} \sum_{\ell = 1}^{k} Y_{n} \left( \ell \right).
\]
For $k=1$ we have that $S_{n} \left( 1 \right) = Y_{n} \left( 1 \right) = Y_{n}$ 
and everything proved in Section~\ref{sec:estimation_tcorr_coarse} applies. 
For $k > 1$ (and $k$ not too large, to be made precise later), we still have 
that $S_{n} \left( k \right)$ is concentrated around~$1/\tcorr$. 
The improvement in $S_{n}(k)$ for large $k$, compared to $S_{n}(1)$, 
is that $S_{n} \left( k \right)$ has smaller variance than $S_{n} \left( 1 \right)$, by roughly a factor of order $k$. 

\begin{figure}[t]
\centering
\includegraphics[width=0.8\textwidth]{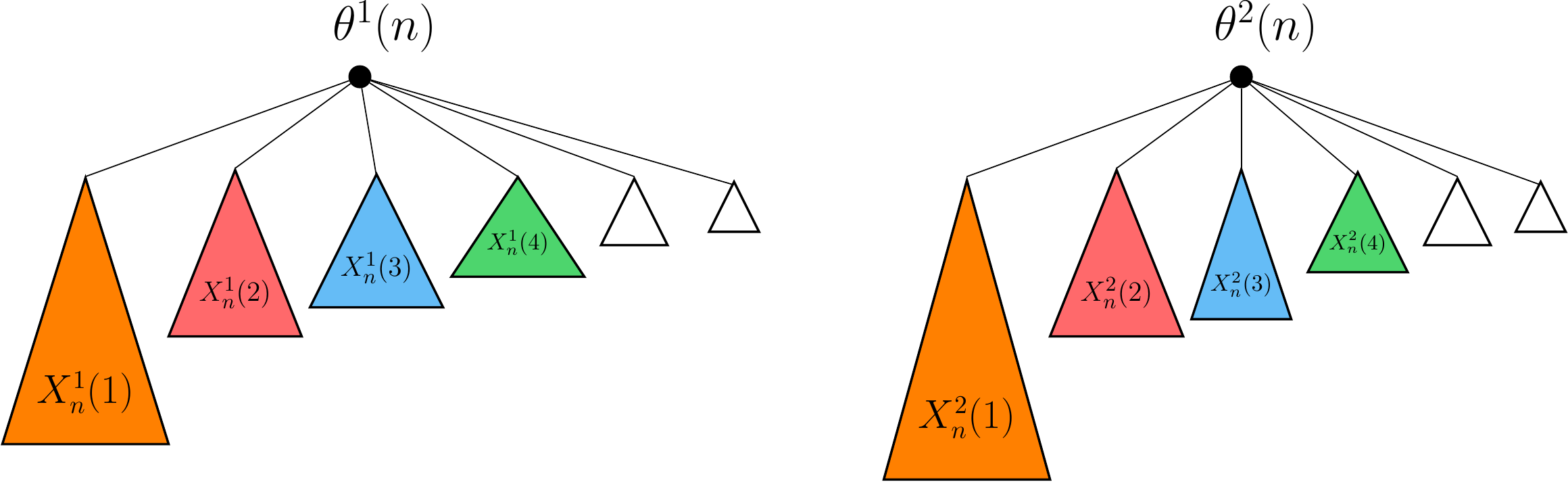}
\caption{The pendent subtrees of the centroids $\theta^{1}(n)$ and $\theta^{2}(n)$ in $T_{n}^{1}$ and $T_{n}^{2}$, respectively, are ordered in decreasing order. The estimator studied in Section~\ref{sec:estimation_large_tcorr} matches several of the largest pendent subtrees in the two trees, as indicated by the colors in the figure.}
\label{fig:finer_estimate}
\end{figure}

In order to obtain a significant improvement over $S_{n} \left( 1 \right)$, we aim to use $S_{n} \left( k \right)$ with a choice of~$k$ that diverges as $\tcorr \to \infty$. 
The catch is that $\tcorr$ is unknown---in fact, it is the quantity that we desire to estimate. This is where it is useful to have an initial, coarse estimate of $\tcorr$, 
which allows to choose an appropriate $k$. 
To this end, define 
\[
K_{n} := \left\lfloor - \tfrac{1}{400} \log Y_{n} \right\rfloor.
\]
Our estimator for $\tcorr$ is then 
\[
\wh{t}_{n} := \frac{1}{S_{n} \left( K_{n} \right)}.
\]
Theorem~\ref{thm:estimation} then follows immediately from the following result. 
\begin{theorem}\label{thm:estimation_inverse} 
Let $S$ be the unique tree on two vertices and let $\left( T_{n}^{1}, T_{n}^{2} \right) \sim \CPA \left( n, \tcorr, S \right)$. 
We have that 
\[
\lim_{\tcorr \to \infty} \liminf_{n \to \infty} 
\p \left( \left( 1 - \frac{\log \log \tcorr}{2 \sqrt{\log \tcorr}} \right) \frac{1}{\tcorr} 
\leq S_{n} \left( K_{n} \right) 
\leq \left( 1 + \frac{\log \log \tcorr}{\sqrt{\log \tcorr}} \right) \frac{1}{\tcorr} \right) 
= 1.
\]
The same result also holds when $\left( T_{n}^{1}, T_{n}^{2} \right) \sim \CUA \left( n, \tcorr, S \right)$. 
\end{theorem}

In the remainder of this section, which is structured similarly to Section~\ref{sec:estimation_tcorr_coarse}, we prove this theorem. 
We start in Section~\ref{sec:estimation_preliminaries} with some preliminaries: 
specifically, we define a couple of ``nice'' events on the space of sequences of growing trees, on which we will obtain bounds for $S_{n} \left( K_{n} \right)$. 
We state and prove first moment estimates in Section~\ref{sec:estimation_first_second_moments}, 
where we also state a variance estimate whose proof we defer to Section~\ref{sec:estimation_variance_proof}. 
We then prove Theorem~\ref{thm:estimation_inverse} in Section~\ref{sec:estimation_proof}, using the fact that the previously defined ``nice'' events have probability close to $1$. 
We prove this latter fact in Section~\ref{sec:estimation_nice_events}.

%%%%%%%%%%%%%%%%%%%%%%%%%%%%%%%%%%%%%%%%%%%%%%%%%%
\subsection{Preliminaries} \label{sec:estimation_preliminaries} %%%
%%%%%%%%%%%%%%%%%%%%%%%%%%%%%%%%%%%%%%%%%%%%%%%%%%

In Section~\ref{sec:estimation_tcorr_coarse} we defined ``nice'' events $\cA$ and $\cB$. 
Here, we define analogous ``nice'' events, which we denote by $\cC$ and $\cD$. 
First, we define 
\[
K \equiv K \left( \tcorr \right) := 
\left\lfloor \tfrac{1}{384} \log \tcorr \right\rfloor,
\] 
which we fix for the rest of Section~\ref{sec:estimation_large_tcorr}. 
We are now ready to define the event $\cC$. 

\begin{definition}[The event $\cC$]\label{def:nice_event_fine}
Given a sequence of trees $\left\{ T_{n} \right\}_{n \geq \tcorr}$, we say that the event $\cC$ holds if and only if the following three properties all hold: 
\begin{enumerate}[label=(C\arabic*)]
\item\label{prop:nice1_fine}%[(A1)]\label{prop:nice1} 
The centroid $\theta(n)$ is unique for all $n \geq \tcorr$ 
and $\theta \left( n \right) = \theta \left( \tcorr \right)$ for all $n \geq \tcorr$. 
\item\label{prop:nice2_fine} %[(A2)]\label{prop:nice2} 
For all integers $1 \leq k \leq K$, 
the vertex $\wt{\theta}_{n} \left( k \right)$ is uniquely defined for all $n \geq \tcorr$, 
and also $\wt{\theta}_{n} \left( k \right) = \wt{\theta}_{\tcorr} \left( k \right)$ for all $n \geq \tcorr$. 
\item\label{prop:nice3_fine}%[(A3)]\label{prop:nice3} 
For all $n \geq \tcorr$ and all $1 \leq k \leq K$, we have that 
\begin{equation}\label{eq:C3}
\left| \frac{1}{n} \left| T_{n} \left( k \right) \right| 
- \frac{1}{\tcorr} \left| T_{\tcorr} \left( k \right) \right|  \right| 
\leq \frac{1}{\tcorr^{1/3}} \min \left\{ \frac{1}{\tcorr} \left| T_{\tcorr} \left( k \right) \right|, 1 - \frac{1}{\tcorr} \left| T_{\tcorr} \left( k \right) \right| \right\}.
\end{equation}
\end{enumerate}
\end{definition}

As in Definition~\ref{def:nice_event}, the exponent $1/3$ in~\eqref{eq:C3} is chosen for simplicity; 
any positive constant that is less than~$1/2$ is a good choice for everything that follows (though the choice impacts the choice of other constants/exponents later on). 
Also, we always have that $\left| T_{\tcorr} \left( k \right) \right| \leq \left| T_{\tcorr} \left( 1 \right) \right| = \Psi_{\tcorr} \left( \theta \left( \tcorr \right) \right) \leq \tcorr / 2$, 
so the minimum in~\eqref{eq:C3} is always attained by the first term; we include the second term in the definition just for clarity. 
Given a sequence of trees 
$\left\{ T_{n} \right\}_{n \geq 2}$, 
we say that the event $\cC$ holds if and only if it holds for the subsequence $\left\{ T_{n} \right\}_{n \geq \tcorr}$. 
The event $\cC$ clearly depends on $\tcorr$, but we choose to omit $\tcorr$ from the notation in order to keep notation lighter. 
The following lemma shows that for PA and UA trees the event $\cC$ holds with probability close to $1$ when $\tcorr$ is large. 

\begin{lemma}\label{lem:nice_event_fine}
Let $\left\{ T_{n} \right\}_{n \geq 2}$ be a sequence of trees started from the seed $S$ and grown according to PA or UA. 
There exists a finite constant $C$ such that 
for every $\tcorr \geq 2$ we have that 
\begin{equation}\label{eq:notC}
\p \left( \cC^{c} \right) \leq \frac{C}{\tcorr^{1/2000}},
\end{equation}
where $\cC^{c}$ denotes the complement of $\cC$. 
\end{lemma}

Lemma~\ref{lem:nice_event_fine} follows directly from Lemma~\ref{lem:D} below. 

Since the event $\cC$ is analogous to the event $\cA$, the intuition is similar. 
Let $\cC^{1}$ and $\cC^{2}$ denote the ``nice'' events corresponding to $\left\{ T_{n}^{1} \right\}_{n \geq 2}$ and $\left\{ T_{n}^{2} \right\}_{n \geq 2}$, respectively. 
The key point of the construction is that on the event $\cC^{1} \cap \cC^{2}$, 
studying $X_{n}^{1} \left( k \right)$ and $X_{n}^{2} \left( k \right)$ reduces to studying the evolution of \emph{fixed} subtrees that are present in the tree at time $\tcorr$. 

Formally, condition on the tree $T_{\tcorr}^{1} = T_{\tcorr}^{2} =: T_{\tcorr}$. 
To abbreviate notation, 
we write 
$\theta := \theta \left( \tcorr \right)$ 
and 
$\wt{\theta} \left( k \right) := \wt{\theta}_{\tcorr} \left( k \right)$ for all $1 \leq k \leq K$; 
importantly, note that these are now \emph{fixed} vertices 
(i.e., they do not change with $n$). Define the random variables 
\[
Z_{n}^{i} \left( k \right) 
:= \frac{1}{n} \left| \left( T_{n}^{i}, \theta \right)_{\wt{\theta} \left( k \right) \downarrow} \right|
\]
for $i \in \left\{ 1, 2 \right\}$, $1 \leq k \leq K$,  and $n \geq \tcorr$. 
On the event $\cC^{i}$ we have that 
$X_{n}^{i} \left( k \right) = Z_{n}^{i} \left( k \right)$ 
for all $n \geq \tcorr$ and all $1 \leq k \leq K$. 

As discussed in Section~\ref{sec:estimation_tcorr_coarse} for $k = 1$, 
by classical results on P\'olya urns it follows that the limiting random variables 
\[
Z^{i} \left( k \right) := \lim_{n \to \infty} Z_{n}^{i} \left( k \right) 
\]
exist almost surely for $i \in \left\{ 1, 2 \right\}$ and $1 \leq k \leq K$, 
for both PA and UA trees. 
Moreover, for any $1 \leq k \leq K$, 
we have that 
$Z^{1} \left( k \right)$ and $Z^{2} \left( k \right)$ are i.i.d.\ 
(this is conditional independence given $T_{\tcorr}$) 
beta random variables, with parameters given as follows: 
\begin{equation}\label{eq:Zk}
Z \left( k \right) \sim 
\begin{cases} 
\Beta \left( \left| T_{\tcorr} \left( k \right) \right| , \tcorr - \left| T_{\tcorr} \left( k \right) \right| \right) &\text{ for } \UA, \\
\Beta \left( \left| T_{\tcorr} \left( k \right) \right| - \frac{1}{2}, \tcorr - \left| T_{\tcorr} \left( k \right) \right| - \frac{1}{2} \right) &\text{ for } \PA. 
\end{cases}
\end{equation}
Here $Z \left( k \right)$ is a random variable with the same distribution as $Z^{1} \left( k \right)$ and $Z^{2} \left( k \right)$.

From~\eqref{eq:Zk} it is clear that the quantity 
$\left| T_{\tcorr} \left( k \right) \right|$ plays an important role in the distribution of~$Z \left( k \right)$. 
In Section~\ref{sec:estimation_tcorr_coarse} we defined $\cB$ 
to be the event that 
$\left| T_{\tcorr} \left( 1 \right) \right| 
\geq \tcorr / \sqrt{\log \tcorr}$. 
Here we analogously want to define an event $\cD$ on which we have lower bounds for 
$\left| T_{\tcorr} \left( k \right) \right|$ 
for all $1 \leq k \leq K$. 
However, it turns out that we need some further properties from the event 
$\cD$; 
because of this we do not define it explicitly here---see Section~\ref{sec:estimation_nice_events} for an implicit definition. 
The following lemma guarantees the existence of an event $\cD$ 
with the appropriate properties.

\begin{lemma}\label{lem:D}
Let $\left\{ T_{n} \right\}_{n \geq 2}$ be a sequence of trees started from the seed $S$ and grown according to PA or UA. 
There exists a finite constant $C$ such that for every $\tcorr \geq C$ the following holds. 
There exists a $T_{\tcorr}$-measurable event $\cD$ such that the following three things hold. 
First, on $\cD$ we have for all $1 \leq k \leq K$ that 
\[
\left| T_{\tcorr} \left( k \right) \right| \geq \tcorr^{7/8}.
\]
Second, 
\[
\p \left( \cD^{c} \right) \leq \frac{C}{\tcorr^{1/2000}}. 
\]
Finally, 
\begin{equation}\label{eq:probCc_given_D}
\p \left( \cC^{c} \, \middle| \, \cD \right) \leq \frac{C}{\tcorr^{3}}. 
\end{equation}
\end{lemma} 
We note that the bound in~\eqref{eq:probCc_given_D} can be improved to a bound that decays faster than any polynomial in $\tcorr$; however, we only state this simpler, weaker bound, since this is all we need for our purposes. 
The proof of Lemma~\ref{lem:D} is deferred to Section~\ref{sec:estimation_nice_events}. 
In the following, $\cD$ always refers to the event guaranteed by Lemma~\ref{lem:D}.

%%%%%%%%%%%%%%%%%%%%%%%%%%%%%%%%%%%%%%%%%%%%%%%%%%
\subsection{First and second moment estimates} \label{sec:estimation_first_second_moments} %%%
%%%%%%%%%%%%%%%%%%%%%%%%%%%%%%%%%%%%%%%%%%%%%%%%%%

We first state and prove the following first moment estimates. 

\begin{lemma}\label{lem:first_moment}
Let $\left( T_{n}^{1}, T_{n}^{2} \right) \sim \CPA \left( n, \tcorr, S \right)$. 
Fix $k \in \left\{ 1, 2, \ldots, K \right\}$. 
For all $\tcorr$ large enough we have that 
\begin{equation}\label{eq:first_moment_UB}
\limsup_{n \to \infty} 
\E \left[ S_{n} \left( k \right) \mathbf{1}_{\cC^{1} \cap \cC^{2}} \, \middle| \, \cD \right] 
\leq 
\frac{1 + 3 \tcorr^{-1/3}}{\tcorr}
\end{equation}
and that 
\begin{equation}\label{eq:first_moment_LB}
\liminf_{n \to \infty} 
\E \left[ S_{n} \left( k \right) \mathbf{1}_{\cC^{1} \cap \cC^{2}} \, \middle| \, \cD \right] 
\geq 
\frac{1 - 3 \tcorr^{-1/3}}{\tcorr}.
\end{equation}
The same bounds also hold when $\left( T_{n}^{1}, T_{n}^{2} \right) \sim \CUA \left( n, \tcorr, S \right)$. 
\end{lemma}

\begin{proof} 
We start with the upper bound. 
By the exact same arguments as in the proof of Lemma~\ref{lem:coarse_first_moment}, we have for every $\ell \in \left\{ 1, \ldots, K \right\}$ that 
\[
\limsup_{n \to \infty} \E \left[ Y_{n} \left( \ell \right) \mathbf{1}_{\cC^{1} \cap \cC^{2}} \, \middle| \, \cD \right] 
\leq 
\left( 1 + \tfrac{1}{\tcorr - 1} \right)^{2} \left( 1 - \tcorr^{-1/3} \right)^{-2} \frac{1}{\tcorr}.
\]
Therefore by linearity of expectation we also have that 
\[
\limsup_{n \to \infty} 
\E \left[ S_{n} \left( k \right) \mathbf{1}_{\cC^{1} \cap \cC^{2}} \, \middle| \, \cD \right] 
\leq 
\left( 1 + \tfrac{1}{\tcorr - 1} \right)^{2} \left( 1 - \tcorr^{-1/3} \right)^{-2} \frac{1}{\tcorr}.
\]
The right hand side of the display above is at most $\left( 1 + 3 \tcorr^{-1/3} \right) / \tcorr$ for all $\tcorr$ large enough, which concludes the proof of~\eqref{eq:first_moment_UB}. 

We now turn to the lower bound. This follows similar lines as the upper bound, but an additional argument is needed. Fix $\ell \in \left\{ 1, \ldots, K \right\}$. We again condition on the tree $T_{\tcorr}$ at time $\tcorr$; 
by the tower rule we have that 
\[
\E \left[ Y_{n} \left( \ell \right) \mathbf{1}_{\cC^{1} \cap \cC^{2}} \, \middle| \, \cD \right] 
= \E \left[ \E \left[ Y_{n} \left( \ell \right) \mathbf{1}_{\cC^{1} \cap \cC^{2}} \, \middle| \, T_{\tcorr} \right] \, \middle| \, \cD \right].
\]
Now given $T_{\tcorr}$ such that $\cD$ holds, 
property~\ref{prop:nice3_fine} in Definition~\ref{def:nice_event_fine} implies that on the event $\cC^{1}$ we have that 
\[
X_{n}^{1} \left( \ell \right) \left( 1 - X_{n}^{1} \left( \ell \right) \right)
\leq 
\frac{1}{\tcorr} \left| T_{\tcorr}^{1} \left( \ell \right) \right| \left( 1 - \frac{1}{\tcorr} \left| T_{\tcorr}^{1} \left( \ell \right) \right| \right)
\left( 1 + \tcorr^{-1/3} \right)^{2}
\]
for $n \geq \tcorr$. Plugging this inequality into the definition of $Y_{n} \left( \ell \right)$ we obtain that 
\begin{align*}
\E \left[ Y_{n} \left( \ell \right) \mathbf{1}_{\cC^{1} \cap \cC^{2}} \, \middle| \, T_{\tcorr} \right] 
&\geq \frac{\E \left[ \left( X_{n}^{1} \left( \ell \right) - X_{n}^{2} \left( \ell \right) \right)^{2} \mathbf{1}_{\cC^{1} \cap \cC^{2}} \, \middle| \, T_{\tcorr} \right]}{2 \cdot \frac{1}{\tcorr} \left| T_{\tcorr} \left( \ell \right) \right| \left( 1 - \frac{1}{\tcorr} \left| T_{\tcorr} \left( \ell \right) \right| \right) \left( 1 + \tcorr^{-1/3} \right)^{2}} \\
&= \frac{\E \left[ \left( Z_{n}^{1} \left( \ell \right) - Z_{n}^{2} \left( \ell \right) \right)^{2} \mathbf{1}_{\cC^{1} \cap \cC^{2}} \, \middle| \, T_{\tcorr} \right]}{2 \cdot \frac{1}{\tcorr} \left| T_{\tcorr} \left( \ell \right) \right| \left( 1 - \frac{1}{\tcorr} \left| T_{\tcorr} \left( \ell \right) \right| \right) \left( 1 + \tcorr^{-1/3} \right)^{2}},
\end{align*}
where the equality follows by observing that on the event $\cC^{1} \cap \cC^{2}$ we have that $X_{n}^{i} \left( \ell \right) = Z_{n}^{i} \left( \ell \right)$ for $i \in \left\{ 1, 2 \right\}$. 
Now writing the indicator as 
$\mathbf{1}_{\cC^{1} \cap \cC^{2}} 
= 1 - \mathbf{1}_{\left( \cC^{1} \cap \cC^{2} \right)^{c}}$, 
we have that  
\begin{multline}\label{eq:LB_two_terms}
\E \left[ Y_{n} \left( \ell \right) \mathbf{1}_{\cC^{1} \cap \cC^{2}} \, \middle| \, T_{\tcorr} \right] \\
\geq 
\frac{\E \left[ \left( Z_{n}^{1} \left( \ell \right) - Z_{n}^{2} \left( \ell \right) \right)^{2} \, \middle| \, T_{\tcorr} \right]}{2 \cdot \frac{1}{\tcorr} \left| T_{\tcorr} \left( \ell \right) \right| \left( 1 - \frac{1}{\tcorr} \left| T_{\tcorr} \left( \ell \right) \right| \right) \left( 1 + \tcorr^{-1/3} \right)^{2}}
- 
\frac{\E \left[ \left( Z_{n}^{1} \left( \ell \right) - Z_{n}^{2} \left( \ell \right) \right)^{2} \mathbf{1}_{\left( \cC^{1} \cap \cC^{2} \right)^{c}} \, \middle| \, T_{\tcorr} \right]}{2 \cdot \frac{1}{\tcorr} \left| T_{\tcorr} \left( \ell \right) \right| \left( 1 - \frac{1}{\tcorr} \left| T_{\tcorr} \left( \ell \right) \right| \right) \left( 1 + \tcorr^{-1/3} \right)^{2}}
\end{multline} 
We deal with the two terms in~\eqref{eq:LB_two_terms} separately, starting with the first term, for which the analysis is similar to that in the upper bound.

By the bounded convergence theorem we have that 
\[
\lim_{n \to \infty} \E \left[ \left( Z_{n}^{1} \left( \ell \right) - Z_{n}^{2} \left( \ell \right) \right)^{2} \, \middle| \, T_{\tcorr} \right]
= 
\E \left[ \left( Z^{1} \left( \ell \right) - Z^{2} \left( \ell \right) \right)^{2} \, \middle| \, T_{\tcorr} \right]. 
\]
Now using conditional independence, the limiting conditional distribution obtained in~\eqref{eq:Zk}, and plugging in the variance of the beta distribution, 
we have that 
\[
\E \left[ \left( Z^{1} \left( \ell \right) - Z^{2} \left( \ell \right) \right)^{2} \, \middle| \, T_{\tcorr} \right] 
= 2 \Var \left( Z \left( \ell \right) \, \middle| \, T_{\tcorr} \right) 
= 
\begin{cases} 
\frac{2 \left| T_{\tcorr} \left( \ell \right) \right| \left( \tcorr - \left| T_{\tcorr} \left( \ell \right) \right|\right)}{\tcorr^{2} \left( \tcorr + 1 \right)} 
&\text{ for } \UA, \\
\frac{2 \left( \left| T_{\tcorr} \left( \ell \right) \right| - 1/2 \right) \left( \tcorr - \left| T_{\tcorr} \left( \ell \right) \right| - 1/2 \right)}{\left( \tcorr - 1 \right)^{2} \tcorr} 
&\text{ for } \PA. 
\end{cases}
\]
Plugging these formulas into the above, we obtain for UA trees that 
\[
\lim_{n \to \infty} 
\frac{\E \left[ \left( Z_{n}^{1} \left( \ell \right) - Z_{n}^{2} \left( \ell \right) \right)^{2} \, \middle| \, T_{\tcorr} \right]}{2 \cdot \frac{1}{\tcorr} \left| T_{\tcorr} \left( \ell \right) \right| \left( 1 - \frac{1}{\tcorr} \left| T_{\tcorr} \left( \ell \right) \right| \right) \left( 1 + \tcorr^{-1/3} \right)^{2}} 
= \frac{1}{\left( \tcorr + 1 \right) \left( 1 + \tcorr^{-1/3} \right)^{2}} 
\geq \frac{1 - 2.5 \tcorr^{-1/3}}{\tcorr},
\]
where the inequality holds for all $\tcorr$ large enough. 
For PA trees we obtain that
\begin{multline*}
\lim_{n \to \infty} 
\frac{\E \left[ \left( Z_{n}^{1} \left( \ell \right) - Z_{n}^{2} \left( \ell \right) \right)^{2} \, \middle| \, T_{\tcorr} \right]}{2 \cdot \frac{1}{\tcorr} \left| T_{\tcorr} \left( \ell \right) \right| \left( 1 - \frac{1}{\tcorr} \left| T_{\tcorr} \left( \ell \right) \right| \right) \left( 1 + \tcorr^{-1/3} \right)^{2}} \\
= 
\frac{1}{\tcorr} 
\cdot 
\frac{1}{\left( 1 + \tcorr^{-1/3} \right)^{2}} 
\cdot 
\frac{\tcorr^{2}}{\left( \tcorr - 1 \right)^{2}} 
\cdot 
\frac{\left| T_{\tcorr} \left( \ell \right) \right| - 1/2}{\left| T_{\tcorr} \left( \ell \right) \right|} 
\cdot 
\frac{\tcorr - \left| T_{\tcorr} \left( \ell \right) \right| - 1/2}{\tcorr - \left| T_{\tcorr} \left( \ell \right) \right|}.
\end{multline*}
We always have that 
$\left| T_{\tcorr} \left( \ell \right) \right| \leq \tcorr / 2$. 
Since $T_{\tcorr}$ is such that $\cD$ holds, 
by Lemma~\ref{lem:D} 
we also have that 
$\left| T_{\tcorr} \left( \ell \right) \right| \geq \tcorr^{7/8}$. 
Plugging these inequalities into the display above, we obtain that 
\begin{multline*}
\lim_{n \to \infty} 
\frac{\E \left[ \left( Z_{n}^{1} \left( \ell \right) - Z_{n}^{2} \left( \ell \right) \right)^{2} \, \middle| \, T_{\tcorr} \right]}{2 \cdot \frac{1}{\tcorr} \left| T_{\tcorr} \left( \ell \right) \right| \left( 1 - \frac{1}{\tcorr} \left| T_{\tcorr} \left( \ell \right) \right| \right) \left( 1 + \tcorr^{-1/3} \right)^{2}} \\
\geq
\frac{1}{\tcorr} 
\cdot 
\frac{1}{\left( 1 + \tcorr^{-1/3} \right)^{2}} 
\cdot 
\frac{\tcorr^{2}}{\left( \tcorr - 1 \right)^{2}} 
\cdot 
\frac{\tcorr^{7/8} - 1/2}{\tcorr^{7/8}} 
\cdot 
\frac{\tcorr / 2 - 1/2}{\tcorr / 2} 
\geq \frac{1 - 2.5 \tcorr^{-1/3}}{\tcorr},
\end{multline*}
where the second inequality holds for all $\tcorr$ large enough.

We now turn to the second term in~\eqref{eq:LB_two_terms}. 
Since $Z_{n}^{1} \left( \ell \right) - Z_{n}^{2} \left( \ell \right) \in [-1,1]$, 
we have that 
\[
\E \left[ \left( Z_{n}^{1} \left( \ell \right) - Z_{n}^{2} \left( \ell \right) \right)^{2} \mathbf{1}_{\left( \cC^{1} \cap \cC^{2} \right)^{c}} \, \middle| \, T_{\tcorr} \right] 
\leq 
\p \left( \left( \cC^{1} \cap \cC^{2} \right)^{c} \, \middle| \, T_{\tcorr} \right).
\]
As mentioned above, we always have that 
$\left| T_{\tcorr} \left( \ell \right) \right| \leq \tcorr / 2$; 
moreover,  
since $T_{\tcorr}$ is such that $\cD$ holds, 
by Lemma~\ref{lem:D} 
we also have that 
$\left| T_{\tcorr} \left( \ell \right) \right| \geq \tcorr^{7/8}$. 
Using these inequalities we may bound the second term in~\eqref{eq:LB_two_terms}: 
\[
\frac{\E \left[ \left( Z_{n}^{1} \left( \ell \right) - Z_{n}^{2} \left( \ell \right) \right)^{2} \mathbf{1}_{\left( \cC^{1} \cap \cC^{2} \right)^{c}} \, \middle| \, T_{\tcorr} \right]}{2 \cdot \frac{1}{\tcorr} \left| T_{\tcorr} \left( \ell \right) \right| \left( 1 - \frac{1}{\tcorr} \left| T_{\tcorr} \left( \ell \right) \right| \right) \left( 1 + \tcorr^{-1/3} \right)^{2}} 
\leq 
\tcorr^{1/8} \p \left( \left( \cC^{1} \cap \cC^{2} \right)^{c} \, \middle| \, T_{\tcorr} \right).
\] 
Taking an expectation over $T_{\tcorr}$, 
this bound becomes 
$\tcorr^{1/8} \p \left( \left( \cC^{1} \cap \cC^{2} \right)^{c} \, \middle| \, \cD \right)$. 
By Lemma~\ref{lem:D} we have that 
$\p \left( \left( \cC^{1} \cap \cC^{2} \right)^{c} \, \middle| \, \cD \right) 
\leq C \tcorr^{-3}$ 
for some finite constant $C$ and all $\tcorr$ large enough. 
Thus ultimately the bound becomes 
$C \tcorr^{-23/8}$, 
which is at most 
$0.5 \tcorr^{-4/3}$ for all $\tcorr$ large enough. 

Overall, we have thus shown that 
\[
\liminf_{n \to \infty} \E \left[ Y_{n} \left( \ell \right) \mathbf{1}_{\cC^{1} \cap \cC^{2}} \, \middle| \, \cD \right] 
\geq 
\frac{1 - 3 \tcorr^{-1/3}}{\tcorr} 
\]
for all $\tcorr$ large enough (where here ``large enough'' does not depend on $\ell$).
The bound in~\eqref{eq:first_moment_LB} follows by linearity of expectation. 
\end{proof}

The following lemma gives a variance bound that we will use. 

\begin{lemma}\label{lem:variance}
Let $\left( T_{n}^{1}, T_{n}^{2} \right) \sim \CPA \left( n, \tcorr, S \right)$. 
There exists a finite constant $C$ such that 
for all $\tcorr$ large enough we have 
for all $k \in \left\{ 1, 2, \ldots, K \right\}$ that 
\begin{equation}\label{eq:var_bound}
\limsup_{n \to \infty} 
\Var \left( S_{n} \left( k \right) \mathbf{1}_{\cC^{1} \cap \cC^{2}} \, \middle| \, \cD \right) 
\leq 
\frac{C}{k \tcorr^{2}}.
\end{equation}
The same bound also holds when $\left( T_{n}^{1}, T_{n}^{2} \right) \sim \CUA \left( n, \tcorr, S \right)$. 
\end{lemma}

The proof of Lemma~\ref{lem:variance} is somewhat lengthy, so we defer it to Section~\ref{sec:estimation_variance_proof}. 

%%%%%%%%%%%%%%%%%%%%%%%%%%%%%%%%%%%%%%%%%%%%%%%%%%
\subsection{Putting everything together: proof of Theorem~\ref{thm:estimation_inverse}} \label{sec:estimation_proof} %%%
%%%%%%%%%%%%%%%%%%%%%%%%%%%%%%%%%%%%%%%%%%%%%%%%%%

\begin{proof}[Proof of Theorem~\ref{thm:estimation_inverse}] 
In the following we set 
\begin{equation}\label{eq:eps}
\eps := \frac{\log\log \tcorr}{2 \sqrt{\log \tcorr}}
\end{equation}
to abbreviate notation. Our goal is to show that 
\begin{equation}\label{eq:estimation_goal}
\lim_{\tcorr \to \infty} \limsup_{n \to \infty} 
\p \left( \left| S_{n} \left( K_{n} \right) - \frac{1}{\tcorr} \right| \geq \frac{\eps}{\tcorr} \right) = 0. 
\end{equation}

To do this, we first fix $k \in \left\{ 1, 2, \ldots, K \right\}$ 
and bound the probability 
$\p \left( \left| S_{n} \left( k \right) - 1 / \tcorr \right| \geq \eps / \tcorr \right)$. 
By conditioning on the ``nice'' event $\cD$, we have that 
\[
\p \left( \left| S_{n} \left( k \right) - \frac{1}{\tcorr} \right| \geq \frac{\eps}{\tcorr} \right) 
\leq 
\p \left( \left| S_{n} \left( k \right) - \frac{1}{\tcorr} \right| \geq \frac{\eps}{\tcorr} \, \middle| \, \cD \right) 
+ \p \left( \cD^{c} \right). 
\]
The second term above is at most $C / \tcorr^{1/2000}$ by Lemma~\ref{lem:D}. 
We can break the first term above into two further terms, based on whether the ``nice'' event $\cC^{1} \cap \cC^{2}$ holds or not: by a union bound we have that 
\[
\p \left( \left| S_{n} \left( k \right) - \frac{1}{\tcorr} \right| \geq \frac{\eps}{\tcorr} \, \middle| \, \cD \right)  
\leq 
\p \left( \left| S_{n} \left( k \right) \mathbf{1}_{\cC^{1} \cap \cC^{2}} - \frac{1}{\tcorr} \right| \geq \frac{\eps}{\tcorr} \, \middle| \, \cD \right) 
+ 
\p \left( \left( \cC^{1} \cap \cC^{2} \right)^{c} \, \middle| \, \cD \right).
\] 
The second term in the display above is at most $C/\tcorr^{3}$ by Lemma~\ref{lem:D}, 
so it remains to deal with the first term above. 
Recall that Lemma~\ref{lem:first_moment} implies that 
for all $\tcorr$ large enough 
we have for all $n$ large enough that 
\[
\left| 
\E \left[ S_{n} \left( k \right) \mathbf{1}_{\cC^{1} \cap \cC^{2}} \, \middle| \, \cD \right] 
- \frac{1}{\tcorr} 
\right| 
\leq 
\frac{4 \tcorr^{-1/3}}{\tcorr}.
\]
Recalling the definition of $\eps$ from~\eqref{eq:eps}, 
note that 
$\eps \geq 8 \tcorr^{-1/3}$ for all $\tcorr$ large enough 
and hence $\eps - 4 \tcorr^{-1/3} \geq \eps / 2$ for all $\tcorr$ large enough. 
By the triangle inequality we thus have that 
\[
\p \left( \left| S_{n} \left( k \right) \mathbf{1}_{\cC^{1} \cap \cC^{2}} - \frac{1}{\tcorr} \right| \geq \frac{\eps}{\tcorr} \, \middle| \, \cD \right) 
\leq 
\p \left( \left| S_{n} \left( k \right) \mathbf{1}_{\cC^{1} \cap \cC^{2}} - \E \left[ S_{n} \left( k \right) \mathbf{1}_{\cC^{1} \cap \cC^{2}} \, \middle| \, \cD \right] \right| \geq \frac{\eps}{2\tcorr} \, \middle| \, \cD \right). 
\]
Finally, by Chebyshev's inequality we have that 
\[
\p \left( \left| S_{n} \left( k \right) \mathbf{1}_{\cC^{1} \cap \cC^{2}} - \E \left[ S_{n} \left( k \right) \mathbf{1}_{\cC^{1} \cap \cC^{2}} \, \middle| \, \cD \right] \right| \geq \frac{\eps}{2\tcorr} \, \middle| \, \cD \right) 
\leq 
\frac{4 \tcorr^{2}}{\eps^{2}} \Var \left( S_{n} \left( k \right) \mathbf{1}_{\cC^{1} \cap \cC^{2}} \, \middle| \, \cD \right).
\]
Taking a limit as $n \to \infty$ and putting all the above bounds together we have thus obtained that 
\begin{equation}\label{eq:final_bound_k}
\limsup_{n \to \infty} 
\p \left( \left| S_{n} \left( k \right) - \frac{1}{\tcorr} \right| \geq \frac{\eps}{\tcorr} \right)  
\leq \frac{C}{k \eps^{2}} + \frac{C}{\tcorr^{1/2000}} 
\end{equation}
for some finite constant $C$ and all $\tcorr$ large enough.

Now we are ready to show~\eqref{eq:estimation_goal}. Define the event 
\[
\cE := \left\{ 
\log \tcorr - \log \log \tcorr 
\leq 
- \log Y_{n}
\leq 
\log \tcorr + \log \log \tcorr
\right\}. 
\]
By a union bound we have that 
\[
\p \left( \left| S_{n} \left( K_{n} \right) - \frac{1}{\tcorr} \right| \geq \frac{\eps}{\tcorr} \right) 
\leq 
\p \left( \left\{ \left| S_{n} \left( K_{n} \right) - \frac{1}{\tcorr} \right| \geq \frac{\eps}{\tcorr} \right\} \cap \cE \right) 
+ \p \left( \cE^{c} \right). 
\]
By Theorem~\ref{thm:estimation_coarse_inverse} we have that 
$\lim_{\tcorr \to \infty} \limsup_{n \to \infty} \p \left( \cE^{c} \right) = 0$, 
so what remains is to deal with the first term on the right hand side of the display above. 
On the event $\cE$ we have that 
\[
\left\lfloor \tfrac{1}{400} \log \tcorr - \tfrac{1}{400} \log \log \tcorr \right\rfloor
\leq 
K_{n} 
\leq
\left\lfloor \tfrac{1}{400} \log \tcorr + \tfrac{1}{400} \log \log \tcorr \right\rfloor, 
\]
so by a union bound we have that 
\[
\p \left( \left\{ \left| S_{n} \left( K_{n} \right) - \frac{1}{\tcorr} \right| \geq \frac{\eps}{\tcorr} \right\} \cap \cE \right) 
\leq 
\sum_{k = \left\lfloor \frac{1}{400} \log \tcorr - \frac{1}{400} \log \log \tcorr \right\rfloor}^{\left\lfloor \frac{1}{400} \log \tcorr + \frac{1}{400} \log \log \tcorr \right\rfloor} 
\p \left( \left| S_{n} \left( k \right) - \frac{1}{\tcorr} \right| \geq \frac{\eps}{\tcorr} \right). 
\]
Note that 
$\frac{1}{400} \log \tcorr - \frac{1}{400} \log \log \tcorr \leq K$ 
for all $\tcorr$ large enough, 
so we can apply the bound~\eqref{eq:final_bound_k} that holds for fixed $k \leq K$. 
Thus taking a limit as $n \to \infty$ and applying~\eqref{eq:final_bound_k} we thus obtain that 
\[
\limsup_{n \to \infty} 
\p \left( \left\{ \left| S_{n} \left( K_{n} \right) - \frac{1}{\tcorr} \right| \geq \frac{\eps}{\tcorr} \right\} \cap \cE \right) 
\leq 
\frac{C \log \log \tcorr}{\eps^{2} \log \tcorr} + \frac{C \log \log \tcorr}{\tcorr^{1/2000}} 
\leq 
\frac{C'}{\log \log \tcorr}
\]
for some finite constants $C$ and $C'$, and all $\tcorr$ large enough, 
where in the second inequality we used the definition of $\eps$ from~\eqref{eq:eps}. 
Taking the limit as $\tcorr \to \infty$ concludes the proof. 
\end{proof}

%%%%%%%%%%%%%%%%%%%%%%%%%%%%%%%%%%%%%%%%%%%%%%%%%%
\subsection{Proof of Lemma~\ref{lem:D}} \label{sec:estimation_nice_events} %%%
%%%%%%%%%%%%%%%%%%%%%%%%%%%%%%%%%%%%%%%%%%%%%%%%%%

We start with a preliminary lemma. 

\begin{lemma}\label{lem:deg_lb}
Let $\left\{ T_{n} \right\}_{n \geq 2}$ be a sequence of trees started from the seed $S = S_{2}$ and grown according to PA or UA. Let $v_{1}, v_{2}, v_{3}, \ldots$ denote the vertices in order of appearance. Let $d_{n} (v)$ denote the degree of $v$ in $T_{n}$. 
Fix $\eps > 0$. 
There exists a finite constant $C$ such that the following holds. 
For every $\tcorr \geq C$, $t \geq \tcorr^{\eps}$, and $i \leq 100 \log \tcorr$, 
we have that 
\[
\p \left( d_{t} \left( v_{i} \right) \leq \tfrac{1}{6} \log t \right) 
\leq 
t^{-1/28}.
\]
\end{lemma}
\begin{proof} 
If $i \leq i'$, then $d_{t}(v_{i})$ stochastically dominates $d_{t}(v_{i'})$ for every $t \geq i'$, 
so it suffices to prove the inequality for 
$i = i_{\star} := \left\lfloor 100 \log \tcorr \right\rfloor > 2$. 
Let $t > i_{\star}$, 
and let 
$X_{i_{\star}+1}, \ldots, X_{t}$ be independent Bernoulli random variables 
such that $\E \left[ X_{k} \right] = 1 / (2k-4)$. 
Conditioned on $T_{t-1}$, the probability that $v_{t}$ connects to $v_{i_{\star}}$ is at least $1 / (2t - 4)$, for both PA and UA trees. 
This implies that $d_{t} \left( v_{i_{\star}} \right)$ stochastically dominates 
$Y_{t} := 1 + X_{i_{\star} + 1} + \ldots X_{t}$. 
Thus we have that 
\[
\p \left( d_{t} \left( v_{i_{\star}} \right) \leq \tfrac{1}{6} \log t \right) 
\leq 
\p \left( Y_{t} \leq \tfrac{1}{6} \log t \right).
\]
Since $Y_{t} - 1$ is the sum of independent Bernoulli random variables, we have that $\Var \left( Y_{t} \right) \leq \E \left[ Y_{t} \right]$. 
Thus by Bernstein's inequality we have for every $x \geq 0$ that 
\[
\p \left( Y_{t} \leq \E \left[ Y_{t} \right] - x \right) 
\leq \exp \left( - \frac{x^{2} / 2}{\E \left[ Y_{t} \right] + x /3} \right).
\]
Setting $x = \E \left[ Y_{t} \right] / 2$, we obtain that 
\[
\p \left( Y_{t} \leq \tfrac{1}{2} \E \left[ Y_{t} \right] \right) 
\leq \exp \left( - \tfrac{3}{28} \E \left[ Y_{t} \right] \right).
\]
We have that 
\[
\E \left[ Y_{t} \right] 
= 1 + \frac{1}{2} \sum_{k = i_{\star} - 1}^{t-2} \frac{1}{k}
\geq 1 + \frac{1}{2} \int_{i_{\star}-1}^{t-1} \frac{1}{x} dx 
= 1 + \frac{1}{2} \log \left( \frac{t-1}{i_{\star}-1} \right)
\]
and so $\E \left[ Y_{t} \right] \geq \frac{1}{3} \log t$ 
for all $\tcorr$ large enough. 
Plugging this inequality into the displays above and putting them together concludes the proof. 
\end{proof}

%We are now ready to prove Lemma~\ref{lem:D}.

\begin{proof}[Proof of Lemma~\ref{lem:D}]
The proof is similar to that of Lemma~\ref{lem:nice_event} and most of the work has already been done there. However, we modify the proof in a few key places to show the desired result. 
First, we slightly change the definition of an early vertex. 
Fix $\gamma := 18$. 
In the following we informally call a vertex an ``early'' vertex if its timestamp is at most $\gamma \log \tcorr$. 
We also fix 
$s_{1} := \tcorr^{1/64}$ 
and 
$s_{2} := \tcorr^{1/48}$, 
and note that $s_{1} = s_{2}^{3/4}$. 

\textbf{Modified Step 1:} \emph{The centroid is an early vertex.} 

Recall the definition of $\cA_{1}(i)$ from Lemma~\ref{lem:nice_event}. 
Define 
$\cC_{1} := \cap_{i > \gamma \log \tcorr} \cA_{1} \left( i \right)$, 
the event that only early vertices are ever a centroid. 
Similarly as in Step~1 of Lemma~\ref{lem:nice_event}, we thus have, for all~$\tcorr$ large enough, that 
\begin{equation}\label{eq:centroid_early_vertices}
\p \left( \cC_{1}^{c} \right) 
\leq \sum_{i > \gamma \log \tcorr} \p \left( \cA_{1} \left( i \right)^{c} \right) 
\leq \sum_{i > \gamma \log \tcorr} e^{-i/3} 
\leq \frac{4}{\tcorr^{\gamma/3}}. 
\end{equation}
Now let $\cD_{1}$ denote the event that 
$T_{\tcorr}$ satisfies 
$\p \left( \cC_{1}^{c} \, \middle| \, T_{\tcorr} \right) \leq \tcorr^{-\gamma / 6}$. 
By Markov's inequality, the tower rule, and~\eqref{eq:centroid_early_vertices} we have that 
\begin{equation}\label{eq:D1_bound}
\p \left( \cD_{1}^{c} \right) 
= \p \left( \p \left( \cC_{1}^{c} \, \middle| \, T_{\tcorr} \right) > \tcorr^{-\gamma / 6} \right) 
\leq \tcorr^{\gamma / 6} \E \left[ \p \left( \cC_{1}^{c} \, \middle| \, T_{\tcorr} \right) \right]
= \tcorr^{\gamma / 6} \p \left( \cC_{1}^{c} \right) 
\leq 4 \tcorr^{-\gamma / 6},
\end{equation}
where the last inequality holds for all $\tcorr$ large enough. 

\textbf{Modified Step 2:} \emph{Subtrees formed before time $s_{2}$ are large in $T_{\tcorr}$.} 

In Step~2 of Lemma~\ref{lem:nice_event} we proved that early subtrees are large in $T_{\tcorr}$. Here we need to show that many more subtrees are large---though what ``large'' means is relaxed here---for reasons that will become clear in later steps. 
Formally, define the $T_{\tcorr}$-measurable events 
\[
\cD_{2} \left( i, j \right) 
:= \left\{ \left| \left( T_{\tcorr}, v_{i} \right)_{v_{j} \downarrow} \right| \geq \tcorr^{7/8} \right\}.
\]
for distinct positive integers $i, j \leq s_{2}$, and also 
\[
\cD_{2} := \bigcap_{\substack{1 \leq i, j \leq s_{2} \\ i \neq j}}\cD_{2} \left( i, j \right).
\]
We proceed by bounding the probability of the complement of $\cD_{2} \left( i, j \right)$. 
Since the arguments are identical to those in Step~2 of Lemma~\ref{lem:nice_event}, 
we omit most details and only give the final bounds. 

Assume in the following that $1 \leq i < j \leq s_{2}$. 
In Step~2 of Lemma~\ref{lem:nice_event} we showed that, for both PA and UA trees, and for every $z \in [0,1]$, we have that 
\[
\max \left\{ \p \left( \frac{1}{\tcorr} \left| \left( T_{\tcorr}, v_{i} \right)_{v_{j} \downarrow} \right| \leq z \right), \p \left( \frac{1}{\tcorr} \left| \left( T_{\tcorr}, v_{j} \right)_{v_{i} \downarrow} \right| \leq z \right) \right\} 
\leq 
12 \sqrt{jz}.
\]
Setting $z = \tcorr^{-1/8}$ and using the bound $j \leq s_{2}$, we obtain that 
\[
\max \left\{ \p \left( \cD_{2} \left( i, j \right)^{c} \right), \p \left( \cD_{2} \left( i, j \right)^{c} \right) \right\} 
\leq 12 s_{2}^{1/2} \tcorr^{-1/16}.
\] 
By a union bound we thus have, for both PA and UA trees, that 
\begin{equation}\label{eq:D2c_bound}
\p \left( \cD_{2}^{c} \right) 
\leq 12 s_{2}^{5/2} \tcorr^{-1/16}
\leq 12 \tcorr^{-1/96}.
\end{equation}

\textbf{Modified Step 3:} \emph{The size-based ranking in $T_{\tcorr}$of subtrees formed before time $s_{2}$ persists.}

This is similar to Step~3 of Lemma~\ref{lem:nice_event}, but with some differences, which we highlight. 
Define the events 
\begin{align*}
\cC_{3} \left( i, j \right) &:= \left\{ \forall\, t \geq \tcorr : \left( \left| \left( T_{\tcorr}, v_{i} \right)_{v_{j} \downarrow} \right| - \left| \left( T_{\tcorr}, v_{j} \right)_{v_{i} \downarrow} \right| \right) \left( \left| \left( T_{t}, v_{i} \right)_{v_{j} \downarrow} \right| - \left| \left( T_{t}, v_{j} \right)_{v_{i} \downarrow} \right| \right) > 0 \right\}, \\
\cH_{3} \left( i, j \right) &:= \left\{ \forall\, t \geq \tcorr : \left| \frac{\left| \left( T_{t}, v_{i} \right)_{v_{j} \downarrow} \right|}{\left| \left( T_{t}, v_{i} \right)_{v_{j} \downarrow} \right| + \left| \left( T_{t}, v_{j} \right)_{v_{i} \downarrow} \right|} - \frac{1}{2} \right| > \tcorr^{-1/4} \right\}, \\
\cD_{3} \left( i, j \right) &:= \left\{ \left| \frac{\left| \left( T_{\tcorr}, v_{i} \right)_{v_{j} \downarrow} \right|}{\left| \left( T_{\tcorr}, v_{i} \right)_{v_{j} \downarrow} \right| + \left| \left( T_{\tcorr}, v_{j} \right)_{v_{i} \downarrow} \right|} - \frac{1}{2} \right| \geq 2 \tcorr^{-1/4} \right\}
\end{align*}
for distinct positive integers $i, j \leq s_{2}$, and also 
\[
\cC_{3} := \bigcap_{\substack{1 \leq i, j \leq s_{2} \\ i \neq j}} \cC_{3} \left( i, j \right) 
\qquad \qquad 
\text{ and }
\qquad \qquad 
\cD_{3} := \bigcap_{\substack{1 \leq i, j \leq s_{2} \\ i \neq j}} \cD_{3} \left( i, j \right).
\]
By the same arguments as in Step~3 of Lemma~\ref{lem:nice_event}, 
we have that if $\cH_{3} \left( i, j \right) \cap \cD_{2}$ holds, 
then $\cC_{3} \left( i, j \right)$ must also hold. By Lemma~\ref{lem:Polya_concentration} we have, for every tree $T_{\tcorr}$ such that $\cD_{3} \left( i, j \right) \cap \cD_{2}$ holds, that 
\[
\p \left( \cH_{3} \left( i, j \right)^{c} \, \middle| \, T_{\tcorr} \right) 
\leq 2 \exp \left( - \left( 2\tcorr^{7/8} - 1 \right) \tcorr^{-1/2} / 8 \right) 
\leq 2 \exp \left( - \tcorr^{3/8} / 8 \right).
\]
Thus by a union bound we have, for every tree $T_{\tcorr}$ such that $\cD_{3} \cap \cD_{2}$ holds, that 
\begin{equation}\label{eq:C3_bound}
\p \left( \cC_{3}^{c} \, \middle| \, T_{\tcorr} \right) 
\leq 2 s_{2}^{2} \exp \left( - \tcorr^{3/8} / 8 \right),
\end{equation}
and note that this decays faster than any polynomial in $\tcorr$. 

In the remainder of this step we bound the probability $\p \left( \left( \cD_{3} \cap \cD_{2} \right)^{c} \right)$. A union bound shows that 
$\p \left( \left( \cD_{3} \cap \cD_{2} \right)^{c} \right) 
\leq \p \left( \cD_{2}^{c} \right) + \p \left( \cD_{3}^{c} \cap \cD_{2} \right)$; 
the former probability is at most $12 \tcorr^{-1/96}$ by~\eqref{eq:D2c_bound}, 
so it suffices to bound $\p \left( \cD_{3}^{c} \cap \cD_{2} \right)$. 
By a further union bound, and incurring a factor of~$s_{2}^{2}$, 
it suffices to bound 
$\p \left( \cD_{3} \left( i, j \right)^{c} \cap \cD_{2} \right)$, 
where $1 \leq i < j \leq s_{2}$. 
To this end, define $\varphi_{i,j}$ as in~\eqref{eq:phi_ij}; 
again 
this limiting random variable exists almost surely. 
By a union bound we then have that 
\[
\p \left( \cD_{3} \left( i, j \right)^{c} \cap \cD_{2} \right) 
\leq 
\p \left( \left| \varphi_{i,j} - \tfrac{1}{2} \right| \leq 4 \tcorr^{-1/4} \right)
+ \p \left( \cD_{3} \left( i, j \right)^{c} \cap \cD_{2} \cap \left\{ \left| \varphi_{i,j} - \tfrac{1}{2} \right| > 4 \tcorr^{-1/4} \right\} \right).
\]
Both of these terms can be bounded by the same arguments as in Step~3 of Lemma~\ref{lem:nice_event}. 
First, there exists a finite absolute constant $C$ such that the first term above is at most $C \tcorr^{-1/4}$. 
Next, the second term is at most $2 \exp \left( - \tcorr^{3/8} / 8 \right)$. 
Altogether this gives that 
$\p \left( \cD_{3} \left( i, j \right)^{c} \cap \cD_{2} \right) 
\leq C' \tcorr^{-1/4}$ 
for some finite absolute constant $C'$. By a union bound we thus have that 
$\p \left( \cD_{3}^{c} \cap \cD_{2} \right) \leq C' \tcorr^{-5/24}$. 
Putting everything together we have thus obtained that 
\begin{equation}\label{eq:D3_bound}
\p \left( \left( \cD_{3} \cap \cD_{2} \right)^{c} \right) 
\leq C'' \tcorr^{-1/96}
\end{equation}
for some finite absolute constant $C''$.

\textbf{Modified Step 4:} \emph{The roots of the $K$ largest pendent subtrees of the centroid have timestamp at most~$s_{2}$.}

This is similar to Step~4 of Lemma~\ref{lem:nice_event}, but with significant differences---this step has the biggest differences among all. 
For one, we have to additionally show that the centroid has degree at least~$K$.

For a positive integer $i \leq \gamma \log \tcorr$ define the event 
\[
\cD_{4}' \left( i \right) 
:= \left\{ d_{s_{1}} \left( v_{i} \right) > \tfrac{1}{6} \log s_{1} \right\},
\]
and also define $\cD_{4}' := \cap_{1 \leq i \leq \gamma \log \tcorr} \cD_{4}' \left( i \right)$. 
By Lemma~\ref{lem:deg_lb} and a union bound we have that 
\begin{equation}\label{eq:D4'_bound}
\p \left( \left( \cD_{4}' \right)^{c} \right) 
\leq 
\sum_{i=1}^{\gamma \log \tcorr} \p \left( \left( \cD_{4}' \left( i \right) \right)^{c} \right) 
\leq \left( \gamma \log \tcorr \right) s_{1}^{-1/28} 
= \left( \gamma \log \tcorr \right) \tcorr^{-1/1792}
\end{equation}
for all $\tcorr$ large enough. 
Observe that if $T_{s_{1}}$ is such that $\cD_{4}'$ holds, 
then---since $K \leq (1/6) \log s_{1}$---all early vertices have degree at least $K$ in $T_{s_{1}}$, and hence also in $T_{t}$ for $t \geq s_{1}$ (in particular $t = \tcorr$). 
For every $T_{s_{1}}$ such that $\cD_{4}'$ holds, 
and for every $i \leq \gamma \log \tcorr$, 
choose and fix $K$ neighbors of $v_{i}$ in $T_{s_{1}}$ arbitrarily (e.g., the $K$ neighbors with largest pendent subtrees: $\wt{v}_{i,s_{1}}(1), \ldots, \wt{v}_{i,s_{1}}(K)$, with ties broken by favoring earlier vertices), and label them as $u_{1}^{i}, \ldots, u_{K}^{i}$. 
In the following, whenever we refer to a tree $T_{s_{1}}$ such that $\cD_{4}'$ holds, 
we automatically assume this fixed choice of $K \gamma \log \tcorr$ labeled vertices (where repetitions are possible). 
In the following we fix $T_{s_{1}}$ such that $\cD_{4}'$ holds and condition on $T_{s_{1}}$. 

Now fix $i \leq \gamma \log \tcorr$. 
To simplify notation, we write $u_{1}, \ldots, u_{K}$ instead of $u_{1}^{i}, \ldots, u_{K}^{i}$. 
By P\'olya urn arguments it follows that (conditioned on $T_{s_{1}}$) 
for every $\ell \in [K]$ the limiting random variable 
\[
\phi_{i,\ell} := \lim_{t \to \infty} \frac{1}{t} \left| \left( T_{t}, v_{i} \right)_{u_{\ell} \downarrow} \right| 
\]
exists almost surely. Moreover, its distribution (conditioned on $T_{s_{1}}$) is given by 
\[
\phi_{i,\ell} \sim 
\begin{cases} 
\Beta \left( \left| \left( T_{s_{1}}, v_{i} \right)_{u_{\ell} \downarrow} \right| , s_{1} - \left| \left( T_{s_{1}}, v_{i} \right)_{u_{\ell} \downarrow} \right| \right) &\text{ for } \UA, \\
\Beta \left( \left| \left( T_{s_{1}}, v_{i} \right)_{u_{\ell} \downarrow} \right| - \frac{1}{2}, s_{1} - \left| \left( T_{s_{1}}, v_{i} \right)_{u_{\ell} \downarrow} \right| - \frac{1}{2} \right) &\text{ for } \PA. 
\end{cases}
\]
We now argue that, for both PA and UA trees, for all $x \in [0,1]$ we have that 
\begin{equation}\label{eq:phi_bound}
\p \left( \phi_{i,\ell} < x \, \middle| \, T_{s_{1}} \right) \leq 2 \sqrt{s_{1} x}.
\end{equation}
We start with UA trees. When 
$\left| \left( T_{s_{1}}, v_{i} \right)_{u_{\ell} \downarrow} \right| = 1$, 
a direct computation shows that 
$\p \left( \phi_{i,\ell} < x \, \middle| \, T_{s_{1}} \right) = 1 - (1 - x)^{s_{1} - 1} \leq s_{1} x$. 
Otherwise, 
Markov's inequality implies that 
\begin{equation}\label{eq:phi_markov}
\p \left( \phi_{i,\ell} < x \, \middle| \, T_{s_{1}} \right) 
= \p \left( \phi_{i,\ell}^{-1} > x^{-1} \, \middle| \, T_{s_{1}} \right) 
\leq x \E \left[ \phi_{i,\ell}^{-1} \, \middle| \, T_{s_{1}} \right] 
= x \frac{s_{1} - 1}{\left| \left( T_{s_{1}}, v_{i} \right)_{u_{\ell} \downarrow} \right| - 1} 
%\leq x (s_{1} - 1) 
\leq s_{1} x.
\end{equation}
For PA trees, when 
$\left| \left( T_{s_{1}}, v_{i} \right)_{u_{\ell} \downarrow} \right| = 1$, 
a direct computation shows that 
\[
\p \left( \phi_{i,\ell} < x \, \middle| \, T_{s_{1}} \right) 
= \frac{1}{B \left( \frac{1}{2}, s_{1} - \frac{3}{2} \right)} \int_{0}^{x} y^{-1/2} (1-y)^{s_{1} - 5/2} dy 
\leq \sqrt{s_{1}} \int_{0}^{x} y^{-1/2} dy 
= 2 \sqrt{s_{1} x},  
\]
where in the inequality we used that 
$B \left( \frac{1}{2}, s_{1} - \frac{3}{2} \right) \geq 1 / \sqrt{s_{1}}$ (see~\eqref{eq:beta_function_lb}) 
and also that $\tcorr$ is large enough (so that $s_{1} \geq 5/2$). 
Otherwise, Markov's inequality (just like in~\eqref{eq:phi_markov}) implies a bound of $2 s_{1} x$. 
In conclusion, we have shown~\eqref{eq:phi_bound} in all cases. 
As a consequence, using the same martingale arguments as in the proof of Lemma~\ref{lem:B}, we have that 
\begin{equation}\label{eq:bound_s2}
\p \left( \frac{1}{s_{2}} \left| \left( T_{s_{2}}, v_{i} \right)_{u_{\ell} \downarrow} \right| \leq x \, \middle| \, T_{s_{1}} \right) 
\leq 
2 \p \left( \phi_{i,\ell} \leq 4 x \, \middle| \, T_{s_{1}} \right) 
\leq 8 \sqrt{s_{1} x}.
\end{equation}
Now define the event 
\[
\cD_{4}'' \left( i \right) 
:= 
\bigcap_{1 \leq \ell \leq K} \left\{ \left| \left( T_{s_{2}}, v_{i} \right)_{u_{\ell} \downarrow} \right| \geq s_{2}^{1/8} \right\}, 
\]
which is well-defined when $T_{s_{1}}$ is such that $\cD_{4}'$ holds. 
By a union bound and using~\eqref{eq:bound_s2} with $x = s_{2}^{-7/8}$, 
we have that 
\[
\p \left( \cD_{4}'' \left( i \right)^{c} \, \middle| \, T_{s_{1}} \right) 
\leq \sum_{\ell = 1}^{K} \p \left( \left| \left( T_{s_{2}}, v_{i} \right)_{u_{\ell} \downarrow} \right| < s_{2}^{1/8} \, \middle| \, T_{s_{1}} \right) 
\leq 8 K s_{2}^{-1/16} 
= 8 K \tcorr^{-1/768}.
\]

Now define the event $\cD_{4}'' := \cap_{1 \leq i \leq \gamma \log \tcorr} \cD_{4}'' \left( i \right)$, which is well-defined when $T_{s_{1}}$ is such that $\cD_{4}'$ holds. By the display above, together with a union bound, we have, for every $T_{s_{1}}$ such that $\cD_{4}'$ holds, that 
\begin{equation}\label{eq:D4''_bound}
\p \left( \left( \cD_{4}'' \right)^{c} \, \middle| \, T_{s_{1}} \right) 
\leq \left( 8 K \gamma \log \tcorr \right) \tcorr^{-1/768}.
\end{equation}

For $i \leq \gamma \log \tcorr$ define the event 
\[
\cC_{4} \left( i \right) 
:= \left\{ d_{\tcorr} \left( v_{i} \right) \geq K \right\} \cap \left\{ \forall\, t \geq \tcorr: 
\text{the timestamps of } \wt{v}_{i,t} \left( 1 \right), \ldots, \wt{v}_{i,t} \left( K \right) \text{are all at most } s_{2} \right\},
\]
and also let 
$\cC_{4} := \cap_{1 \leq i \leq \gamma \log \tcorr} \cC_{4} \left( i \right)$. 
Note that if $\cD_{4}'$ holds, 
then $\left\{ d_{\tcorr} \left( v_{i} \right) \geq K \right\}$ holds as well, 
so to understand $\cC_{4} \left( i \right)$ we need to understand the second event in the display above. 
To do this, we consider the subtree $T_{t}'$ of $\left( T_{t}, v_{i} \right)$ which is rooted at $v_{i}$ and consists of~$v_{i}$ together with all subtrees of $v_{i}$ that are formed after time~$s_{2}$. 
We can then define the event
\[
\cH_{4} \left( i \right) 
:= 
\bigcap_{1 \leq \ell \leq K} 
\left\{ \forall\, t \geq s_{2} : \left| \frac{\left| T_{t}' \right|}{\left| T_{t}' \right| + \left| \left( T_{t}, v_{i} \right)_{u_{\ell} \downarrow} \right|} - \frac{1}{1+\left| \left( T_{s_{2}}, v_{i} \right)_{u_{\ell} \downarrow} \right|} \right| \leq \frac{1}{3} \right\}, 
\] 
which is well-defined whenever $T_{s_{1}}$ is such that $\cD_{4}'$ holds. 
Provided that $\tcorr$ is large enough, 
if $\cH_{4} \left( i \right)$ holds, 
then 
$\left| T_{t}' \right| / \left( \left| T_{t}' \right| + \left| \left( T_{t}, v_{i} \right)_{u_{\ell} \downarrow} \right| \right) < 1/2$ 
for all $t \geq s_{2}$, 
which implies that no subtree born after time $s_{2}$ will ever become larger than any of the subtrees with roots $u_{1}, \ldots, u_{K}$. 
This, in turn, means that no subtree born after time $s_{2}$ will ever become one of the $K$ largest subtrees of $v_{i}$. 
Therefore 
$\cH_{4} \left( i \right) \subseteq \cC_{4} \left( i \right)$.

If $T_{s_{1}}$ is such that $\cD_{4}'$ holds, 
and also $T_{s_{2}}$ is such that $\cD_{4}''$ holds, 
then by Lemma~\ref{lem:Polya_concentration} and a union bound 
we have that 
\[
\p \left( \cH_{4} \left( i \right)^{c} \, \middle| \, T_{s_{1}}, T_{s_{2}} \right) 
\leq 2 K \exp \left( - \tfrac{1}{72} \tcorr^{1/384} \right).
\] 
Together with the previous paragraph and a union bound we thus have that 
\begin{equation}\label{eq:C4_bound}
\p \left( \cC_{4}^{c} \, \middle| \, T_{s_{1}}, T_{s_{2}} \right) 
\leq \left( 2 K \gamma \log \tcorr \right) \exp \left( - \tfrac{1}{72} \tcorr^{1/384} \right) 
\end{equation}
whenever $T_{s_{1}}$ is such that $\cD_{4}'$ holds, 
and also $T_{s_{2}}$ is such that $\cD_{4}''$ holds. 

The display above motivates defining $\cD_{4}$ 
to be the event that $T_{\tcorr}$ satisfies 
\begin{equation}\label{eq:C4_bound_2}
\p \left( \cC_{4}^{c} \, \middle| \, T_{\tcorr} \right) \leq \exp \left( - \tfrac{1}{144} \tcorr^{1/384} \right);
\end{equation}
note that $\cD_{4}$ is $T_{\tcorr}$-measurable. In the rest of this step we bound $\p \left( \cD_{4}^{c} \right)$. 
By conditioning first on~$T_{s_{1}}$ and then on $T_{s_{2}}$, together with a couple of union bounds, we obtain that 
\begin{equation}\label{eq:D4_bound_union_bound}
\p \left( \cD_{4}^{c} \right) 
\leq 
\E \left[ \p \left( \cD_{4}^{c} \, \middle| \, T_{s_{1}}, T_{s_{2}} \right) \mathbf{1}_{\cD_{4}'} \mathbf{1}_{\cD_{4}''} \right] 
+ \E \left[ \p \left( \left( \cD_{4}'' \right)^{c} \, \middle| \, T_{s_{1}} \right) \mathbf{1}_{\cD_{4}'} \right] 
+ \p \left( \left( \cD_{4}' \right)^{c} \right).
\end{equation}
By~\eqref{eq:D4'_bound} and~\eqref{eq:D4''_bound} 
we have that the second and the third term in the display above are together at most 
$\tcorr^{-1/1800}$ for all $\tcorr$ large enough. 
Turning to the first term in the display above, 
let $T_{s_{1}}$ be such that $\cD_{4}'$ holds, 
and subsequently let $T_{s_{2}}$ be such that $\cD_{4}''$ holds. 
Then by Markov's inequality we have that 
\begin{align*}
\p \left( \cD_{4}^{c} \, \middle| \, T_{s_{1}}, T_{s_{2}} \right) 
&= 
\p \left( \p \left( \cC_{4}^{c} \, \middle| \, T_{\tcorr} \right) > \exp \left( - \tfrac{1}{144} \tcorr^{1/384} \right) \, \middle| \, T_{s_{1}}, T_{s_{2}} \right) \\
&\leq 
\exp \left(  \tfrac{1}{144} \tcorr^{1/384} \right) 
\E \left[ \p \left( \cC_{4}^{c} \, \middle| \, T_{\tcorr} \right) \, \middle| \, T_{s_{1}}, T_{s_{2}} \right] 
= 
\exp \left(  \tfrac{1}{144} \tcorr^{1/384} \right) 
\p \left( \cC_{4}^{c} \, \middle| \, T_{s_{1}}, T_{s_{2}} \right). 
\end{align*}
Now plugging in~\eqref{eq:C4_bound}, we obtain that 
\[
\p \left( \cD_{4}^{c} \, \middle| \, T_{s_{1}}, T_{s_{2}} \right) 
\leq \left( 2 K \gamma \log \tcorr \right) \exp \left( - \tfrac{1}{144} \tcorr^{1/384} \right).
\]
Plugging this back into~\eqref{eq:D4_bound_union_bound} we finally obtain, for all $\tcorr$ large enough, that 
\begin{equation}\label{eq:D4_bound}
\p \left( \cD_{4}^{c} \right) 
\leq 2 \tcorr^{-1/1800}.
\end{equation}

\textbf{Modified Step 5:} 
In Step~5 of Lemma~\ref{lem:nice_event} we showed that early subtree rankings are stable. Here we already showed in Modified Step~3 that the size-based ranking in $T_{\tcorr}$ of subtrees formed before time $s_{2}$ persists.

\textbf{Modified Step 6:} \emph{Concentration of subtree sizes.} 

In light of the previous steps, we define the events 
\[
\cC_{6} \left( i, j \right) := \left\{ 
\forall\, t \geq \tcorr : 
\left| \frac{1}{t} \left| \left( T_{t}, v_{i} \right)_{v_{j} \downarrow} \right| 
- \frac{1}{\tcorr} \left| \left( T_{\tcorr}, v_{i} \right)_{v_{j} \downarrow} \right| 
\right| 
\leq 
\frac{1}{\tcorr^{1/3}} \cdot \frac{1}{\tcorr} \left| \left( T_{\tcorr}, v_{i} \right)_{v_{j} \downarrow} \right|
\right\}
\]
for distinct positive integers $i, j \leq s_{2}$, and also 
\[
\cC_{6} := \bigcap_{\substack{1 \leq i, j \leq s_{2} \\ i \neq j}} \cC_{6} \left( i, j \right).
\]
In Step~6 of Lemma~\ref{lem:nice_event} we showed that 
\[
\p \left( \cC_{6} \left( i, j \right)^{c} \, \middle| \, T_{\tcorr} \right) 
\leq 
2 \exp \left( - c \tcorr^{-5/3} \left| \left( T_{\tcorr}, v_{i} \right)_{v_{j} \downarrow} \right|^{2} \right)
\]
for some positive constant $c$ and all $\tcorr$ large enough. 
Thus if $T_{\tcorr}$ is such that $\cD_{2}$ holds, then 
\[
\p \left( \cC_{6} \left( i, j \right)^{c} \, \middle| \, T_{\tcorr} \right) 
\leq 
2 \exp \left( - c \tcorr^{-5/3} \tcorr^{7/4} \right)
= 2 \exp \left( -c \tcorr^{1/12} \right). 
\]
Thus by a union bound we have that if $T_{\tcorr}$ is such that $\cD_{2}$ holds, then 
\begin{equation}\label{eq:C6_bound}
\p \left( \cC_{6}^{c} \, \middle| \, T_{\tcorr} \right) 
\leq 
2 s_{2}^{2} \exp \left( -c \tcorr^{1/12} \right), 
\end{equation}
which decays faster than any polynomial in $\tcorr$.

\textbf{Putting everything together.} 
Define the events 
\begin{align*}
\cD &:= \cD_{1} \cap \cD_{2} \cap \cD_{3} \cap \cD_{4}, \\
\wt{\cC} &:= \cC_{1} \cap \cC_{3} \cap \cC_{4} \cap \cC_{6}.
\end{align*}
The event $\cD$ is $T_{\tcorr}$-measurable by construction. 
Putting together~\eqref{eq:D1_bound},~\eqref{eq:D3_bound}, and~\eqref{eq:D4_bound}, we have that 
$\p \left( \cD^{c} \right) \leq 3 \tcorr^{-1/1800}$ for all $\tcorr$ large enough. 

Next, we argue that if $\cD$ holds, then $\left| T_{\tcorr} \left( k \right) \right| \geq \tcorr^{7/8}$. 
First, note that if $\cD_{1}$ holds, then the centroid at time $\tcorr$ is an early vertex. 
If $\cD_{4}$ holds, then all early vertices have degree at least $K$ in~$T_{\tcorr}$, and for every early vertex the timestamps of their neighbors corresponding to the $K$ largest pendent subtrees are all at most $s_{2}$. 
Finally, if $\cD_{2}$ holds, then all subtrees formed before time $s_{2}$ have size at least $\tcorr^{7/8}$, 
and if $\cD_{3}$ holds, then none of these subtree sizes are equal (i.e., everything is well defined). 
Putting these observations together we indeed have that 
$\left| T_{\tcorr} \left( k \right) \right| \geq \tcorr^{7/8}$ if $\cD$ holds. 

Finally, turning to the event $\cC$, observe that $\wt{\cC} \subseteq \cC$ by construction. Therefore 
\[
\p \left( \cC^{c} \, \middle| \, \cD \right)
\leq 
\p \left( \wt{\cC}^{c} \, \middle| \, \cD \right)
\]
and it suffices to bound this latter quantity. 
Putting together the definition of $\cD_{1}$,~\eqref{eq:C3_bound}, the definition of $\cD_{4}$ (see~\eqref{eq:C4_bound_2}), and~\eqref{eq:C6_bound}, 
we have that for every tree $T_{\tcorr}$ such that $\cD$ holds, 
we have that 
\[
\p \left( \wt{\cC}^{c} \, \middle| \, T_{\tcorr} \right)
\leq C \tcorr^{-\gamma / 6}
\]
for some universal finite constant $C$. 
Taking an expectation over $T_{\tcorr}$ and recalling that $\gamma = 18$ concludes the proof of~\eqref{eq:probCc_given_D}, 
and thus also the proof of the lemma.  
\end{proof}

%%%%%%%%%%%%%%%%%%%%%%%%%%%%%%%%%%%%%%%%%%%%%%%%%%
\subsection{Proof of the variance estimate} \label{sec:estimation_variance_proof} %%%
%%%%%%%%%%%%%%%%%%%%%%%%%%%%%%%%%%%%%%%%%%%%%%%%%%

We start with two preliminary lemmas regarding the variance and covariance of functions of Beta and Dirichlet random variables, which will be useful in the proof of Lemma~\ref{lem:variance}. 

\begin{lemma}\label{lem:beta}
There exists a finite constant $C$ such that the following holds. 
Let $\alpha$ and $t$ be such that $1/2 \leq \alpha < t$ and $1/2 \leq t - \alpha$. 
Let $\psi_{1}$ and $\psi_{2}$ be i.i.d.\ $\Beta \left( \alpha, t - \alpha \right)$ random variables. Then 
\[
\Var \left( \left( \psi_{1} - \psi_{2} \right)^{2} \right) 
\leq 
\frac{C \alpha^{2} \left( t - \alpha \right)^{2}}{t^{6}}.
\]
\end{lemma}

\begin{proof}
Let $\psi \sim \Beta \left( \alpha, t - \alpha \right)$. 
Bounding the variance by the second moment we have that 
\begin{align*}
\Var \left( \left( \psi_{1} - \psi_{2} \right)^{2} \right) 
\leq 
\E \left[ \left( \psi_{1} - \psi_{2} \right)^{4} \right] 
&= \E \left[ \psi_{1}^{4} - 4 \psi_{1}^{3} \psi_{2} + 6 \psi_{1}^{2} \psi_{2}^{2} - 4 \psi_{1} \psi_{2}^{3} + \psi_{2}^{4} \right] \\
&= 2 \E \left[ \psi^{4} \right] 
- 8 \E \left[ \psi^{3} \right] \E \left[ \psi \right] 
+ 6 \E \left[ \psi^{2} \right]^{2}.
\end{align*}
For every positive integer $k$ 
we have that 
$\E \left[ \psi^{k} \right] = \prod_{i=0}^{k-1} \left( \alpha + i \right) / \left( t + i \right)$. Plugging this into the display above we obtain that 
\[
\E \left[ \left( \psi_{1} - \psi_{2} \right)^{4} \right] 
= \frac{12 \alpha \left( \alpha + 1 \right) \left( t - \alpha \right) \left( t - \alpha + 1 \right)}{t^{2} \left( t + 1 \right)^{2} \left( t + 2 \right) \left( t + 3 \right)}
\]
and the claim follows. 
\end{proof}

\begin{lemma}\label{lem:dirichlet}
There exists a finite constant $C$ such that the following holds. 
Let $\alpha_{1}$, $\alpha_{2}$, and $t$ be such that $1/2 \leq \alpha_{1}, \alpha_{2}$ and $\alpha_{1} + \alpha_{2} < t$. 
Let $\left( \psi_{1}, \phi_{1}, 1 - \psi_{1} - \phi_{1} \right)$ 
and $\left( \psi_{2}, \phi_{2}, 1 - \psi_{2} - \phi_{2} \right)$ 
be i.i.d.\ $\Dir \left( \alpha_{1}, \alpha_{2}, t - \alpha_{1} - \alpha_{2} \right)$ random vectors, 
where $\Dir$ denotes the Dirichlet distribution. 
Then 
\[
\Cov \left( \left( \psi_{1} - \psi_{2} \right)^{2}, \left( \phi_{1} - \phi_{2} \right)^{2} \right) 
\leq 
\frac{C \alpha_{1}^{2} \alpha_{2}^{2}}{t^{6}}.
\]
\end{lemma}
\begin{proof}
Let $\left( \psi, \phi, 1 - \psi - \phi \right) \sim \Dir \left( \alpha_{1}, \alpha_{2}, t - \alpha_{1} - \alpha_{2} \right)$. 
By expanding the terms in the definition of the covariance and using independence, we have that 
\begin{align*}
\Cov \left( \left( \psi_{1} - \psi_{2} \right)^{2}, \left( \phi_{1} - \phi_{2} \right)^{2} \right) 
&= 2 \E \left[ \psi^{2} \phi^{2} \right] - 2 \E \left[ \psi^{2} \right] \E \left[ \phi^{2} \right]
+ 4 \E \left[ \psi \right] \left\{ \E \left[ \psi \right] \E \left[ \phi^{2} \right] - \E \left[ \psi \phi^{2} \right] \right\} \\ 
&\quad 
+ 4 \E \left[ \phi \right] \left\{ \E \left[ \psi^{2} \right] \E \left[ \phi \right] - \E \left[ \psi^{2} \phi \right] \right\} 
+ 4 \E \left[ \psi \phi \right]^{2} - 4 \left( \E \left[ \psi \right] \E \left[ \phi \right] \right)^{2}.
\end{align*}
For nonnegative integers $\beta_{1}$ and $\beta_{2}$, 
the joint moments of $\psi$ and $\phi$ are given by 
\[
\E \left[ \psi^{\beta_{1}} \phi^{\beta_{2}} \right] 
= \frac{\prod_{i=0}^{\beta_{1}-1} \left( \alpha_{1} + i \right) \prod_{j=0}^{\beta_{2} - 1} \left( \alpha_{2} + j \right)}{\prod_{i=0}^{\beta_{1} + \beta_{2} - 1} \left( t + i \right)}.
\]
Plugging this into the display above we obtain that 
\begin{multline*}
\Cov \left( \left( \psi_{1} - \psi_{2} \right)^{2}, \left( \phi_{1} - \phi_{2} \right)^{2} \right)  \\
= 
\frac{4 \alpha_{1} \alpha_{2} \left\{ -2 t^{3} + \left( 2 \alpha_{1} \alpha_{2} + 5 \alpha_{1} + 5 \alpha_{2} - 3 \right) t^{2} + \left( - 5 \alpha_{1} \alpha_{2} + 6 \alpha_{1} + 6 \alpha_{2} \right) t - 6 \alpha_{1} \alpha_{2}  \right\}}{t^{4} \left( t + 1 \right)^{2} \left( t + 2 \right) \left( t + 3 \right)}.
\end{multline*}
To obtain an upper bound, we can drop all negative terms in the numerator. Using also the trivial bounds $t \leq t^{2}$ and $\alpha_{1}, \alpha_{2} \leq 2 \alpha_{1} \alpha_{2}$, we thus obtain that 
\[
\Cov \left( \left( \psi_{1} - \psi_{2} \right)^{2}, \left( \phi_{1} - \phi_{2} \right)^{2} \right)  
\leq 
\frac{200 \alpha_{1}^{2} \alpha_{2}^{2}}{t^{2} \left( t + 1 \right)^{2} \left( t + 2 \right) \left( t + 3 \right)}
\]
and the claim follows. 
\end{proof}

We are now ready to prove Lemma~\ref{lem:variance}. 

\begin{proof}[Proof of Lemma~\ref{lem:variance}]
We bound the variance by conditioning on the tree $T_{\tcorr}$. 
By the law of total variance we have that 
\begin{multline*}
\Var \left( S_{n} \left( k \right) \mathbf{1}_{\cC^{1} \cap \cC^{2}} \, \middle| \, \cD \right) \\
= \E \left[ \Var \left( S_{n} \left( k \right) \mathbf{1}_{\cC^{1} \cap \cC^{2}} \, \middle| \, T_{\tcorr} \right) \, \middle| \, \cD \right] 
+ \E \left[ \E \left[ S_{n} \left( k \right) \mathbf{1}_{\cC^{1} \cap \cC^{2}} \, \middle| \, T_{\tcorr} \right]^{2} \, \middle| \, \cD \right] 
- \E \left[ S_{n} \left( k \right) \mathbf{1}_{\cC^{1} \cap \cC^{2}} \, \middle| \, \cD \right]^{2}.
\end{multline*}
From the proof of Lemma~\ref{lem:first_moment} (see also the proof of Lemma~\ref{lem:coarse_first_moment}) it follows that 
\[
\limsup_{n \to \infty} \E \left[ \E \left[ S_{n} \left( k \right) \mathbf{1}_{\cC^{1} \cap \cC^{2}} \, \middle| \, T_{\tcorr} \right]^{2} \, \middle| \, \cD \right] 
\leq 
\left( \frac{1 + 3 \tcorr^{-1/3}}{\tcorr} \right)^{2}
\]
for all $\tcorr$ large enough, 
and by Lemma~\ref{lem:first_moment} we also have that 
\[
\liminf_{n \to \infty} \E \left[ S_{n} \left( k \right) \mathbf{1}_{\cC^{1} \cap \cC^{2}} \, \middle| \, \cD \right]^{2} 
\geq \left( \frac{1 - 3 \tcorr^{-1/3}}{\tcorr} \right)^{2}
\] 
for all $\tcorr$ large enough (where in both cases ``large enough'' does not depend on $k$). Putting these displays together we obtain that 
\[
\limsup_{n \to \infty} \Var \left( S_{n} \left( k \right) \mathbf{1}_{\cC^{1} \cap \cC^{2}} \, \middle| \, \cD \right)
\leq 
\limsup_{n \to \infty} \E \left[ \Var \left( S_{n} \left( k \right) \mathbf{1}_{\cC^{1} \cap \cC^{2}} \, \middle| \, T_{\tcorr} \right) \, \middle| \, \cD \right] 
+ \frac{12}{\tcorr^{7/3}}
\]
for all $\tcorr$ large enough. Since $k \leq K \leq \log \tcorr$, the latter term in the display above is at most $12/(k \tcorr^{2})$, so it remains to bound the first term. 

Interchanging the limsup and the expectation, we have that 
\begin{equation}\label{eq:exch_limsup_exp}
\limsup_{n \to \infty} \E \left[ \Var \left( S_{n} \left( k \right) \mathbf{1}_{\cC^{1} \cap \cC^{2}} \, \middle| \, T_{\tcorr} \right) \, \middle| \, \cD \right]  
\leq 
\E \left[ \limsup_{n \to \infty} \Var \left( S_{n} \left( k \right) \mathbf{1}_{\cC^{1} \cap \cC^{2}} \, \middle| \, T_{\tcorr} \right) \, \middle| \, \cD \right], 
\end{equation}
so in what follows we study the conditional variance of 
$S_{n} \left( k \right) \mathbf{1}_{\cC^{1} \cap \cC^{2}}$ 
given $T_{\tcorr}$ (with $T_{\tcorr}$ such that~$\cD$ holds). 
Expanding the variance of the sum we have that 
\begin{equation}\label{eq:var_expanded}
\Var \left( S_{n} \left( k \right) \mathbf{1}_{\cC^{1} \cap \cC^{2}} \, \middle| \, T_{\tcorr} \right) 
= 
\frac{1}{k^{2}} \sum_{\ell = 1}^{k} \sum_{m = 1}^{k} \Cov \left( Y_{n} \left( \ell \right) \mathbf{1}_{\cC^{1} \cap \cC^{2}}, Y_{n} \left( m \right) \mathbf{1}_{\cC^{1} \cap \cC^{2}} \, \middle| \, T_{\tcorr} \right).
\end{equation}
Recall from Section~\ref{sec:estimation_preliminaries} the definition of $Z_{n}^{i} \left( \ell \right)$, the limit $Z^{i} \left( \ell \right) := \lim_{n \to \infty} Z_{n}^{i} \left( \ell \right)$, 
and the distribution of the limit from~\eqref{eq:Zk}. In particular, recall that on the event $\cC^{1} \cap \cC^{2}$ we have that $X_{n}^{i} \left( \ell \right) = Z_{n}^{i} \left( \ell \right)$ for all $n \geq \tcorr$ and all $1 \leq \ell \leq K$.  
To bound the covariance in~\eqref{eq:var_expanded}, we bound from above the expectation of the product, and bound from below the individual expectations. 
First, using property~\ref{prop:nice3_fine} of Definition~\ref{def:nice_event_fine} 
we have that 
\begin{align}
\E \left[ Y_{n} \left( \ell \right) Y_{n} \left( m \right) \mathbf{1}_{\cC^{1} \cap \cC^{2}} \, \middle| \, T_{\tcorr} \right] 
&\leq 
\frac{\E \left[ \left( X_{n}^{1} \left( \ell \right) - X_{n}^{2} \left( \ell \right) \right)^{2} \left( X_{n}^{1} \left( m \right) - X_{n}^{2} \left( m \right) \right)^{2} \mathbf{1}_{\cC^{1} \cap \cC^{2}}  \, \middle| \, T_{\tcorr} \right]}{4 \left( 1 - \tcorr^{-1/3} \right)^{4} \frac{\left| T_{\tcorr} \left( \ell \right) \right|}{\tcorr}  \left( 1 - \frac{\left| T_{\tcorr} \left( \ell \right) \right|}{\tcorr} \right)  \frac{\left| T_{\tcorr} \left( m \right) \right|}{\tcorr}  \left( 1 - \frac{\left| T_{\tcorr} \left( m \right) \right|}{\tcorr} \right)} \notag \\
&\leq 
\frac{\E \left[ \left( Z_{n}^{1} \left( \ell \right) - Z_{n}^{2} \left( \ell \right) \right)^{2} \left( Z_{n}^{1} \left( m \right) - Z_{n}^{2} \left( m \right) \right)^{2} \, \middle| \, T_{\tcorr} \right]}{4 \left( 1 - \tcorr^{-1/3} \right)^{4} \frac{\left| T_{\tcorr} \left( \ell \right) \right|}{\tcorr}  \left( 1 - \frac{\left| T_{\tcorr} \left( \ell \right) \right|}{\tcorr} \right)  \frac{\left| T_{\tcorr} \left( m \right) \right|}{\tcorr}  \left( 1 - \frac{\left| T_{\tcorr} \left( m \right) \right|}{\tcorr} \right)}, \label{eq:exp_product_UB}
\end{align}
where the second inequality follows by replacing 
$X_{n}^{i} \left( \ell \right)$ 
and 
$X_{n}^{i} \left( m \right)$
with 
$Z_{n}^{i} \left( \ell \right)$ 
and 
$Z_{n}^{i} \left( m \right)$ 
on the event 
$\cC^{1} \cap \cC^{2}$, 
and then removing the indicator. 
Turning to the lower bound, from the proof of Lemma~\ref{lem:first_moment} 
we have, for any $\ell \leq K$ and any $T_{\tcorr}$ such that $\cD$ holds, that 
\[
\E \left[ Y_{n} \left( \ell \right) \mathbf{1}_{\cC^{1} \cap \cC^{2}} \, \middle| \, T_{\tcorr} \right] 
\geq 
\frac{\E \left[ \left( Z_{n}^{1} \left( \ell \right) - Z_{n}^{2} \left( \ell \right) \right)^{2} \, \middle| \, T_{\tcorr} \right]}{2 \left( 1 + \tcorr^{-1/3} \right)^{2} \frac{\left| T_{\tcorr} \left( \ell \right) \right|}{\tcorr}  \left( 1 - \frac{\left| T_{\tcorr} \left( \ell \right) \right|}{\tcorr} \right)} 
- \tcorr^{1/8} \p \left( \left( \cC^{1} \cap \cC^{2} \right)^{c} \, \middle| \, T_{\tcorr} \right).
\]
On the event $\cD$ we have that 
$\left| T_{\tcorr} \left( \ell \right) \right| \geq \tcorr^{7/8}$, 
which implies that the fraction in the display above is at most $\tcorr^{1/8}$. 
Therefore multiplying the bounds in the display above with indices $\ell$ and $m$ we obtain that 
\begin{multline}\label{eq:exp_product_LB}
\E \left[ Y_{n} \left( \ell \right) \mathbf{1}_{\cC^{1} \cap \cC^{2}} \, \middle| \, T_{\tcorr} \right] 
\E \left[ Y_{n} \left( m \right) \mathbf{1}_{\cC^{1} \cap \cC^{2}} \, \middle| \, T_{\tcorr} \right] \\
\geq 
\frac{\E \left[ \left( Z_{n}^{1} \left( \ell \right) - Z_{n}^{2} \left( \ell \right) \right)^{2} \, \middle| \, T_{\tcorr} \right] \E \left[ \left( Z_{n}^{1} \left( m \right) - Z_{n}^{2} \left( m \right) \right)^{2} \, \middle| \, T_{\tcorr} \right]}{4 \left( 1 + \tcorr^{-1/3} \right)^{4} \frac{\left| T_{\tcorr} \left( \ell \right) \right|}{\tcorr}  \left( 1 - \frac{\left| T_{\tcorr} \left( \ell \right) \right|}{\tcorr} \right)  \frac{\left| T_{\tcorr} \left( m \right) \right|}{\tcorr}  \left( 1 - \frac{\left| T_{\tcorr} \left( m \right) \right|}{\tcorr} \right)} 
- 2 \tcorr^{1/4} \p \left( \left( \cC^{1} \cap \cC^{2} \right)^{c} \, \middle| \, T_{\tcorr} \right).
\end{multline}
Putting together~\eqref{eq:exp_product_UB} and~\eqref{eq:exp_product_LB}, 
we obtain an upper bound on the covariance in~\eqref{eq:var_expanded} that consists of three terms: 
\begin{multline}\label{eq:cov_bound}
\Cov \left( Y_{n} \left( \ell \right) \mathbf{1}_{\cC^{1} \cap \cC^{2}}, Y_{n} \left( m \right) \mathbf{1}_{\cC^{1} \cap \cC^{2}} \, \middle| \, T_{\tcorr} \right) \\
\begin{aligned}
&\leq \frac{\Cov \left( \left( Z_{n}^{1} \left( \ell \right) - Z_{n}^{2} \left( \ell \right) \right)^{2},  \left( Z_{n}^{1} \left( m \right) - Z_{n}^{2} \left( m \right) \right)^{2} \, \middle| \, T_{\tcorr} \right)}{4 \left( 1 - \tcorr^{-1/3} \right)^{4} \frac{\left| T_{\tcorr} \left( \ell \right) \right|}{\tcorr}  \left( 1 - \frac{\left| T_{\tcorr} \left( \ell \right) \right|}{\tcorr} \right)  \frac{\left| T_{\tcorr} \left( m \right) \right|}{\tcorr}  \left( 1 - \frac{\left| T_{\tcorr} \left( m \right) \right|}{\tcorr} \right)} \\
&\quad + \left\{ \left( 1 - \tcorr^{-1/3} \right)^{-4} - \left( 1 + \tcorr^{-1/3} \right)^{-4} \right\} \frac{\E \left[ \left( Z_{n}^{1} \left( \ell \right) - Z_{n}^{2} \left( \ell \right) \right)^{2} \, \middle| \, T_{\tcorr} \right] \E \left[ \left( Z_{n}^{1} \left( m \right) - Z_{n}^{2} \left( m \right) \right)^{2} \, \middle| \, T_{\tcorr} \right]}{4 \frac{\left| T_{\tcorr} \left( \ell \right) \right|}{\tcorr}  \left( 1 - \frac{\left| T_{\tcorr} \left( \ell \right) \right|}{\tcorr} \right)  \frac{\left| T_{\tcorr} \left( m \right) \right|}{\tcorr}  \left( 1 - \frac{\left| T_{\tcorr} \left( m \right) \right|}{\tcorr} \right)}  \\
&\quad + 2 \tcorr^{1/4} \p \left( \left( \cC^{1} \cap \cC^{2} \right)^{c} \, \middle| \, T_{\tcorr} \right).
\end{aligned}
\end{multline} 

We now deal with each term in turn, starting with the last one. 
Since this term does not depend on the indices $\ell$ and $m$, nor on $n$, averaging over $\ell$ and $m$, and taking the limit as $n \to \infty$, this term remains 
$2 \tcorr^{1/4} \p \left( \left( \cC^{1} \cap \cC^{2} \right)^{c} \, \middle| \, T_{\tcorr} \right)$. 
Taking an expectation over $T_{\tcorr}$ (see~\eqref{eq:exch_limsup_exp}), 
this becomes 
$2 \tcorr^{1/4} \p \left( \left( \cC^{1} \cap \cC^{2} \right)^{c} \, \middle| \, \cD \right)$, 
which by Lemma~\ref{lem:D} is at most $C/\tcorr^{11/4}$ for some finite constant $C$.  

Turning to the second term in~\eqref{eq:cov_bound}, first note that 
\[
\left( 1 - \tcorr^{-1/3} \right)^{-4} - \left( 1 + \tcorr^{-1/3} \right)^{-4} 
\leq 9 \tcorr^{-1/3}
\]
for all $\tcorr$ large enough. In the proof of Lemma~\ref{lem:first_moment} we showed that 
\[
\lim_{n \to \infty} 
\frac{\E \left[ \left( Z_{n}^{1} \left( \ell \right) - Z_{n}^{2} \left( \ell \right) \right)^{2} \, \middle| \, T_{\tcorr} \right]}{\frac{\left| T_{\tcorr} \left( \ell \right) \right|}{\tcorr}  \left( 1 - \frac{\left| T_{\tcorr} \left( \ell \right) \right|}{\tcorr} \right)} 
\leq \frac{C}{\tcorr}
\]
for all $\ell \leq K$ and some universal finite constant $C$. 
Putting these bounds together, we obtain that, after taking a limit as $n \to \infty$ (which exists), the second term in~\eqref{eq:cov_bound} is at most $C / \tcorr^{7/3}$ for some universal finite constant $C$. This holds for all indices $\ell$ and $m$, and for all trees $T_{\tcorr}$. Thus after averaging over all these we still have a bound of $C / \tcorr^{7/3}$. 

Finally, we turn to the first term in~\eqref{eq:cov_bound}, which is the main term among the three. By the bounded convergence theorem the limit as $n \to \infty$ of this term exists and is equal to 
\[
\frac{\Cov \left( \left( Z^{1} \left( \ell \right) - Z^{2} \left( \ell \right) \right)^{2},  \left( Z^{1} \left( m \right) - Z^{2} \left( m \right) \right)^{2} \, \middle| \, T_{\tcorr} \right)}{4 \left( 1 - \tcorr^{-1/3} \right)^{4} \frac{\left| T_{\tcorr} \left( \ell \right) \right|}{\tcorr}  \left( 1 - \frac{\left| T_{\tcorr} \left( \ell \right) \right|}{\tcorr} \right)  \frac{\left| T_{\tcorr} \left( m \right) \right|}{\tcorr}  \left( 1 - \frac{\left| T_{\tcorr} \left( m \right) \right|}{\tcorr} \right)}.
\]
To obtain a slightly simpler expression, 
recall that 
$\left| T_{\tcorr} \left( \ell \right) \right| \leq \tcorr / 2$ 
for all $\ell \in \left\{ 1, \ldots, K \right\}$, 
and hence the display above is bounded from above by 
\begin{equation}\label{eq:main_cov_term_to_bound}
\frac{C \tcorr^{2} \Cov \left( \left( Z^{1} \left( \ell \right) - Z^{2} \left( \ell \right) \right)^{2},  \left( Z^{1} \left( m \right) - Z^{2} \left( m \right) \right)^{2} \, \middle| \, T_{\tcorr} \right)}{\left| T_{\tcorr} \left( \ell \right) \right| \left| T_{\tcorr} \left( m \right) \right|} 
\end{equation} 
for some universal finite constant $C$. 
We now distinguish two cases based on whether or not the indices $\ell$ and $m$ are equal. 

First, when $\ell = m$, we have from~\eqref{eq:Zk} and Lemma~\ref{lem:beta} that 
\[
\Var \left( \left( Z^{1} \left( \ell \right) - Z^{2} \left( \ell \right) \right)^{2} \, \middle| \, T_{\tcorr} \right) 
\leq \frac{C \left| T_{\tcorr} \left( \ell \right) \right|^{2} \left( \tcorr - \left| T_{\tcorr} \left( \ell \right) \right| \right)^{2}}{\tcorr^{6}} 
\leq \frac{C \left| T_{\tcorr} \left( \ell \right) \right|^{2}}{\tcorr^{4}} 
\]
for some universal finite constant $C$. Thus the expression in~\eqref{eq:main_cov_term_to_bound} is bounded from above by 
$C' / \tcorr^{2}$ for some universal finite constant $C'$. 
There are $k$ terms in~\eqref{eq:var_expanded} where the indices are equal; furthermore, there is a $1/k^{2}$ factor in front of the sum. 
Putting all this together we see that the contribution from these terms is at most 
$C' / \left( k \tcorr^{2} \right)$, which is the bound in the claim. 

We turn now to the case when $\ell \neq m$. 
By P\'olya urn arguments (see, e.g.,~\cite[Section~4.5]{RB17}) it follows that 
$\left( Z^{1} \left( \ell \right), Z^{1} \left( m \right), 1 - Z^{1} \left( \ell \right) - Z^{1} \left( m \right) \right)$ 
and 
$\left( Z^{2} \left( \ell \right), Z^{2} \left( m \right), 1 - Z^{2} \left( \ell \right) - Z^{2} \left( m \right) \right)$ 
are i.i.d.\ (conditionally given $T_{\tcorr}$) Dirichlet random vectors, 
with parameters given as follows: 
\begin{multline*}
\left( Z \left( \ell \right), Z \left( m \right), 1 - Z \left( \ell \right) - Z \left( m \right) \right) \\
\sim 
\begin{cases} 
\Dir \left( \left| T_{\tcorr} \left( \ell \right) \right|, \left| T_{\tcorr} \left( m \right) \right|, \tcorr - \left| T_{\tcorr} \left( \ell \right) \right| - \left| T_{\tcorr} \left( m \right) \right|  \right) &\text{ for } \UA, \\
\Dir \left( \left| T_{\tcorr} \left( \ell \right) \right| - \frac{1}{2}, \left| T_{\tcorr} \left( m \right) \right| - \frac{1}{2}, \tcorr - \left| T_{\tcorr} \left( \ell \right) \right| - \left| T_{\tcorr} \left( m \right) \right|  \right)  &\text{ for } \PA. 
\end{cases}
\end{multline*}
By Lemma~\ref{lem:dirichlet} we thus have for $\ell \neq m$ that 
\[
 \Cov \left( \left( Z^{1} \left( \ell \right) - Z^{2} \left( \ell \right) \right)^{2},  \left( Z^{1} \left( m \right) - Z^{2} \left( m \right) \right)^{2} \, \middle| \, T_{\tcorr} \right) 
 \leq 
 \frac{C \left| T_{\tcorr} \left( \ell \right) \right|^{2} \left| T_{\tcorr} \left( m \right) \right|^{2}}{\tcorr^{6}}
\]
for some universal finite constant $C$. Thus the expression in~\eqref{eq:main_cov_term_to_bound} is bounded from above by 
$C' \left| T_{\tcorr} \left( \ell \right) \right| \left| T_{\tcorr} \left( m \right) \right| / \tcorr^{4}$ 
for some universal finite constant $C'$. 
Plugging this back into~\eqref{eq:var_expanded} we see that 
the contribution to this expression from terms where $\ell \neq m$ is at most 
\[
\frac{C'}{k^{2} \tcorr^{4}} 
\sum_{\ell = 1}^{k} \sum_{m = 1}^{k} \left| T_{\tcorr} \left( \ell \right) \right| \left| T_{\tcorr} \left( m \right) \right| 
= 
\frac{C'}{k^{2} \tcorr^{4}} 
\left( \sum_{\ell = 1}^{k} \left| T_{\tcorr} \left( \ell \right) \right| \right)^{2} 
\leq 
\frac{C'}{k^{2} \tcorr^{4}}  \tcorr^{2} 
= \frac{C'}{k^{2} \tcorr^{2}},
\]
which concludes the claim. 
\end{proof}

%%%%%%%%%%%%%%%%%%%%%%%%
%%% Acknowledgements %%%
%%%%%%%%%%%%%%%%%%%%%%%%

% \section*{Acknowledgements}

%%%%%%%%%%%%%%%%%%
%%% References %%%
%%%%%%%%%%%%%%%%%%

\bibliographystyle{abbrv}
\bibliography{bib}

%%%%%%%%%%%%%%%%
%%% Appendix %%%
%%%%%%%%%%%%%%%%

% \newpage

% \appendix

\end{document}